\documentclass[11pt]{amsart}
\usepackage{multirow,bigdelim}
\usepackage{amsfonts}
\usepackage{dsfont}
\usepackage{amsmath}
\usepackage{cases}
\usepackage{mathrsfs}
\usepackage{hyperref}
\usepackage{multimedia}
\usepackage{graphicx}
\usepackage{indentfirst}
\usepackage{bm}
\usepackage{amsthm}
\usepackage{mathrsfs}
\usepackage{latexsym}
\usepackage{amsmath}
\usepackage{amssymb}
\usepackage{microtype}
\usepackage{relsize}
\usepackage[all]{xy}

\usepackage{hyperref}

\DeclareSymbolFont{SY}{U}{psy}{m}{n}
\DeclareMathSymbol{\emptyset}{\mathord}{SY}{'306}

\theoremstyle{plain}

\newtheorem{thm}{Theorem}[section]
\newtheorem{corollary}[thm]{Corollary}
\newtheorem{lemma}[thm]{Lemma}
\newtheorem{proposition}[thm]{Proposition}
\newtheorem{definition}[thm]{Definition}

\theoremstyle{definition}
\newtheorem{remark}[thm]{Remark}
\newtheorem{example}[thm]{Example}

\numberwithin{equation}{section}

\begin{document}
\title[On the flag structure and classification of the holomorphic curves on C*-algebras]{On the flag structure and classification of the holomorphic curves on C*-algebras}

\author{Zhimeng Chen} \author{Jing Xu}
\email{chenzhimeng2020@163.com, xujingmath@outlook.com}
\address{School of Mathematical Sciences, Hebei Normal University, Shijiazhuang, Hebei 050016, China}
\thanks{This work was supported by the National
Natural Science Foundation of China, Grant
No. 11922108 and 12001159. }



\begin{abstract}
In this note, we will define the formulas of curvature and it's covariant derivatives for holomorphic curves on C*-algebras for the multivariable case. As applications, the unitarily and similarly classification theorems for holomorphic bundle and commuting operator tuples in Cowen-Douglas class are given.
\end{abstract}

\maketitle

\section{Introduction}

Given an infinite dimensional and complex separable Hilbert space $\mathcal{H}$, let $\mathcal{L}(\mathcal{H})$ be the algebra of bounded linear operators on $\mathcal{H}$, and
$\mathcal{G}r(n,\mathcal{H})$ be the Grassmann manifold, the set of all $n$-dimensional subspaces of $\mathcal{H}$.
Given a bounded domain $\Omega$ of $m$-dimensional complex space $\mathbb{C}^{m}$,
a map $f:\Omega \rightarrow \mathcal{G}r(n,\mathcal{H})$ is called as a holomorphic carve,
if there is $n$ holomorphic $\mathcal{H}$-valued functions $\gamma_{1},\cdots,\gamma_{n}$ on $\Omega$ such that $f(\lambda)=\bigvee\{\gamma_{1}(\lambda),\cdots,\gamma_{n}(\lambda)\}$ for all $\lambda\in\Omega,$ where $\bigvee$ denotes the closure of linear span.
Then a natural $n$-dimensional Hermitian holomorphic vector bundle $E_{f}$ is induced over $\Omega,$ that is,
$$E_{f}=\{(x,\lambda)\in\mathcal{H}\times\Omega : x\in f(\lambda)\}\quad \textup{and}\,\,\pi:E_{f}\rightarrow\Omega,\quad \textup{where}\,\,\pi(x,\lambda)=\lambda.$$

M. Martin \cite{Flag1987} shows that a smooth function $f:\Omega\rightarrow\mathcal{G}r(n,\mathcal{H})$ is analytic if and only if its corresponding self-adjoint projection $[f]:\Omega\rightarrow\mathcal{L}(\mathcal{H})$ satisfies $(1-[f])\overline{\partial}_{i}[f]=0$ for $1\leq i\leq m.$
In Cowen-Douglas theory, the holomorphic curve in Grassmann manifold is a fundamental and important concept.
Two holomorphic curves $f,\widetilde{f} :\Omega\rightarrow\mathcal{G}r(n,\mathcal{H})$ are congruent, if there is a unitary operator $U$ in $\mathcal{L}(\mathcal{H})$ such that $Uf=\widetilde{f}$. M.J. Cowen and R.G. Douglas \cite{CD, CD3} show that the congruence of $f$ and $\widetilde{f}$ if and only if $E_{f}$ and $E_{\widetilde{f}}$ are locally equivalent as Hermitian holomorphic vector bundles. This result, known as the rigidity theorem, gives a link between the congruence problem for maps and the equivalence problem for bundles.
Using some results and techniques in the theory of involutive algebras and  certain operator theoretic techniques developed in \cite{AM}, M. Martin \cite{Flag1987,M1985} also gives the congruence theorem and equivalence theorem in general cases.

M. Martin and N. Salinas \cite{Flag} first proposed the holomorphic curve on $C^{*}$-algebras. Let $\mathcal{U}$ be a unital $C^{*}$-algebra, if an $n$-tuple $e=(e_{1},\cdots,e_{n})$ of idempotents in $\mathcal{U}$ and satisfied that $e_{1}\leq\cdots\leq e_{n},$ then $e$ is called an extended $n$-flag of $\mathcal{U}.$ If an $n$-tuple $P=(P_{1},\cdots,P_{n})$ is extended $n$-flag of $\mathcal{U}$ and every entry $P_{i},\,1\leq i\leq n,$ is a projection in $\mathcal{U}$, then $P$ is called $n$-flag in $\mathcal{U}$.
Let $\mathcal{I}_{n}(\mathcal{U})$ and $P_{n}(\mathcal{U})$ be the set of all extended $n$-flags of $\mathcal{U}$ and of all $n$-flags of $\mathcal{U}$, respectively.
Moreover, for $n=1,$ the spaces $\mathcal{I}(\mathcal{U})$ and $P(\mathcal{U})$ are alternatively called the extended Grassmann manifold and the Grassmann manifold of $\mathcal{U},$ respectively.
M. Martin \cite{Flag2022} shows that a smooth projection valued mapping $P:\Omega\rightarrow P(\mathcal{U})$ for $\mathcal{U}=\mathcal{L}(\mathcal{H})$ is holomorphic if and only if
it satisfies the following equivalent Cauchy-Riemann equations,
\begin{align*}
&\partial P(\lambda)\cdot P(\lambda)=\partial P(\lambda),\quad
P(\lambda)\cdot\partial P(\lambda)=0,\\
&P(\lambda)\cdot\overline{\partial} P(\lambda)=\overline{\partial} P(\lambda),\quad
\overline{\partial} P(\lambda)\cdot P(\lambda)=0,\quad\lambda\in\Omega\subseteq\mathbb{C}^{m},
\end{align*}
where $\partial=\sum\limits_{i=1}^{m}\frac{\partial}{\partial\lambda_{i}}$ and $\overline{\partial}=\sum\limits_{j=1}^{m}\frac{\partial}{\partial\overline{\lambda}_{j}}$.
For $\Omega\subseteq\mathbb{C}$, M. Martin and N. Salinas \cite{Flag,Flag1987,S1988} show that a mapping $P: \Omega \rightarrow P(\mathcal{U})$ is holomorphic if and only if $\overline{\partial}P(\lambda)P(\lambda)=0$ for all $\lambda\in\Omega$.
M. Martin and N. Salinas  \cite{Flag1998} studied the self-adjoint idempotent from the point of view of differential geometry, while giving a smooth mapping $P=(P_{1},\cdots,P_{n})$ from a complex manifold into the $n$-flag manifold $\mathcal{P}_{n}(A)$ is holomorphic if and only if $(1-P_{j})\overline{\partial}P_{j}=0,\,1\leq j\leq n$, where $A$ is a Banach $*$-algebra,
which generalizes the results \cite{Flag1987,S1988}. In addition, it is proved that the Gram-Schmidt map from extended $n$-flag to $n$-flag is holomorphic map in $C^{*}$-algebras.

In \cite{M1981},\cite{M1985}, M. Martin makes a connection between the various properties of holomorphic maps from connected complex manifold $M$ into the Grassmann manifold $\mathcal{G}r(n,\mathcal{H})$ and the theory of Hermintian holomorphic vector bundles on $M$ of rank $n$ that is real-analytic metrics. M. Martin discusses the pointwise equivalence of hermitian smooth vector bundles equipped with compatible linear connections if and only if they are locally equivalent which are restricted to open densn subsets if and only if Hermitian holomorphic vector bundles are locally equivalent. Generalizes the conclusion in \cite{CD}, gives a similar result of the Rigidity theorem in the general case that $M$, which is an arbitrarily dimensional complex manifold, and proves the Congruence theorem of holomorphic maps by using Hermitian holomorphic bundles and the Equivalence theorem. In fact, the Congruence theorem can be regarded as a corollary of the Equivalence theorem.

For two holomorphic curves $P$ and $Q$ are congruent if and only if there exists a unitary operator $U\in\mathcal{L}(\mathcal{H})$ such that $UP(\lambda)=Q(\lambda)U,\,\lambda\in\Omega$ (denoted by $P\sim_{u}Q$). In \cite{Flag}, M. Martin and N. Salinas reduce the unitary equivalence of Cowen-Douglas class to the congruence equivalence of projection-valued holomorphic maps and give the unitarily invariants of holomorphic curve $P$ by considering the partial derivatives $\partial^{I}P\overline{\partial}^{J}P,\,I,J\in\mathbb{N}$.
In \cite{Flag2022}, M. Martin gives the completes sets of invariants for untarily equivalent the tuples of Cowen-Douglas operators, and for congruent holomorphic mappings by studies $D^{I}P\overline{D}^{J}P,I,J\in\mathbb{Z}_{+}^{m}.$
For two holomorphic curves $P$ and $Q$ are similarity equivalence if and only if there exists an invertible operator $X\in\mathcal{L}(\mathcal{H})$ such that $XP(\lambda)=Q(\lambda)X,\,\lambda\in\Omega$ (denoted by $P\sim_{s}Q$). In \cite{Jiang2004,Jiang2005,Jiang2007,Jiang2011}, L.C. Jiang and the other authors characterize the decomposition uniquensess of operators and the similarity classification of Cowen-Douglas operators by $K_{0}$-group of operator transposition algebra, and generalized the result to holomorphic curve.
In \cite{CAOT}, K. Ji gives the definition of curvature and it's covariant derivative of holomorphic curve on $C^{*}$-algebras in the case of single variable, and proves the relation between it and the classical curvature and it's covariant derivative of holomorphic curve in the case of single variable.
In \cite{JS}, K. Ji and J. Sarkar study curvatures, a projection formula and a trace-curvature formula in terms of Hilbert-Schmidt norm by geometrizing method. Furthermore, the necessary and sufficient conditions for a pure $\mathcal{M}$-contractive Hilbert module $\mathcal{H}\in\mathcal{B}_{m}(\mathbb{D})$ to be similar to $\mathcal{M}\otimes\mathbb{C}^{m}$ are given.

\par The paper is organized as follows. In Sect.\ref{section2}, some notations and well known results will be introduced. We will define the curvature and it's covariant derivatives of extended holomorphic curves in the multivariate on $C^{*}$-algebras and we explore the relationship between it and classical curves and it's covariant derivatives of extended holomorphic curves. We prove that unitarily classification and similarity classification theorems of extended holomorphic curves and extended holomorphic curves on $C^{*}$-algebras in the multivariate. In Sect.\ref{section3}, we give that the representation of the orthogonal projection corresponding to operators $T\in\mathcal{FB}_{2}(\Omega).$ We will characterize the unitary equivalence of two operators $T,S\in\mathcal{FB}_{2}(\Omega)$ by holomorphic curves.

\section{\sf The curvature of extended holomorphic curves on $C^{*}$-algebra}\label{section2}
The first part of this paper is devoted to introducing a suitable setting for the
study of extended holomorphic curves on $C^{*}$-algebras in the multivariate case. We shall be
primarily concerned with the relation between the curvature and it's covariant derivatives (connection) of extended holomorphic curves on $C^{*}$-algebras and the classical definitions in the multivariate case. In the following subsection, we recall a few general definitions and some basic properties.

\subsection{\sf Extended holomorphic curves}\label{sec2.1}
\
\newline
\indent Let $\Omega$ be a domain (open and connected set) in $\mathbb{C}^{m}$, and let $\mathbb{Z}_{+}^{m}$ be the set of $m$-tuples of nonnegative integers. If $I=(i_{1},\cdots,i_{m})\in\mathbb{Z}_{+}^{m}$, we
set, as usual, $\vert I\vert=i_{1}+\cdots+i_{m}$, $I!=i_{1}!\cdots i_{m}!$,
$D^{I}=\partial_{1}^{i_{1}}\cdots \partial_{m}^{i_{m}}$,
$\overline{D}^{I}=\overline{\partial}_{1}^{i_{1}}\cdots\overline{\partial}_{m}^{i_{m}}$,
 $\partial^{I}=\partial_{1}^{i_{1}}+\cdots+\partial_{m}^{i_{m}}$,
and $\overline{\partial}^{I}=\overline{\partial}_{1}^{i_{1}}+\cdots+\overline{\partial}_{m}^{i_{m}}$, where $\partial_{i}=\frac{\partial}{\partial\lambda_{i}}$ and $\overline{\partial}_{i}=\frac{\partial}{\partial\overline{\lambda}_{i}}$, $1\leq i\leq m$.
In particular, $\partial=\partial_{1}+\cdots+\partial_{m}$ and $\overline{\partial}=\overline{\partial}_{1}+\cdots+\overline{\partial}_{m}$.
Without causing confusion, denote $0=(0,\cdots,0)$ and $e_{i}=(0,\cdots,0,1,0,\cdots,0)\in\mathbb{Z}_{+}^{m}$ with 1 on the $i$th position.

\begin{definition}\label{def2.1}
Let $\Omega\subseteq\mathbb{C}^{m}$ be a domain, and let $\mathcal{U}$ be a unital $C^{*}$-algebra, we say a real-analytic map $\mathcal{I}:\Omega\rightarrow \mathcal{I}(\mathcal{U})$ as an extended holomorphic curve if the following statements hold:
\begin{align*}
&\partial\mathcal{I}(\lambda)\cdot\mathcal{I}(\lambda)=\partial\mathcal{I}(\lambda),\quad
\mathcal{I}(\lambda)\cdot\partial\mathcal{I}(\lambda)=0,\\
&\mathcal{I}(\lambda)\cdot\overline{\partial}\mathcal{I}(\lambda)=\overline{\partial}\mathcal{I}(\lambda),\quad
\overline{\partial}\mathcal{I}(\lambda)\cdot\mathcal{I}(\lambda)=0,
\quad\lambda\in\Omega.
\end{align*}
Obviously, the above formula is equivalent to
\begin{align*}
&\partial_{i}\mathcal{I}(\lambda)\cdot\mathcal{I}(\lambda)=\partial_{i}\mathcal{I}(\lambda),\quad
\mathcal{I}(\lambda)\cdot \partial_{i}\mathcal{I}(\lambda)=0,\\
&\mathcal{I}(\lambda)\cdot\overline{\partial}_{j}\mathcal{I}(\lambda)=\overline{\partial}_{j}\mathcal{I}(\lambda),\quad
\overline{\partial}_{j}\mathcal{I}(\lambda)\cdot\mathcal{I}(\lambda)=0,
\quad 1\leq i,j\leq m, \lambda\in\Omega.
\end{align*}
\end{definition}
Let $B$ be a unital $C^{*}$-algebra. The Hilbert $B$-module $l^{2}(\mathbb{N},B)$ is defined as
$$l^{2}(\mathbb{N},B):=\Big\{(a_{i})_{i\in\mathbb{N}}:a_{i}\in B\textup{ and } \sum\limits_{i\in\mathbb{N}}\Vert a_{i}\Vert^{2}\leq \infty\Big\}.$$
Let $\mathcal{U}:=\mathscr{L}(l^{2}(\mathbb{N},B))$ be the algebra of all bounded linear operators on $l^{2}(\mathbb{N},B)$
and $F,G:\Omega\rightarrow \mathscr{L}(\mathcal{H},l^{2}(\mathbb{N},B))$ are operator valued functions that satisfy
$$F(\lambda)\xi=\sum_{i=0}^{\infty}\langle\xi,\widehat{e}_{i}\rangle F_{i}(\lambda)\quad \text{and}\quad G(\lambda)\xi=\sum_{i=0}^{\infty}\langle\xi,\widehat{e}_{i}\rangle G_{i}(\lambda),\quad \xi\in\mathcal{H}, \lambda\in\Omega,$$
where $F_{i},G_{i}:\Omega \rightarrow l^{2}(\mathbb{N},B),i\geq0$, are holomorphic and $\{\widehat{e}_{i}\}_{i=0}^{\infty}$ is a standard orthogonal basis of $\mathcal{H}.$
Assuming that the mapping $\mathcal{I}:\Omega\rightarrow\mathcal{I}(\mathcal{U})$ satisfies $\mathcal{I}(\lambda)=F(\lambda)(G^{*}(\lambda)F(\lambda))^{-1}G^{*}(\lambda)$ for $F(\lambda)$ and $G(\lambda)$ that satisfy the above, it can be seen from Definition \ref{def2.1} that the mapping $\mathcal{I}:\Omega\rightarrow \mathcal{I}(\mathcal{U})$ is an extended holomorphic curve and $\text{Ran}\,F(\lambda)=\text{Ran}\,\mathcal{I}(\lambda).$ We denote the set composed of all such mappings as $\mathcal{I}_{n}(\Omega,\mathcal{U})$.

\begin{definition}\label{def2.2}
Let $\mathcal{I}\in\mathcal{I}_{n}(\Omega,\mathcal{U})$, $\mathcal{I}(\lambda)=F(\lambda)(G^{*}(\lambda)F(\lambda))^{-1}G^{*}(\lambda)$, and the metric $$H(\lambda)=\left(\langle F_{j}(\lambda),G_{i}(\lambda)\rangle\right)_{n\times n}=G^{*}(\lambda)F(\lambda).$$
The classical connection $\Theta(\lambda)$ and classical curvature $\mathcal{K}_{\mathcal{I}}$ of $\mathcal{I}$ are defined as
\begin{align}\label{e0}
 \Theta(\lambda)=H^{-1}(\lambda)\sum_{i=1}^{m}\left(\frac{\partial}{\partial{\lambda}_{i}}H(\lambda)\right)\quad \text{and}\quad \mathcal{K}_{\mathcal{I}}(\lambda)=-\sum_{i,j=1}^{m}\frac{\partial}{\partial\overline{\lambda}_{j}}\left(H^{-1}(\lambda)\frac{\partial}{\partial{\lambda}_{i}}H(\lambda)\right).
\end{align}
And the covariant derivatives of the curvature $\mathcal{K}_{\mathcal{I}}$ are defined as the following:
\begin{itemize}
\item[(1)]$\mathcal{K}_{\mathcal{I},\overline{\lambda}_{j}}
=\frac{\partial}{\partial\overline{\lambda}_{j}}(\mathcal{K}_{\mathcal{I}})$,
\item[(2)]$\mathcal{K}_{\mathcal{I},\lambda_{i}}
=\frac{\partial}{\partial{\lambda}_{i}}(\mathcal{K}_{\mathcal{I}})+[H^{-1}\frac{\partial}{\partial{\lambda}_{i}}H,\mathcal{K}_{\mathcal{I}}],\quad\lambda\in\Omega.$
\end{itemize}
\end{definition}

In particular, for any $m$-tuple of commuting operators $\mathbf{T}=(T_{1},\cdots,T_{m}) \in \mathbf{\mathcal{B}}_{n}^{m}(\Omega)\bigcap \mathcal{L}(\mathcal{H})^m,$
assuming that $\{F_{i}\}_{i=1}^{n}$ and $\{G_{i}\}_{i=1}^{n}$ are the frames of the Hermitian holomorphic vector bundle $E_{\textbf{T}}$ and auxiliary Hermitian holomorphic vector bundle $E$, respectively.
Setting $F=(F_{1},\cdots,F_{n})$ and $G=(G_{1},\cdots,G_{n})$,
define the mapping $\mathcal{I}_{\textbf{T}}:\Omega\rightarrow\mathcal{I}(\mathcal{L}(\mathcal{H}))$ as follows
$$\mathcal{I}_{\textbf{T}}(\lambda):=F(\lambda)(G^{*}(\lambda)F(\lambda))^{-1}G^{*}(\lambda).$$
Then $\mathcal{I}_{\textbf{T}}:\Omega\rightarrow\mathcal{I}(\mathcal{L}(\mathcal{H}))$ is an extended holomorphic curve and $\mathcal{I}_{\textbf{T}}(\lambda)$ is the matrix form of idempotent from $\mathcal{H}$ to $\ker \mathscr{D}_{\textbf{T}-\lambda}$ as an operator.
If $m\geq2$ is a positive integer, then $$-((\mathcal{K}_{\mathcal{I}_{\textbf{T}}})_{\lambda_{l}})_{\overline{\lambda}_{k}}+
((\mathcal{K}_{\mathcal{I}_{\textbf{T}}})_{\overline{\lambda}_{k}})_{\lambda_{l}}
=\overline{\partial}_{k}(H^{-1}\partial_{l}H)\overline{\partial}(h^{-1}\partial H)-\overline{\partial}(H^{-1}\partial H)\overline{\partial}_{k}(H^{-1}\partial_{l}H),\quad 1\leq l, k\leq m,$$ where $\lambda=(\lambda_{1},\ldots,\lambda_{m})\in\Omega.$ From the properties of the matrix, in general, the covariant derivative $\mathcal{K}_{\mathcal{I}_{\textbf{T}},\lambda^{I},\overline{\lambda}^{J}},I,J\in\mathbb{Z}_{+}^{m}\setminus\{0\},$ of classical curvature depends on the order of $I$ and $J$.

\subsection{\sf The Cowen-Douglas class}

Let $\textbf{T}=(T_{1},\cdots,T_{m}),$ $\textbf{S}=(S_{1},\cdots,S_{m})\in\mathcal{L}(\mathcal{H})^{m},$ which is all $m$-tuples of commuting operators on $\mathcal{H}.$ $\textbf{T}$ and $\textbf{S}$ are unitarily equivalent (denotes $\textbf{T}\sim_{u}\textbf{S}$), if there is a unitary operator $U$ such that $UT_{i}=S_{i}U,\,1\leq i\leq m.$ $\textbf{T}$ and $\textbf{S}$ are similar (denotes $\textbf{T}\sim_{s}\textbf{S}$), if there is a invertible operator $X$ such that $XT_{i}=S_{i}X,\,1\leq i\leq m.$
\par Given $m$-tuples $\textbf{T}=(T_{1},\cdots T_{m})\in\mathcal{L}(\mathcal{H})^{m},$ define an operator $\mathscr{D}_{\textbf{T}}:\mathcal{H}\rightarrow\mathcal{H}\oplus\cdots\oplus\mathcal{H}$ by
$$\mathscr{D}_{\textbf{T}}h=(T_{1}h,\cdots,T_{m}h),$$
for any $h\in\mathcal{H}$. Set $\textbf{T}-\lambda:=(T_{1}-\lambda_{1},\cdots,T_{m}-\lambda_{m}),$ then $\mathscr{D}_{\textbf{T}-\lambda}= \bigcap\limits_{i=1}^{m}\ker(T_{i}-\lambda_{i}),\,\lambda=(\lambda_{1},\cdots,\lambda_{m})\in\Omega$ is joint kernel.
In , M.J. Cowen and R.G. Douglas introduced the class $\mathcal{B}_{n}^{m}(\Omega),$ let $\textbf{T}=(T_{1},\cdots T_{m})\in\mathcal{L}(\mathcal{H})^{m},$ we say that $\textbf{T}\in\mathcal{B}_{n}^{m}(\Omega)$ if satisfying the following conditions:
\begin{equation*}
 \begin{aligned}
\mathcal{B}_{n}^{m}(\Omega)=\{\textbf{T}\in\mathcal{L}(\mathcal{H})^{m}:\quad
&(1)\,\text{ran}\mathscr{D}_{\textbf{T}-\lambda}\textit{ is closed for any }\lambda\in\Omega;\\
&(2)\, \bigvee_{\lambda\in\Omega}\ker\mathscr{D}_{\textbf{T}-\lambda}=\mathcal{H};\textit{and}\\
&(3)\,\dim\ker\mathscr{D}_{\textbf{T}-\lambda}=n\textit{ for any }\lambda\in\Omega.\,\}
 \end{aligned}
\end{equation*}
where $\bigvee$ denotes the closure of linear span. For $m$-tuple $\textbf{T}\in\mathcal{B}_{n}^{m}(\Omega),$ there is a Hermitian holomorphic vector bundle $E_{\textbf{T}}$ of rank $n,$ where
$$E_{\textbf{T}}=\{(\lambda,x)\in\Omega\times\mathcal{H}:x\in\ker{(\textbf{T}-\lambda)}\},\quad\pi(\lambda,x)=\lambda.$$
For any $m$-tuples $\textbf{T}\in\mathcal{B}_{n}^{m}(\Omega)$, we can find a holomorphic frame $\{e_{i}\}_{i=1}^{n}$ of $E_{\textbf{T}}$ such that $\textbf{T}_{i}e_{i}(\lambda)=\lambda_{i}e_{i}(\lambda),\lambda\in\Omega.$ A map $f:\Omega\rightarrow\mathcal{G}r(n,\mathcal{H})$ with $n$ dimension defined by
$$f(\lambda):=\ker(\textbf{T}-\lambda)=\bigvee\{e_{i}(\lambda),i=1,2,\cdots,n\},\quad \lambda\in\Omega,$$
then map $f$ is said to be a holomorphic curve.
Furthermore, M.J. Cowen and R.G. Douglas proved that
for any two $m$-tuples of operators $\textbf{T},\,\textbf{S}\in\mathcal{B}_{n}^{m}(\Omega),$ $\textbf{T}$ and $\textbf{S}$ are unitarily equvilent if and only if the complex bundles $E_{\textbf{T}}$ and $E_{\textbf{S}}$ are equivalent as Hermitian holomorphic vector bundles.
Let $\gamma=\{\gamma_{1},\cdots,\gamma_{n}\}$ be a holomorphic frame of $E_{\textbf{T}}$ on $\Omega$ and form the Gram matrix $h(\lambda)=(\langle\gamma_{j}(\lambda),\gamma_{i}(\lambda)\rangle)_{i,j=1}^{n}$. Then the curvature $\mathcal{K}_{\textbf{T}}$ of the $E_{\textbf{T}}$ is $\mathcal{K}_{\textbf{T}}(\lambda)=-\sum\limits_{j=1}^{m}\frac{\partial}{\partial \overline{\lambda}_{j}}\left(h^{-1}(\lambda)\sum\limits_{i=1}^{m}\frac{\partial}{\partial \lambda_{i}}h(\lambda)\right).$
In particular, for $\textbf{T},\textbf{S}\in\mathcal{B}_{1}^{1}(\Omega),$ we have that $\mathcal{K}_{\textbf{T}}(\lambda)=-\frac{\partial^{2}\log\Vert\gamma(\lambda)\Vert^{2}}{\partial \lambda\partial \overline{\lambda}},$ then $\textbf{T}$ and $\textbf{S}$ are unitarily equivalent if and only if $\mathcal{K}_{\textbf{T}}=\mathcal{K}_{\textbf{S}}.$


\subsection{\sf The curvature and connection}
\indent In the following, we define the connection, curvature and it's covariant derivatives of extended holomorphic curves on $C^{*}$-algebras in the multivariate case, and prove their relationship to the classical definitions of connections, curvatures and covariant derivatives on holomorphic curves, which can be viewed as the natural and essential generalizations of the classical definition.

\begin{definition}\label{D22}
Let $\mathcal{I}\in\mathcal{I}_{n}(\Omega,\mathcal{U})$ and $\mathcal{I}(\lambda)=F(\lambda)(G^{*}(\lambda)F(\lambda))^{-1}G^{*}(\lambda)$, define the connection and curvature of the extended holomorphic curves $\mathcal{I}$ as
\begin{equation}\label{eq1}
\pmb{\Theta}=(\partial F)H^{-1}G^{*}-\partial \mathcal{I}\quad \text{and}\quad\mathscr{K}(\mathcal{I})=
\overline{\partial}\mathcal{I}\partial\mathcal{I},
\end{equation}
respectively, where $H(\lambda)=G^{*}(\lambda)F(\lambda)$. And the covariant derivative of the curvature $\mathscr{K}(\mathcal{I})$ is defined as follows:
\begin{equation}\label{e2.3}
\mathscr{K}_{I+e_{l},J}(\mathcal{I})=\mathscr{K}_{\sum\limits_{p\neq l}i_{p}e_{p},J}(\mathcal{I})
+\mathcal{I}(\partial_{l}(\mathscr{K}_{i_{l}e_{l},J}(\mathcal{I}))),\quad i_{p}\neq 0,
\end{equation}
\begin{equation}\label{e2.4}
\mathscr{K}_{I,J+e_{k}}(\mathcal{I})=\mathscr{K}_{I,\sum\limits_{q\neq k}j_{q}e_{q}}(\mathcal{I})
+(\overline{\partial}_{k}(\mathscr{K}_{I,j_{k}e_{k}}(\mathcal{I})))\mathcal{I},\quad j_{q}\neq0,
\end{equation}
where $I=(i_{1},\ldots,i_{m})$ and $J=(j_{1},\ldots,j_{m})$ in $\mathbb{Z}_{+}^{m}.$
\end{definition}

\begin{remark}\label{re2}
Similar to the covariant derivatives of classical curvature mentioned in the previous subsection, note that for $1\leq l,k\leq m$, there are
$$\mathscr{K}_{0+e_{l},e_{k}}(\mathcal{I})=\mathcal{I}(\partial_{l}(\mathscr{K}_{0,e_{k}}))=
\overline{\partial}_{k}\partial\mathcal{I}\partial_{l}\partial\mathcal{I}-\overline{\partial}\mathcal{I}\partial_{l}\mathcal{I}\overline{\partial}_{k}\mathcal{I}\partial\mathcal{I}-
\overline{\partial}_{k}\mathcal{I}\partial_{l}\mathcal{I}\overline{\partial}\mathcal{I}\partial\mathcal{I}$$
and
$$\mathscr{K}_{e_{l},0+e_{k}}(\mathcal{I})=(\overline{\partial}_{k}(\mathscr{K}_{e_{l},0}))\mathcal{I}
=\overline{\partial}_{k}\partial\mathcal{I}\partial_{l}\partial\mathcal{I}-\overline{\partial}\mathcal{I}\partial_{l}\mathcal{I}\overline{\partial}_{k}\mathcal{I}\partial\mathcal{I}
-\overline{\partial}\mathcal{I}\partial\mathcal{I}\overline{\partial}_{k}\mathcal{I}\partial_{l}\mathcal{I},$$ which means that the covariant derivative $\mathscr{K}_{I,J}(\mathcal{I}),I,J\in\mathbb{Z}_{+}^{m}\setminus\{0\},$ of curvature of extended holomorphic curve on $C^{*}$-algebras depends on the order of $I$ and $J$.
\end{remark}

M. J. Cowen and R. G. Douglas \cite{CD,CD3} used complex geometry techniques to prove that the classical curvature and its covariant derivatives are complete unitary invariants of the Cowen-Douglas class. The curvature and its covariant derivatives of holomorphic curves on $C^{*}$-algebras in the case of one variable are proved by K. Ji to be completely unitary invariant in Theorem 2.9 \cite{CAOT}, while it is also shown from Theorem 2.10 and Theorem 2.18 \cite{CAOT} that the curvature and its covariant derivatives are also crucial in the similarity classification of holomorphic curves and extended holomorphic curves on $C^{*}$-algebras in the case of one variable. We use the following two lemmas to show that the covariant derivatives $\mathscr{K}_{I,J}(\mathcal{I})$ and
$\mathcal{K}_{\mathcal{I}, \lambda^{I},\overline{\lambda}^{J}}$ of curvature in Definition \ref{def2.2} and Definition \ref{D22} are equal to the sum of forms $\mathscr{K}_{i_{p}e_{p},j_{q}e_{q}}(\mathcal{I})$ and $\mathcal{K}_{\mathcal{I}, \lambda^{i_{p}e_{p}},\overline{\lambda}^{j_{q}e_{q}}},$ respectively, which will be useful for the following study of the covariant derivative of the curvature of extended holomorphic curves on $C^{*}$-algebras.

\begin{lemma}\label{D3}
Let $\mathcal{I}\in\mathcal{I}_{n}(\Omega,\mathcal{U})$ be an extended holomorphic curve on $C^{*}$-algebras. Then the covariant derivative $\mathscr{K}_{I,J}(\mathcal{I})$ of the curvature $\mathscr{K}(\mathcal{I})$ of $\mathcal{I}$ has the following properties:
\begin{itemize}
  \item [(1)]$\mathscr{K}_{I,0}(\mathcal{I})=\sum\limits_{i_{p}\neq 0}\mathscr{K}_{i_{p}e_{p},0}(\mathcal{I})$ for any $I=\sum\limits_{p=1}^{m}i_{p}e_{p}\in\mathbb{Z}_{+}^{m}\setminus\{0\}$;
  \item [(2)]$\mathscr{K}_{0,J}(\mathcal{I})=\sum\limits_{j_{q}\neq 0}
  \mathscr{K}_{0,j_{q}e_{q}}(\mathcal{I})$ for any $J=\sum\limits_{q=1}^{m}j_{q}e_{q}\in\mathbb{Z}_{+}^{m}\setminus\{0\}$;
  \item [(3)] $\mathscr{K}_{I,J}(\mathcal{I})=\sum\limits_{i_{p},j_{q}\neq 0}
  \mathscr{K}_{i_{p}e_{p},j_{q}e_{q}}(\mathcal{I})$, where the right derivative order of the equation follows the left derivative order and $I,J\in\mathbb{Z}_{+}^{m}\setminus\{0\}.$
\end{itemize}
\end{lemma}
\begin{proof}
Remark \ref{re2} shows that the covariant derivative $\mathscr{K}_{I,J}(\mathcal{I}),I,J\in\mathbb{Z}_{+}^{m}\setminus\{0\},$ of the curvature $\mathscr{K}(\mathcal{I})$ depends on the order of $I$ and $J$.
If $I=e_{p_{1}}+e_{p_{2}}, 1\leq p_{1}\neq p_{2}\leq m,$ and the order of $\mathscr{K}_{I,0}(\mathcal{I})$ is first $\lambda_{p_{1}}$ followed by $\overline{\lambda}_{p_{2}}$,
then from (\ref{e2.3}),
$$\mathscr{K}_{e_{p_{1}}+e_{p_{2}},0}(\mathcal{I})=\mathscr{K}_{e_{p_{1}},0}(\mathcal{I})+
\mathcal{I}(\partial_{p_{2}}(\mathscr{K}(\mathcal{I})))=
\mathscr{K}_{e_{p_{1}},0}(\mathcal{I})+\mathscr{K}_{e_{p_{2}},0}(\mathcal{I}).$$
Without losing generality, assume that $K\in\mathbb{Z}_{+}^{m}$ and $\mathscr{K}_{I,0}(\mathcal{I})=\sum\limits_{i_{p}\neq0}\mathscr{K}_{i_{p}e_{p},0}(\mathcal{I})$
for all $0<I\leq K$.
For $\mathscr{K}_{I+e_{l},0}(\mathcal{I})$, where $1\leq l\leq m$ and $I=\sum\limits_{p=1}^{m}i_{p}e_{p}$, we know that
$$ \mathscr{K}_{I+e_{l},0}(\mathcal{I})=\mathscr{K}_{\sum\limits_{i_{p}\neq0\atop p\neq l}i_{p}e_{p},0}(\mathcal{I})+\mathcal{I}(\partial_{l}(\mathscr{K}_{i_{l}e_{l},0}(\mathcal{I})))=\sum\limits_{i_{p}\neq0\atop p\neq l}\mathscr{K}_{i_{p}e_{p},0}(\mathcal{I})+\mathscr{K}_{(i_{l}+1)e_{l},0}(\mathcal{I}).$$
This means that
\begin{equation}\label{eq2}
\mathscr{K}_{I,0}(\mathcal{I})=\sum\limits_{i_{p}\neq0}\mathscr{K}_{i_{p}e_{p},0}(\mathcal{I}),\quad I=\sum\limits_{p=1}^{m}i_{p}e_{p}\in\mathbb{Z}_{+}^{m}\setminus\{0\}.
\end{equation}
Similarly, we have that $\mathscr{K}_{0,J}(\mathcal{I})=\sum\limits_{j_{q}\neq0}\mathscr{K}_{j_{q}e_{q},0}(\mathcal{I})$ for any $J=\sum\limits_{p=1}^{m}j_{p}e_{p}\in\mathbb{Z}_{+}^{m}\setminus\{0\}.$
For property $(3)$, letting $I:=e_{p_{1}}+e_{p_{2}}$ and $J:=e_{q_{1}}+e_{q_{2}}$, where $1\leq p_{1},p_{2},q_{1},q_{2}\leq m$. Assume that the derivative order of the covariant derivative $\mathscr{K}_{I,J}(\mathcal{I})$ is $\lambda_{p_{1}}, \overline{\lambda}_{q_{1}}, \lambda_{p_{2}}, \overline{\lambda}_{q_{2}}$. Therefore, from Definition \ref{D22}, we have that
$$\mathscr{K}_{0+e_{p_{1}},0}(\mathcal{I})=
\mathcal{I}\partial_{p_{1}}(\mathscr{K}(\mathcal{I}))\quad\text{and}\quad\mathscr{K}_{e_{p_{1}},0+e_{q_{1}}}(\mathcal{I})=(\overline{\partial}_{q_{1}}
(\mathscr{K}_{e_{p_{1}},0}(\mathcal{I})))\mathcal{I}.$$
If $p_{1}\neq p_{2}$, we have that
\begin{equation}\label{eq3}
\mathscr{K}_{e_{p_{1}}+e_{p_{2}},e_{q_{1}}}(\mathcal{I})=
\mathscr{K}_{e_{p_{1}},e_{q_{1}}}(\mathcal{I})+
\mathcal{I}(\partial_{p_{2}}(\mathscr{K}_{0,e_{q_{1}}}(\mathcal{I})))
=\mathscr{K}_{e_{p_{1}},e_{q_{1}}}(\mathcal{I})+\mathscr{K}_{e_{p_{2}},e_{q_{1}}}(\mathcal{I}),
\end{equation}
then
\begin{eqnarray*}
\mathscr{K}_{I,e_{q_{1}}+e_{q_{2}}}(\mathcal{I})&=&\mathscr{K}_{I,e_{q_{1}}}(\mathcal{I})+
(\overline{\partial}_{q_{2}}(\mathscr{K}_{I,0}(\mathcal{I})))\mathcal{I}\\[4pt]\nonumber
&=&\mathscr{K}_{e_{p_{1}},e_{q_{1}}}(\mathcal{I})+\mathscr{K}_{e_{p_{2}},e_{q_{1}}}(\mathcal{I})+
(\overline{\partial}_{q_{2}}(\mathscr{K}_{e_{p_{1}},0}(\mathcal{I})+
\mathscr{K}_{e_{p_{2}},0}(\mathcal{I})))\mathcal{I}\\[4pt]\nonumber
&=&\mathscr{K}_{e_{p_{1}},e_{q_{1}}}(\mathcal{I})+\mathscr{K}_{e_{p_{2}},e_{q_{1}}}(\mathcal{I})+\mathscr{K}_{e_{p_{1}},e_{q_{2}}}(\mathcal{I})+
\mathscr{K}_{e_{p_{2}},e_{q_{2}}}(\mathcal{I})\nonumber
\end{eqnarray*}
for $q_{1}\neq q_{2}$, and
\begin{eqnarray*}
&&\mathscr{K}_{I,e_{q_{1}}+e_{q_{2}}}(\mathcal{I})\\[4pt]\nonumber
&=&(\overline{\partial}_{q_{2}}(\mathscr{K}_{e_{p_{1}}+e_{p_{2}},e_{q_{1}}}(\mathcal{I})))\mathcal{I}\\[4pt]\nonumber
&=&(\overline{\partial}_{q_{2}}(\mathscr{K}_{e_{p_{1}},e_{q_{1}}}(\mathcal{I})
+\mathscr{K}_{e_{p_{2}},e_{q_{1}}}(\mathcal{I})))\mathcal{I}\\[4pt]\nonumber
&=&\mathscr{K}_{e_{p_{1}},2e_{q_{1}}}(\mathcal{I})+\mathscr{K}_{e_{p_{2}},2e_{q_{1}}}(\mathcal{I})\nonumber
\end{eqnarray*}
for $q_{1}=q_{2}$.
If $p_{1}= p_{2}$, we have $\mathscr{K}_{e_{p_{1}}+e_{p_{2}},e_{q_{1}}}(\mathcal{I})=
\mathscr{K}_{2e_{p_{1}},e_{q_{1}}}(\mathcal{I})$, then
$$\mathscr{K}_{I,e_{q_{1}}+e_{q_{2}}}(\mathcal{I})=\mathscr{K}_{I,e_{q_{1}}}(\mathcal{I})+
(\overline{\partial}_{q_{2}}(\mathscr{K}_{I,0}(\mathcal{I})))\mathcal{I}=\mathscr{K}_{2e_{p_{1}},e_{q_{1}}}(\mathcal{I})+\mathscr{K}_{2e_{p_{1}},e_{q_{2}}}(\mathcal{I})$$
for $q_{1}\neq q_{2}$, and
$$\mathscr{K}_{I,e_{q_{1}}+e_{q_{2}}}(\mathcal{I})=
(\overline{\partial}_{q_{2}}(\mathscr{K}_{e_{p_{1}}+e_{p_{2}},e_{q_{1}}}(\mathcal{I})))\mathcal{I}=
(\overline{\partial}_{q_{2}}(\mathscr{K}_{2e_{p_{1}},e_{q_{1}}}(\mathcal{I})))\mathcal{I}
=\mathscr{K}_{2e_{p_{1}},2e_{q_{1}}}(\mathcal{I})$$
for $q_{1}=q_{2}$. Without losing generality, assume that for some $K\in\mathbb{Z}_{+}^{m}\setminus\{0\}$, there is
$$\mathscr{K}_{I,J}(\mathcal{I})=\sum\limits_{i_{p},j_{q}\neq 0}
  \mathscr{K}_{i_{p}e_{p},j_{q}e_{q}}(\mathcal{I}),\quad 0\leq I,J\leq K,$$
where the right derivative order of the equation follows the left derivative order.
Now we just have to prove that there is property (3) for $\mathscr{K}_{I+e_{l},J}(\mathcal{I})$ and $\mathscr{K}_{I,J+e_{k}}(\mathcal{I})$, $1\leq l,k\leq m.$
Since
\begin{eqnarray*}
\mathscr{K}_{I+e_{l},J}(\mathcal{I})&=&\mathscr{K}_{\sum\limits_{i_{p}\neq0\atop p\neq l}i_{p}e_{p},J}(\mathcal{I})+\mathcal{I}(\partial_{l}(\mathscr{K}_{i_{l}e_{l},J}(\mathcal{I})))\\ \nonumber
&=&\sum_{i_{p},j_{q}\neq0\atop p\neq l}\mathscr{K}_{i_{p}e_{p},j_{q}e_{q}}(\mathcal{I})
+\mathcal{I}\left(\partial_{l}\left(\sum_{j_{q}\neq0}\mathscr{K}_{i_{l}e_{l},j_{q}e_{q}}(\mathcal{I})\right)\right)\\ \nonumber
&=&\sum_{i_{p},j_{q}\neq0\atop p\neq l}\mathscr{K}_{i_{p}e_{p},j_{q}e_{q}}(\mathcal{I})
+\sum_{j_{q}\neq0}\mathscr{K}_{(i_{l}+1)e_{l},j_{q}e_{q}}(\mathcal{I})\nonumber
\end{eqnarray*}
and
\begin{eqnarray*}
\mathscr{K}_{I,J+e_{k}}(\mathcal{I})&=&\mathscr{K}_{I,\sum\limits_{j_{q}\neq0\atop q\neq k}j_{q}e_{q}}(\mathcal{I})
+(\overline{\partial}_{k}(\mathscr{K}_{I,j_{k}e_{k}}(\mathcal{I})))\mathcal{I}\\ \nonumber
&=&\sum_{i_{p},j_{q}\neq0\atop q\neq k}\mathscr{K}_{i_{p}e_{p},j_{q}e_{q}}(\mathcal{I})
+\left(\overline{\partial}_{k}\left(\sum_{i_{p}\neq0}\mathscr{K}_{i_{p}e_{p},j_{k}e_{k}}(\mathcal{I})\right)\right)\mathcal{I}\\ \nonumber
&=&\sum_{i_{p},j_{q}\neq0\atop q\neq k}\mathscr{K}_{i_{p}e_{p},j_{q}e_{q}}(\mathcal{I})
+\sum_{i_{p}\neq0}\mathscr{K}_{i_{p}e_{p},(j_{k}+1)e_{k}}(\mathcal{I}).\nonumber
\end{eqnarray*}
This completes the proof.
\end{proof}

Similar to the proof of Lemma \ref{D3}, for the covariant derivative of classical curvature we can immediately prove the following result.

\begin{lemma}\label{lm0}
Let $\mathcal{I}\in\mathcal{I}_{n}(\Omega,\mathcal{U})$. Then the covariant derivative $\mathcal{K}_{\mathcal{I}, \lambda^{I},\overline{\lambda}^{J}}$ of classical curvature $\mathcal{K}_{\mathcal{I}}$ has the following properties:
\begin{itemize}
  \item [(1)]$\mathcal{K}_{\mathcal{I}, \lambda^{I}}=\sum\limits_{i_{p}\neq 0}\mathcal{K}_{\mathcal{I}, \lambda^{i_{p}e_{p}}}$ for any $I\in\mathbb{Z}_{+}^{m}\setminus\{0\}$;
  \item [(2)]$\mathcal{K}_{\mathcal{I},\overline{\lambda}^{J}}=\sum\limits_{j_{q}\neq 0}\mathcal{K}_{\mathcal{I}, \overline{\lambda}^{j_{q}e_{q}}}$ for any $J\in\mathbb{Z}_{+}^{m}\setminus\{0\}$;
  \item [(3)] $\mathcal{K}_{\mathcal{I}, \lambda^{I},\overline{\lambda}^{J}}=\sum\limits_{i_{p},j_{q}\neq 0}
  \mathcal{K}_{\mathcal{I}, \lambda^{i_{p}e_{p}},\overline{\lambda}^{j_{q}e_{q}}}$, where the right derivative order of the equation follows the left derivative order and $I,J\in\mathbb{Z}_{+}^{m}\setminus\{0\}.$
\end{itemize}
\end{lemma}

The following theorem is one of our main results, which reveals the relationship between the curvature and covariant derivatives and classical curvature and covariant derivatives of extended holomorphic curves in the multivariate case on $C^{*}$-algebras.

\begin{thm}\label{thm1}
Let $\mathcal{I}\in\mathcal{I}_{n}(\Omega,\mathcal{U})$ and $\mathcal{I}(\lambda)=F(\lambda)(G^{*}(\lambda)F(\lambda))^{-1}G^{*}(\lambda)$. Then
$$\mathscr{K}_{I,J}(\mathcal{I})(\lambda)F(\lambda)+
F(\lambda)\mathcal{K}_{\mathcal{I},\lambda^{I},\overline{\lambda}^{J}}(\lambda)=0,
\quad\lambda\in\Omega,\,\, I,J\in \mathbb{Z}_{+}^{m},$$
where $F:\Omega\rightarrow \mathscr{L}(\mathcal{H},l^{2}(\mathbb{N},B))$ is an operator valued function.
\end{thm}
\begin{proof}
Letting $H(\lambda)=G^{*}(\lambda)F(\lambda),$ we have that
\begin{eqnarray}\label{eq4}
\overline{\partial}\mathcal{I}\partial\mathcal{I}&=&\overline{\partial}(FH^{-1}G^{*})\partial(FH^{-1}G^{*})\\[4pt]\nonumber
&=&F(\overline{\partial}H^{-1} G^{*}+H^{-1}\overline{\partial} G^{*})(\partial F H^{-1}+ F\partial H^{-1})G^{*}\\[4pt]\nonumber
&=&F(\overline{\partial}H^{-1}\partial H H^{-1}+\overline{\partial}H^{-1} H \partial H^{-1}+
H^{-1}\overline{\partial}\partial H H^{-1}+H^{-1}\overline{\partial} H \partial H^{-1})G^{*}\\[4pt]\nonumber
&=&F\overline{\partial}(H^{-1}\partial H)H^{-1}G^{*}.\nonumber
\end{eqnarray}
From (\ref{e0}),(\ref{eq1}) and (\ref{eq4}), we have $\mathscr{K}(\mathcal{I})= F(-\mathcal{K}_{\mathcal{I}})H^{-1} G^{*}$, and then
\begin{equation}\label{e1}
\mathscr{K}(\mathcal{I})F+F\mathcal{K}_{\mathcal{I}}=0.
\end{equation}
Since $\mathcal{I}$ is an extended holomorphic curve, from Definition \ref{def2.1}, we obtain that $\partial_{l}\overline{\partial}\mathcal{I}=\partial_{l}\mathcal{I}\overline{\partial}
\mathcal{I}-\overline{\partial}\mathcal{I}\partial_{l}\mathcal{I}$, $$\mathscr{K}_{e_{l},0}(\mathcal{I})=\mathcal{I}(\partial_{l}(\mathscr{K}(\mathcal{I})))
=\mathcal{I}(\partial_{l}\overline{\partial}\mathcal{I}\partial\mathcal{I}
+\overline{\partial}\mathcal{I}\partial_{l}\partial\mathcal{I})=
\mathcal{I}(\partial_{l}\mathcal{I}\overline{\partial}\mathcal{I}\partial\mathcal{I}
-\overline{\partial}\mathcal{I}\partial_{l}\mathcal{I}\partial\mathcal{I})
+\mathcal{I}\overline{\partial}\mathcal{I}\partial_{l}\partial\mathcal{I}=
\overline{\partial}\mathcal{I}\partial_{l}\partial\mathcal{I},$$
and
\begin{equation}
  \begin{aligned}
\overline{\partial}\mathcal{I}\partial_{l}\partial\mathcal{I}
&=\overline{\partial}( F H ^{-1} G^{*})\partial_{l}(\partial( F H ^{-1} G^{*}))\\
&= F(\overline{\partial}H^{-1} G^{*}+H^{-1}\overline{\partial} G^{*})
(\partial\partial_{l} F H^{-1}+\partial F\partial_{l}H^{-1}+\partial_{l} F\partial H^{-1}+ F\partial\partial_{l}H^{-1}) G^{*}
\\&= F(\overline{\partial}H^{-1}\partial\partial_{l}H+\overline{\partial}H^{-1}\partial H\partial_{l}H^{-1}H+\overline{\partial}H^{-1}\partial_{l}H\partial H^{-1}H\\
&\quad+H^{-1}\overline{\partial}\partial\partial_{l}H+H^{-1}\overline{\partial}\partial H\partial_{l}H^{-1}H+H^{-1}\overline{\partial}\partial_{l}H\partial H^{-1}H)H^{-1} G^{*}\\
&= F(-\partial_{l}(\mathcal{K}_{\mathcal{I}})+\overline{\partial} H^{-1}\partial H\partial_{l}H^{-1}H+H^{-1}\overline{\partial}\partial H\partial_{l}H^{-1}H-\partial_{l}H^{-1}\partial\overline{\partial} H-\partial_{l}H^{-1}\overline{\partial} H\partial H^{-1}H)H^{-1} G^{*}\\
&= F(-\partial_{l}(\mathcal{K}_{\mathcal{I}})-(\overline{\partial} H^{-1}\partial H+H^{-1}\partial\overline{\partial} H)H^{-1}\partial_{l}H+H^{-1}\partial_{l}H(H^{-1}\partial\overline{\partial} H+\overline{\partial} H^{-1}\partial H))H^{-1} G^{*}\\
&= F(-\partial_{l}(\mathcal{K}_{\mathcal{I}})-[H^{-1}\partial_{l}H,\mathcal{K}_{\mathcal{I}}])H^{-1} G^{*}\\
&= F(-\mathcal{K}_{\mathcal{I},\lambda^{e_{l}}})H^{-1} G^{*}
\nonumber
  \end{aligned}
\end{equation}
for any $1\leq l\leq m$. Thus,
\begin{equation}
  \begin{aligned}\label{e2}
\mathscr{K}_{e_{l},0}(\mathcal{I})F+ F\mathcal{K}_{\mathcal{I},\lambda^{e_{l}}}=0,\quad 1\leq l\leq m.
  \end{aligned}
\end{equation}
Similarly,
\begin{equation}
  \begin{aligned}\label{e3}
\mathscr{K}_{0,e_{k}}(\mathcal{I})F+F\mathcal{K}_{\mathcal{I},\overline{\lambda}^{e_{k}}}=0,\quad 1\leq k\leq m.
   \end{aligned}
\end{equation}
Since the covariant derivatives of the curvature depends on the order of $I$ and $J$,
for $\mathscr{K}_{e_{l},e_{k}}(\mathcal{I})=\mathcal{I}(\partial_{l}(\mathscr{K}_{0,e_{k}}(\mathcal{I}))), 1\leq l,k\leq m,$
\begin{eqnarray}\label{eq5}
\mathscr{K}_{e_{l},e_{k}}(\mathcal{I})&=&\mathcal{I}(\partial_{l}(\mathscr{K}_{0,e_{k}}(\mathcal{I})))\\ \nonumber
&=&\mathcal{I}(\partial_{l}(\overline{\partial}_{k}\overline{\partial}\mathcal{I}
\partial\mathcal{I}))\\ \nonumber
&=&\mathcal{I}(\partial_{l}\overline{\partial}_{k}\overline{\partial}
\mathcal{I}\partial\mathcal{I}+\overline{\partial}_{k}\overline{\partial}\mathcal{I}
\partial_{l}\partial\mathcal{I})\\ \nonumber
&=&\mathcal{I}(\overline{\partial}(\partial_{l}\mathcal{I}\overline{\partial}_{k}\mathcal{I}
-\overline{\partial}_{k}\mathcal{I}\partial_{l}\mathcal{I}))\partial\mathcal{I}
+\overline{\partial}_{k}\overline{\partial}\mathcal{I}\partial_{l}\partial\mathcal{I}\\ \nonumber
&=&-\overline{\partial}\mathcal{I}\partial_{l}\mathcal{I}\overline{\partial}_{k}\mathcal{I}
\partial\mathcal{I}-\overline{\partial}_{k}\mathcal{I}\partial_{l}\mathcal{I}
\overline{\partial}\mathcal{I}\partial\mathcal{I}+\overline{\partial}_{k}\overline{\partial}
\mathcal{I}\partial_{l}\partial\mathcal{I}.
\nonumber
\end{eqnarray}
Note that
$-\mathcal{K}_{\mathcal{I},\overline{\lambda}^{e_{k}}}=\overline{\partial}_{k}(\overline{\partial}(H^{-1}\partial H))=\overline{\partial}_{k}\overline{\partial}H^{-1}\partial H+\overline{\partial}H^{-1}\partial\overline{\partial}_{k}H+
\overline{\partial}_{k}H^{-1}\partial\overline{\partial}H+H^{-1}\partial\overline{\partial}_{k}
\overline{\partial}H,$
then
\begin{equation*}
 \begin{aligned}\label{e222}
&\quad-(\mathcal{K}_{\mathcal{I},\overline{\lambda}^{e_{k}}})_{\lambda^{e_{l}}}\\
&=-\left(\partial_{l}(\mathcal{K}_{\mathcal{I},\overline{\lambda}^{e_{k}}})+
[H^{-1}\partial_{l}H,\mathcal{K}_{\mathcal{I},\overline{\lambda}^{e_{k}}}]\right)\\
&=\partial_{l}(\overline{\partial}_{k}\overline{\partial}H^{-1}\partial H+\overline{\partial}H^{-1}\partial\overline{\partial}_{k}H+\overline{\partial}_{k}H^{-1}\partial\overline{\partial}H+H^{-1}\partial\overline{\partial}_{k}\overline{\partial}H
)+H^{-1}\partial_{l}H(\overline{\partial}_{k}\overline{\partial}H^{-1}\partial H+\overline{\partial}H^{-1}\partial\overline{\partial}_{k}H\\
&\quad+\overline{\partial}_{k}H^{-1}\partial\overline{\partial}H+H^{-1}
\partial\overline{\partial}_{k}\overline{\partial}H)-(\overline{\partial}_{k}\overline{\partial}H^{-1}\partial H+\overline{\partial}H^{-1}\partial\overline{\partial}_{k}H+\overline{\partial}_{k}H^{-1}\partial\overline{\partial}H+H^{-1}\partial\overline{\partial}_{k}\overline{\partial}H)H^{-1}\partial_{l}H\\
&=\partial_{l}\overline{\partial}_{k}\overline{\partial} H^{-1}\partial H+\overline{\partial}_{k}\overline{\partial} H^{-1}\partial_{l}\partial H+\partial_{l}\overline{\partial} H^{-1}\partial\overline{\partial}_{k}H+\overline{\partial} H^{-1}\partial_{l}\partial\overline{\partial}_{k}H+\partial_{l}\overline{\partial}_{k}H^{-1}\partial\overline{\partial} H+\overline{\partial}_{k}H^{-1}\partial_{l}\partial\overline{\partial} H\\
&\quad+H^{-1}\partial_{l}\partial\overline{\partial}_{k}\overline{\partial} H+H^{-1}\partial_{l}H\overline{\partial}_{k}\overline{\partial} H^{-1}\partial H+H^{-1}\partial_{l}H\overline{\partial} H^{-1}\partial\overline{\partial}_{k}H+H^{-1}\partial_{l}H\overline{\partial}_{k}H^{-1}\partial\overline{\partial} H\\
&\quad-\overline{\partial}_{k}\overline{\partial} H^{-1}\partial HH^{-1}\partial_{l}H-\overline{\partial} H^{-1}\partial\overline{\partial}_{k}HH^{-1}\partial_{l}H-\overline{\partial}_{k}H^{-1}\partial\overline{\partial} HH^{-1}\partial_{l}H-H^{-1}\partial\overline{\partial}_{k}\overline{\partial} HH^{-1}\partial_{l}H
 \end{aligned}
\end{equation*}
and
\begin{eqnarray}\label{eq6}
&&\quad\quad F(-(\mathcal{K}_{\mathcal{I},\overline{\lambda}^{e_{k}}})_{\lambda^{e_{l}}})H^{-1} G^{*}\\ \nonumber
&=&\overline{\partial}_{k}\overline{\partial}\mathcal{I}\partial_{l}\partial\mathcal{I}+ F(\partial_{l}\overline{\partial} H^{-1}\partial\overline{\partial}_{k}H+\partial_{l}\overline{\partial}_{k}H^{-1}\partial\overline{\partial} H+H^{-1}\partial_{l}H\overline{\partial} H^{-1}\partial\overline{\partial}_{k}H+H^{-1}\partial_{l}H\overline{\partial}_{k}H^{-1}\partial\overline{\partial} H\\ \nonumber
&&-\partial_{l}\overline{\partial}_{k}H^{-1}\overline{\partial} HH^{-1}\partial H-H^{-1}\partial_{l}\overline{\partial} H\overline{\partial}_{k}H^{-1}\partial H-\overline{\partial} H^{-1}\partial_{l}H\overline{\partial}_{k}H^{-1}\partial H+\partial_{l}H^{-1}\overline{\partial}_{k}H\overline{\partial} H^{-1}\partial H)H^{-1} G^{*}\\ \nonumber
&=&\overline{\partial}_{k}\overline{\partial}\mathcal{I}\partial_{l}\partial\mathcal{I}+ F(\partial_{l}\overline{\partial} H^{-1}\partial\overline{\partial}_{k}H+\partial_{l}\overline{\partial}_{k}H^{-1}\partial\overline{\partial} H+H^{-1}\partial_{l}H\overline{\partial} H^{-1}\partial\overline{\partial}_{k}H+H^{-1}\partial_{l}H\overline{\partial}_{k}H^{-1}\partial\overline{\partial} H\\ \nonumber
&&+\partial_{l}\overline{\partial}_{k}H^{-1}H\overline{\partial} H^{-1}\partial H+\partial_{l}H^{-1}\overline{\partial} H\overline{\partial}_{k}H^{-1}\partial H+\partial_{l}\overline{\partial} H^{-1}H\overline{\partial}_{k}H^{-1}\partial H+\partial_{l}H^{-1}\overline{\partial}_{k}H\overline{\partial} H^{-1}\partial H)H^{-1} G^{*}\\ \nonumber
&=&\overline{\partial}_{k}\overline{\partial}\mathcal{I}\partial_{l}\partial\mathcal{I}+ F(\overline{\partial}(\partial_{l}H^{-1}H)\overline{\partial}_{k}(H^{-1}\partial H)+\overline{\partial}_{k}(\partial_{l}H^{-1}H)\overline{\partial}(H^{-1}\partial H))H^{-1} G^{*}\\ \nonumber
&=&\overline{\partial}_{k}\overline{\partial}\mathcal{I}\partial_{l}\partial\mathcal{I}+ F(-\overline{\partial}(H^{-1}\partial_{l}H)\overline{\partial}_{k}(H^{-1}\partial H)-\overline{\partial}_{k}(H^{-1}\partial_{l}H)\overline{\partial}(H^{-1}\partial H))H^{-1} G^{*}\\ \nonumber
&=&\overline{\partial}_{k}\overline{\partial}\mathcal{I}\partial_{l}\partial\mathcal{I}
-\overline{\partial}\mathcal{I}\partial_{l}\mathcal{I}\overline{\partial}_{k}\mathcal{I}
\partial\mathcal{I}-\overline{\partial}_{k}\mathcal{I}\partial_{l}\mathcal{I}
\overline{\partial}\mathcal{I}\partial\mathcal{I}.
\nonumber
\end{eqnarray}
Therefore, from (\ref{eq5}) and (\ref{eq6}), we obtain $\mathcal{I}\left(\partial_{l}(\mathscr{K}_{0,e_{k}}(\mathcal{I}))\right)F+
F(\mathcal{K}_{\mathcal{I},\overline{\lambda}^{e_{k}}})_{\lambda^{e_{l}}}=0$. Similarly, we have
$(\overline{\partial}_{k}(\mathscr{K}_{e_{l},0}(\mathcal{I})))\mathcal{I}F+F(\mathcal{K}_{\mathcal{I},\lambda^{e_{l}}})_{\overline{\lambda}^{e_{k}}}=0.$
Thus,
\begin{equation}
 \begin{aligned}\label{e4}
 \mathscr{K}_{e_{l},e_{k}}(\mathcal{I})F+ F\mathcal{K}_{\mathcal{I},\lambda^{e_{l}},\overline{\lambda}^{e_{k}}}=0,\quad 1\leq l,k\leq m.
 \end{aligned}
\end{equation}
By (\ref{e1})-(\ref{e3}) and (\ref{e4}),
we have that $\mathscr{K}_{i_{l}e_{l},j_{k}e_{k}}(\mathcal{I})F+ F\mathcal{K}_{\mathcal{I},\lambda^{i_{l}e_{l}},\overline{\lambda}^{j_{k}e_{k}}}=0$ for all $0\leq i_{l},j_{k}\leq1.$
Without loss of generality, assume that
\begin{equation}\label{eq7}
\mathscr{K}_{i_{l}e_{l},j_{k}e_{k}}(\mathcal{I})F+ F\mathcal{K}_{\mathcal{I},\lambda^{i_{l}e_{l}},\overline{\lambda}^{j_{k}e_{k}}}=0
\end{equation}
holds for some $i_{l},j_{k}\geq 0$. Then
$\mathscr{K}_{i_{l}e_{l},j_{k}e_{k}}(\mathcal{I})=\mathscr{K}_{i_{l}e_{l},j_{k}e_{k}}(\mathcal{I}) \mathcal{I}$ and (\ref{eq7}) show that
\begin{equation}
\begin{aligned}
&\quad \mathscr{K}_{i_{l}e_{l},j_{k}e_{k}+e_{k}}(\mathcal{I})\\
&=(\overline{\partial}_{k}(\mathscr{K}_{i_{l}e_{l},j_{k}e_{k}}(\mathcal{I})))\mathcal{I}\\
&=\overline{\partial}_{k}(- F\mathcal{K}_{\mathcal{I},\lambda^{i_{l}e_{l}},\overline{\lambda}^{j_{k}e_{k}}}H^{-1} G^{*})\mathcal{I}\\
&=-F\left(\overline{\partial}_{k}\mathcal{K}_{\mathcal{I},\lambda^{i_{l}e_{l}},
\overline{\lambda}^{j_{k}e_{k}}}H^{-1} G^{*}+\mathcal{K}_{\mathcal{I},\lambda^{i_{l}e_{l}},\overline{\lambda}^{j_{k}e_{k}}}
\overline{\partial}_{k}H^{-1}G^{*}
+\mathcal{K}_{\mathcal{I},\lambda^{i_{l}e_{l}},\overline{\lambda}^{j_{k}e_{k}}}
H^{-1}\overline{\partial}_{k}G^{*}\right) F H^{-1} G^{*}\\
&= -F\left(\overline{\partial}_{k}\mathcal{K}_{\mathcal{I},\lambda^{i_{l}e_{l}},
\overline{\lambda}^{j_{k}e_{k}}}H^{-1}G^{*}
+\mathcal{K}_{\mathcal{I},\lambda^{i_{l}e_{l}},\overline{\lambda}^{j_{k}e_{k}}}
(\overline{\partial}_{k}H^{-1}+H^{-1}\overline{\partial}_{k}HH^{-1})G^{*}\right)\\
&=-F\mathcal{K}_{\mathcal{I},\lambda^{i_{l}e_{l}},\overline{\lambda}^{j_{k}e_{k}+e_{k}}}H^{-1}G^{*}
\nonumber
\end{aligned}
\end{equation}
and
\begin{equation}
\begin{aligned}
&\quad \mathscr{K}_{i_{l}e_{l}+e_{l},j_{k}e_{k}}(\mathcal{I})\\
&=\mathcal{I}(\partial_{l}(\mathscr{K}_{i_{l}e_{l},j_{k}e_{k}}(\mathcal{I})))\\
&=\mathcal{I}\left(\partial_{l}( -F\mathcal{K}_{\mathcal{I},\lambda^{i_{l}e_{l}},\overline{\lambda}^{j_{k}e_{k}}}H^{-1} G^{*})\right)\\
&= -F H^{-1} G^{*}\left(\partial_{l} F\mathcal{K}_{\mathcal{I},\lambda^{i_{l}e_{l}},\overline{\lambda}^{j_{k}e_{k}}}H^{-1}G^{*}
+F\partial_{l}\mathcal{K}_{\mathcal{I},\lambda^{i_{l}e_{l}},
\overline{\lambda}^{j_{k}e_{k}}}H^{-1}G^{*}+ F\mathcal{K}_{\mathcal{I},\lambda^{i_{l}e_{l}},
\overline{\lambda}^{j_{k}e_{k}}}\partial_{l}H^{-1}G^{*}\right)\\
&=-F\left(H^{-1}\partial_{l}H\mathcal{K}_{\mathcal{I},\lambda^{i_{l}e_{l}},\overline{\lambda}^{j_{k}e_{k}}}+
\partial_{l}\mathcal{K}_{\mathcal{I},\lambda^{i_{l}e_{l}},\overline{\lambda}^{j_{k}e_{k}}}
-\mathcal{K}_{\mathcal{I},\lambda^{i_{l}e_{l}},\overline{\lambda}^{j_{k}e_{k}}}H^{-1}\partial_{l}H\right)H^{-1}G^{*}\\
&=-F\left(\partial_{l}\mathcal{K}_{\mathcal{I},\lambda^{i_{l}e_{l}},\overline{\lambda}^{j_{k}e_{k}}}+
\left[H^{-1}\partial_{l}H, \mathcal{K}_{\mathcal{I},\lambda^{i_{l}e_{l}},\overline{\lambda}^{j_{k}e_{k}}}\right]\right)H^{-1} G^{*}\\
&= -F\mathcal{K}_{\mathcal{I},\lambda^{i_{l}e_{l}+e_{l}},\overline{\lambda}^{j_{k}e_{k}}}H^{-1} G^{*}
\nonumber
\end{aligned}
\end{equation}
It follows that
\begin{equation}\label{e5}
\mathscr{K}_{i_{l}e_{l},j_{k}e_{k}}(\mathcal{I})F+ F\mathcal{K}_{\mathcal{I},\lambda^{i_{l}e_{l}},\overline{\lambda}^{j_{k}e_{k}}}=0,\quad i_{l},j_{k}\geq 0.
\end{equation}
By (\ref{e5}), Lemma \ref{D3} and Lemma \ref{lm0}, we have that
$$\mathscr{K}_{I,0}(\mathcal{I})=\sum\limits_{i_{l}\neq0}\mathscr{K}_{i_{l}e_{l},0}(\mathcal{I})
=-\sum\limits_{i_{l}\neq0} F\mathcal{K}_{\mathcal{I},\lambda^{i_{l}e_{l}}}H^{-1} G^{*}
= -F\mathcal{K}_{\mathcal{I},\lambda^{I}}H^{-1} G^{*},\quad  I\neq 0.$$
$$\mathscr{K}_{0,J}(\mathcal{I})=\sum\limits_{j_{k}\neq0}\mathscr{K}_{0,j_{k}e_{k}}(\mathcal{I})
=-\sum\limits_{j_{k}\neq0} F\mathcal{K}_{\mathcal{I},\overline{\lambda}^{j_{k}e_{k}}}H^{-1} G^{*}
= -F\mathcal{K}_{\mathcal{I},\overline{\lambda}^{J}}H^{-1} G^{*},\quad J\neq 0,$$
and
$$\mathscr{K}_{I,J}(\mathcal{I})=\sum\limits_{i_{l},j_{k}\neq0}\mathscr{K}_{i_{l}e_{l},j_{k}e_{k}}(\mathcal{I})
= -\sum\limits_{i_{l},j_{k}\neq0} F\mathcal{K}_{\mathcal{I},\lambda^{i_{l}e_{l}},\overline{\lambda}^{j_{k}e_{k}}}H^{-1}G^{*}
=-F\mathcal{K}_{\mathcal{I},\lambda^{I},\overline{\lambda}^{J}}H^{-1} G^{*},\quad I, J\neq 0.
$$
This means that
$\mathscr{K}_{I,J}(\mathcal{I})F+
F\mathcal{K}_{\mathcal{I},\lambda^{I},\overline{\lambda}^{J}}=0$
for any $I,J\in \mathbb{Z}_{+}^{m}.$
\end{proof}

\begin{corollary}
Let $\textbf{T}=(T_{1},\cdots,T_{m})\in\mathcal{B}_{n}^{m}(\Omega)\cap\mathcal{L}(\mathcal{H})^{m}.$  Then
$$\text{trace}\, \mathscr{K}_{I,J}(\mathcal{I}_{\textbf{T}})(\lambda)=
-\text{trace}\, \mathcal{K}_{\mathcal{I}_{\textbf{T}},\lambda^{I},\overline{\lambda}^{J}}(\lambda),
\quad \lambda \in\Omega,\,\, I,J\in\mathbb{Z}_{+}^{m}.$$
\end{corollary}
\begin{proof}
Setting $\{F_{i}\}_{i=1}^{n}$ and $\{G_{i}\}_{i=1}^{n}$ are two holomorphic frames of the Hermitian holomorphic vector bundle $E_{\textbf{T}}$, $F=(F_{1},\cdots,F_{n})$ and $G=(G_{1},\cdots,G_{n})$.
Then the mapping $\mathcal{I}_{\textbf{T}}:\Omega\rightarrow\mathcal{I}(\mathcal{L}(\mathcal{H}))$ by
$$\mathcal{I}_{\textbf{T}}(\lambda)=F(\lambda)(G^{*}(\lambda)F(\lambda))^{-1}G^{*}(\lambda).$$
is a extended holomorphic curve.
From Theorem \ref{thm1}, we have that
$\mathscr{K}_{I,J}(\mathcal{I}_{\textbf{T}})=
-F\mathcal{K}_{\mathcal{I}_{\textbf{T}},\lambda^{I},\overline{\lambda}^{J}}H^{-1}G^{*}$
where $H=G^{*}F$ and $I,J\in\mathbb{Z}_{+}^{m}.$
Then
$$\text{trace}\, \mathscr{K}_{I,J}(\mathcal{I}_{\textbf{T}})=
-\text{trace}\, (F\mathcal{K}_{\mathcal{I}_{\textbf{T}},\lambda^{I},\overline{\lambda}^{J}}H^{-1}G^{*})
=-\text{trace}\, \mathcal{K}_{\mathcal{I}_{\textbf{T}},\lambda^{I},\overline{\lambda}^{J}}.$$
\end{proof}

In particular, the curvature and its covariant derivatives of the Hermitian holomorphic vector bundle $\mathcal{E}_{\mathbf{T}}$ for $\mathbf{T}\in\mathcal{B}_{n}^{m}(\Omega)$, mentioned in \cite{CD}, depend on the choice of the holomorphic frame of $E_{\mathbf{T}}$. The trace of the curvature and its covariant derivatives of $E_{\mathbf{T}}$, however, does not depend on the choice of the holomorphic frame.

\begin{corollary}
Let $\mathbf{T}=(T_{1},T_{2},\ldots, T_{m})\in\mathcal{B}_{n}^{m}(\Omega)$. Then
the curvature of the Hermitian holomorphic vector bundle $\mathcal{E}_{\mathbf{T}}$, and its covariant derivatives, are similar in different holomorphic frames. Their traces, however, are unique, that is, their traces do not depend on the choice of holomorphic frames.
\end{corollary}
\begin{proof}
Letting $F=\{F_{1},\cdots,F_{n}\}$ and $G=\{G_{1},\cdots,G_{n}\}$
are two different frames of $E_{\mathbf{T}}$, and then there is an invertible holomorphic matrix $\phi(\lambda)=(\phi_{i,j}(\lambda))_{n\times n}$ such that
$$(G_{1}(\lambda),\cdots,G_{n}(\lambda))=(F_{1}(\lambda),\cdots,F_{n}(\lambda))\phi(\lambda),\quad\lambda\in\Omega.$$
Let $H=(\langle F_{j}(\lambda),F_{i}(\lambda)\rangle)_{n\times n}$ and $\widetilde{H}=(\langle G_{j}(\lambda),G_{i}(\lambda)\rangle)_{n\times n}$ be the metric matrices corresponding to frames $F$ and $G$, respectively. We obtain that
$$\widetilde{H}(\lambda)=(\langle G_{j}(\lambda),G_{i}(\lambda)\rangle)_{n\times n}=\left(\left\langle\sum\limits_{k=1}^{n}\phi_{k,j}(\lambda)F_{k}(\lambda),
\sum\limits_{k=1}^{n}\phi_{k,i}(\lambda)F_{k}(\lambda)\right\rangle\right)_{n\times n}=\overline{\phi^{T}}(\lambda)H(\lambda)\phi(\lambda)$$
where $\phi^{T}$ is the transposed matrix of $\phi$. Since $\phi(\lambda)$ is holomorphic and invertible, we have that
\begin{align*}
\mathcal{K}_{\mathbf{T}}(G)(\lambda)&=-\sum\limits_{i,j=1}^{m}\frac{\partial}{\partial\overline{\lambda}_{j}}(\widetilde{H}^{-1}(\lambda)\frac{\partial}{\partial\lambda_{i}}\widetilde{H}(\lambda))\\
&=-\sum\limits_{i,j=1}^{m}\frac{\partial}{\partial\overline{\lambda}_{j}}\left[\left(\phi^{T}(\lambda)H(\lambda)\phi(\lambda)\right)^{-1}\frac{\partial}{\partial\lambda_{i}}\left(\phi^{\mathsf{T}}(\lambda)H(\lambda)\phi(\lambda)\right)\right]\\
&=-\sum\limits_{i,j=1}^{m}\frac{\partial}{\partial\overline{\lambda}_{j}}\left[\phi^{-1}(\lambda)H^{-1}(\lambda)\left(\frac{\partial}{\partial\lambda_{i}}H(\lambda)\phi(\lambda)+H(\lambda)\frac{\partial}{\partial\lambda_{i}}\phi(\lambda)\right)\right]\\
&=-\sum\limits_{i,j=1}^{m}\frac{\partial}{\partial\overline{\lambda}_{j}}\left(\phi^{-1}(\lambda)H^{-1}(\lambda)\frac{\partial}{\partial\lambda_{i}}H(\lambda)\phi(\lambda)+\phi^{-1}(\lambda)\frac{\partial}{\partial\lambda_{i}}\phi(\lambda)\right)\\
&=-\sum\limits_{i,j=1}^{m}\left[\phi^{-1}(\lambda)\frac{\partial}{\partial\overline{\lambda}_{j}}\left(H^{-1}(\lambda)\frac{\partial}{\partial\lambda_{i}}H(\lambda)\right)\phi(\lambda)\right]\\
&=\phi^{-1}(\lambda)\mathcal{K}_{\mathbf{T}}(F)(\lambda)\phi(\lambda).
\end{align*}
Hence, for any $1\leq i, j\leq m,$ we have that
\begin{align*}
&\quad\quad~\mathcal{K}_{\mathbf{T}, \lambda_{i}}(G)(\lambda)\\
&=\frac{\partial}{\partial\lambda_{i}}\mathcal{K}_{\mathbf{T}}(G)(\lambda)+\left[\widetilde{H}^{-1}(\lambda)\frac{\partial}{\partial\lambda_{i}}\widetilde{H}(\lambda),\mathcal{K}_{\mathbf{T}}(G)(\lambda)\right]\\
&=\frac{\partial}{\partial\lambda_{i}}\left[\phi^{-1}(\lambda)\mathcal{K}_{\mathbf{T}}(F)(\lambda)\phi(\lambda)\right]+\widetilde{H}^{-1}(\lambda)\frac{\partial}{\partial\lambda_{i}}\widetilde{H}(\lambda)\mathcal{K}_{\mathbf{T}}(G)(\lambda)-\mathcal{K}_{\mathbf{T}}(G)(\lambda)\widetilde{H}^{-1}(\lambda)\frac{\partial}{\partial\lambda_{i}}\widetilde{H}(\lambda)\\
&=\frac{\partial}{\partial\lambda_{i}}\phi^{-1}(\lambda)\mathcal{K}_{\mathbf{T}}(F)(\lambda)\phi(\lambda)+\phi^{-1}(\lambda)\frac{\partial}{\partial\lambda_{i}}\mathcal{K}_{T}(F)(\lambda)\phi(\lambda)+\phi^{-1}(\lambda)\mathcal{K}_{\mathbf{T}}(F)(\lambda)\frac{\partial}{\partial\lambda_{i}}\phi(\lambda)\\
&\quad+\left(\phi^{-1}(\lambda)H^{-1}(\lambda)\frac{\partial}{\partial\lambda_{i}}H(\lambda)\phi(\lambda)+\phi^{-1}(\lambda)\frac{\partial}{\partial\lambda_{i}}\phi(\lambda)\right)\left(\phi^{-1}(\lambda)\mathcal{K}_{\mathbf{T}}(F)(\lambda)\phi(\lambda)\right)\\
&\quad-\left(\phi^{-1}(\lambda)\mathcal{K}_{\mathbf{T}}(F)(\lambda)\phi(\lambda)\right)\left(\phi^{-1}(\lambda)H^{-1}(\lambda)\frac{\partial}{\partial\lambda_{i}}H(\lambda)\phi(\lambda)+\phi^{-1}(\lambda)\frac{\partial}{\partial\lambda_{i}}\phi(\lambda)\right)\\
&=\frac{\partial}{\partial\lambda_{i}}\phi^{-1}(\lambda)\mathcal{K}_{\mathbf{T}}(F)(\lambda)\phi(\lambda)+\phi^{-1}(\lambda)\frac{\partial}{\partial\lambda_{i}}\mathcal{K}_{\mathbf{T}}(F)(\lambda)\phi(\lambda)+\phi^{-1}(\lambda)\mathcal{K}(\gamma)_{\mathbf{T}}(F)\frac{\partial}{\partial\lambda_{i}}\phi(\lambda)\\
&\quad+\phi^{-1}(\lambda)H^{-1}(\lambda)\frac{\partial}{\partial\lambda_{i}}H(\lambda)\mathcal{K}_{\mathbf{T}}(F)(\lambda)\phi(\lambda)
+\phi^{-1}(\lambda)\frac{\partial}{\partial\lambda_{i}}\phi(\lambda)\phi^{-1}(\lambda)\mathcal{K}_{\mathbf{T}}(F)(\lambda)\phi(\lambda)\\
&\quad-\phi^{-1}(\lambda)\mathcal{K}_{\mathbf{T}}(F)(\lambda)H^{-1}(\lambda)\frac{\partial}{\partial\lambda_{i}}H(\lambda)\phi(\lambda)-\phi^{-1}(\lambda)\mathcal{K}_{\mathbf{T}}(F)(\lambda)\frac{\partial}{\partial\lambda_{i}}\phi(\lambda)\\
&=\phi^{-1}(\lambda)\frac{\partial}{\partial\lambda_{i}}\mathcal{K}_{\mathbf{T}}(F)(\lambda)\phi(\lambda)+\phi^{-1}(\lambda)H^{-1}(\lambda)\frac{\partial}{\partial\lambda_{i}}H(\lambda)\mathcal{K}_{\mathbf{T}}(F)(\lambda)\phi(\lambda)\\
&\quad-\phi^{-1}(\lambda)\mathcal{K}_{\mathbf{T}}(F)(\lambda)H^{-1}(\lambda)\frac{\partial}{\partial\lambda_{i}}H(\lambda)\phi(\lambda)\\
&=\phi^{-1}(\lambda)\mathcal{K}_{\mathbf{T},\lambda_{i}}(F)(\lambda)\phi(\lambda)
\end{align*}
and
\begin{align*}
&\quad \mathcal{K}_{\mathbf{T},\overline{\lambda}_{j}}(G)(\lambda)\\
&=\frac{\partial}{\partial\overline{\lambda}_{j}}\mathcal{K}_{\mathbf{T}}(G)(\lambda)\\
&=\frac{\partial}{\partial\overline{\lambda}_{j}}\left[\phi^{-1}(\lambda)\mathcal{K}_{\mathbf{T}}(\gamma)(F)\phi(\lambda)\right]\\
&=\phi^{-1}(\lambda)\mathcal{K}_{\mathbf{T},\overline{\lambda}_{j}}(F)(\lambda)\phi(\lambda).
\end{align*}
Without loss of generality, assume that
$$\mathcal{K}_{\mathbf{T},\lambda_{p}^{i_{p}},\overline{\lambda}^{j_{q}}_{q}}(G)(\lambda)
=\phi^{-1}(\lambda)\mathcal{K}_{\mathbf{T},\lambda_{p}^{i_{p}},
\overline{\lambda}^{j_{q}}_{q}}(F)(\lambda)\phi(\lambda),\quad 1\leq p,q\leq m$$
for some $i_{p}$ and $i_{q}$ in $\mathbb{N}.$ Then
\begin{align*}
&\quad\quad \mathcal{K}_{\mathbf{T}, \lambda_{p}^{i_{p}+1},\overline{\lambda}^{j_{q}}_{q}}(G)(\lambda)\\
&=\frac{\partial}{\partial\lambda_{p}}\mathcal{K}_{\mathbf{T},\lambda_{p}^{i_{p}},\overline{\lambda}^{j_{q}}_{q}}
(G)(\lambda)+\left[\widetilde{H}^{-1}(\lambda)\frac{\partial}{\partial\lambda_{p}}\widetilde{H}
(\lambda),\mathcal{K}_{\mathbf{T},\lambda_{p}^{i_{p}},\overline{\lambda}^{j_{q}}_{q}}
(G)(\lambda)\right]\\
&=\frac{\partial}{\partial\lambda_{p}}\phi^{-1}(\lambda)
\mathcal{K}_{\mathbf{T},\lambda_{p}^{i_{p}},\overline{\lambda}^{j_{q}}_{q}}
(F)(\lambda)\phi(\lambda)+\phi^{-1}(\lambda)\frac{\partial}{\partial\lambda_{p}}
\mathcal{K}_{\mathbf{T},\lambda_{p}^{i_{p}},\overline{\lambda}^{j_{q}}_{q}}(F)(\lambda)
\phi(\lambda)+\phi^{-1}(\lambda)\mathcal{K}_{\mathbf{T},\lambda_{p}^{i_{p}},
\overline{\lambda}^{j_{q}}_{q}}(F)(\lambda)
\frac{\partial}{\partial\lambda_{p}}\phi(\lambda)\\
&\quad+\left(\phi^{-1}(\lambda)H^{-1}(\lambda)\frac{\partial}{\partial\lambda_{p}}H(\lambda)
\phi(\lambda)+\phi^{-1}(\lambda)\frac{\partial}{\partial\lambda_{p}}\phi(\lambda)\right)
\left(\phi^{-1}(\lambda)\mathcal{K}_{\mathbf{T},\lambda_{p}^{i_{p}},
\overline{\lambda}^{j_{q}}_{q}}(F)(\lambda)\phi(\lambda)\right)\\
&\quad-\left(\phi^{-1}(\lambda)\mathcal{K}_{\mathbf{T},\lambda_{p}^{i_{p}},
\overline{\lambda}^{j_{q}}_{q}}(F)(\lambda)\phi(\lambda)\right)
\left(\phi^{-1}(\lambda)H^{-1}(\lambda)\frac{\partial}{\partial\lambda_{p}}H(\lambda)
\phi(\lambda)+\phi^{-1}(\lambda)\frac{\partial}{\partial\lambda_{p}}\phi(\lambda)\right)\\
&=\phi^{-1}(\lambda)\frac{\partial}{\partial\lambda_{p}}\mathcal{K}_{\mathbf{T},\lambda_{p}^{i_{p}},\overline{\lambda}^{j_{q}}_{q}}
(F)(\lambda)\phi(\lambda)+\phi^{-1}(\lambda)H^{-1}(\lambda)\frac{\partial}{\partial\lambda_{p}}
H(\lambda)\mathcal{K}_{\mathbf{T},\lambda_{p}^{i_{p}},\overline{\lambda}^{j_{q}}_{q}}(F)(\lambda)\phi(\lambda)\\
&\quad-\phi^{-1}(\lambda)\mathcal{K}_{\mathbf{T},\lambda_{p}^{i_{p}},\overline{\lambda}^{j_{q}}_{q}}
(F)(\lambda)H^{-1}(\lambda)\frac{\partial}{\partial\lambda_{p}}H(\lambda)\phi(\lambda)\\
&=\phi^{-1}(\lambda)\mathcal{K}_{\mathbf{T}, \lambda_{p}^{i_{p}+1},\overline{\lambda}^{j_{q}}_{q}}(F)(\lambda)\phi(\lambda)
\end{align*}
and
\begin{equation*}
\begin{aligned}
\mathcal{K}_{\mathbf{T}, \lambda_{p}^{i_{p}},\overline{\lambda}^{j_{q}+1}_{q}}(G)(\lambda)
&=\frac{\partial}{\partial\overline{\lambda}_{q}}\mathcal{K}_{\mathbf{T},\lambda_{p}^{i_{p}},\overline{\lambda}^{j_{q}}_{q}}
(G)(\lambda)\\
&=\frac{\partial}{\partial\overline{\lambda}_{q}}\left[\phi^{-1}(\lambda)
\mathcal{K}_{\mathbf{T},\lambda_{p}^{i_{p}},\overline{\lambda}^{j_{q}}_{q}}(F)(\lambda)\phi(\lambda)\right]\\
&=\phi^{-1}(\lambda)\frac{\partial}{\partial\overline{\lambda}_{q}}
\mathcal{K}_{\mathbf{T},\lambda_{p}^{i_{p}},\overline{\lambda}^{j_{q}}_{q}}(F)(\lambda)
\phi(\lambda)\\
&=\phi^{-1}(\lambda)\mathcal{K}_{\mathbf{T},\lambda_{p}^{i_{p}},\overline{\lambda}^{j_{q}+1}_{q}}
(F)(\lambda)\phi(\lambda).
\end{aligned}
\end{equation*}
Therefore, for any $i_{p}$ and $i_{q}$ in $\mathbb{N}$, there is $$\mathcal{K}_{\mathbf{T},\lambda_{p}^{i_{p}},\overline{\lambda}^{j_{q}}_{q}}(G)(\lambda)
=\phi^{-1}(\lambda)\mathcal{K}_{\mathbf{T},\lambda_{p}^{i_{p}},\overline{\lambda}^{j_{q}}_{q}}(F)(\lambda)\phi(\lambda),\quad 1\leq p, q\leq m.$$
By Lemma \ref{lm0}, we obtain that for any $I,J\in\mathbb{Z}_{+}^{m},$
$$\mathcal{K}_{\mathbf{T}, \lambda^{I}}(G)(\lambda)=\sum\limits_{i_{p}\neq0}\mathcal{K}_{\mathbf{T},\lambda_{p}^{i_{p}}}(G)(\lambda)
=\sum\limits_{i_{p}\neq0}\left[\phi^{-1}(\lambda)\mathcal{K}_{\mathbf{T},\lambda_{p}^{i_{p}}}(F)(\lambda)\phi(\lambda)\right]
=\phi^{-1}(\lambda)\mathcal{K}_{\mathbf{T},\lambda^{I}}(F)(\lambda)\phi(\lambda),$$
$$\mathcal{K}_{\mathbf{T}, \overline{\lambda}^{J}}(G)(\lambda)=\sum\limits_{j_{q}\neq0}
\mathcal{K}_{\mathbf{T},\overline{\lambda}_{q}^{j_{q}}}(G)(\lambda)
=\sum\limits_{j_{q}\neq0}\left[\phi^{-1}(\lambda)
\mathcal{K}_{\mathbf{T},\overline{\lambda}_{q}^{j_{q}}}(F)(\lambda)\phi(\lambda)\right]
=\phi^{-1}(\lambda)\mathcal{K}_{\mathbf{T},\overline{\lambda}^{J}}(F)(\lambda)\phi(\lambda),$$
and
$$\mathcal{K}_{\mathbf{T},\lambda^{I}, \overline{\lambda}^{J}}(G)(\lambda)=\sum\limits_{i_{p},j_{q}\neq0}
\mathcal{K}_{\mathbf{T},\lambda_{p}^{i_{p}},\overline{\lambda}_{q}^{j_{q}}}(G)(\lambda)
=\sum\limits_{i_{p},j_{q}\neq0}\left[\phi^{-1}(\lambda)
\mathcal{K}_{\mathbf{T},\lambda_{p}^{i_{p}},\overline{\lambda}_{q}^{j_{q}}}(F)(\lambda)
\phi(\lambda)\right]
=\phi^{-1}(\lambda)\mathcal{K}_{\mathbf{T},\lambda^{I},\overline{\lambda}^{J}}(F)(\lambda)
\phi(\lambda).$$
It follows that $\mathcal{K}_{\mathbf{T},\lambda^{I}, \overline{\lambda}^{J}}(G)(\lambda)$ is similar to
$\mathcal{K}_{\mathbf{T},\lambda^{I}, \overline{\lambda}^{J}}(F)(\lambda)$ and
$$\text{trace\,}\mathcal{K}_{\mathbf{T},\lambda^{I}, \overline{\lambda}^{J}}(G)(\lambda)=\text{trace\,}( \phi^{-1}(\lambda)\mathcal{K}_{\mathbf{T},\lambda^{I}, \overline{\lambda}^{J}}(F)(\lambda)\phi(\lambda))=\text{trace\,}\mathcal{K}_{\mathbf{T},\lambda^{I}, \overline{\lambda}^{J}}(F)(\lambda)$$
for all $I,J\in\mathbb{Z}_{+}^{m}$ and $\lambda\in\Omega$.
\end{proof}

The above results reveal the relationship between the curvatures $\mathscr{K}(\mathcal{I})$ and $\mathcal{K}_{\mathcal{I}}$ of the extended holomorphic curve $\mathcal{I}$, and the corresponding connection $\pmb{\Theta}$ and classical connection $\Theta$ also have a similar relationship.

\begin{proposition}
Let $\mathcal{I}\in\mathcal{I}_{n}(\Omega,\mathcal{U})$ and $\mathcal{I}(\lambda)=F(\lambda)( G^{*}(\lambda)F(\lambda))^{-1}G^{*}(\lambda).$
Then the connection $\pmb{\Theta}$ and the classical connection $\Theta$ of $\mathcal{I}$ satisfy
$$\pmb{\Theta}(\lambda)F(\lambda)=F(\lambda)\Theta(\lambda)\quad \text{and}\quad \pmb{\Theta}(\lambda)\mathcal{I}(\lambda)=\pmb{\Theta}(\lambda),\quad\lambda\in\Omega.$$
\end{proposition}
\begin{proof}
Since $(\partial\mathcal{I})\mathcal{I}=\partial\mathcal{I}$, we have that
$$
\pmb{\Theta}\mathcal{I}=\left[(\partial F)H^{-1}G^{*}-
\partial\mathcal{I}\right]\mathcal{I}
=(\partial F) H^{-1}G^{*}F H^{-1}G^{*}-(\partial\mathcal{I})\mathcal{I}
=(\partial F)H^{-1}G^{*}-
\partial\mathcal{I}=\pmb{\Theta}$$
and
$$
\pmb{\Theta}F=[(\partial F)H^{-1}G^{*}-\partial\mathcal{I}]F
=\partial F-[(\partial F)H^{-1}G^{*}+F(\partial H^{-1})G^{*}]F=-F\partial H^{-1}H=F\Theta.$$
\end{proof}

\subsection{\sf Similarity and unitary equivalent of extended holomorphic curvature}
\
\newline
\indent In this subsection, we give a series of properties of unitarily classification and similarity classification of extended holomorphic curves in the multivariate for $C^{*}$-algebras, and discuss the relationship between extended holomorphic curves in the multivariate for $C^{*}$-algebras and extended holomorphic curves in the Cowen-Douglas class.
\begin{thm}\label{lm1}
Let $\mathcal{I}_{1},\mathcal{I}_{2}\in\mathcal{I}_{n}(\Omega,\mathcal{U}).$ If $\mathcal{I}_{1}\sim_{u}\mathcal{I}_{2}$, then $$\mathscr{K}_{I,J}(\mathcal{I}_{1})(\lambda)\sim_{u}\mathscr{K}_{I,J}(\mathcal{I}_{2})(\lambda),
\quad\lambda\in\Omega,\,\,I,J\in\mathbb{Z}_{+}^{m}.$$
\end{thm}
In order to prove this theorem, we first prove the following lemmas, which reveal the essential and intrinsic structure of the curvature of the extended holomorphic curve and its covariant derivatives.

\begin{lemma}\label{le00}
Let $\mathcal{I}\in\mathcal{I}_{n}(\Omega,\mathcal{U})$. For any $I,J\in\mathbb{Z}_{+}^{m}$ and $1\leq i,j\leq m$, we have that
\begin{itemize}
  \item [(1)]$D^{I}\mathcal{I}=\sum\limits_{0\leq I_{1}\leq I}\frac{I!}{I_{1}!(I-I_{1})!}D^{I_{1}} F D^{I-I_{1}}H^{-1}G^{*},$
  \item[(2)]$\overline{D}^{e_{j}}D^{I}\mathcal{I}=
  D^{I}\mathcal{I}\overline{D}^{e_{j}}\mathcal{I}-
  \overline{D}^{e_{j}}\mathcal{I}D^{I}\mathcal{I}-\sum\limits_{p=1}^{m}\sum\limits_{e_{p}\leq I_{1}\leq I-e_{p}}\frac{I!}{I_{1}!(I-I_{1})!}D^{I-I_{1}}\mathcal{I}\overline{D}^{e_{j}}
  \mathcal{I}D^{I_{1}}\mathcal{I},$
  \item [(3)] $\overline{D}^{J}\mathcal{I}=\sum\limits_{0\leq J_{1}\leq J}\frac{J!}{J_{1}!(J-J_{1})!} F\overline{D}^{J_{1}}H^{-1}\overline{D}^{J-J_{1}}G^{*}$,\,\,\text{and}
  \item[(4)]$\overline{D}^{J}D^{e_{i}}\mathcal{I}=D^{e_{i}}\mathcal{I}\overline{D}^{J}\mathcal{I}
  -\overline{D}^{J}\mathcal{I}D^{e_{i}}\mathcal{I}-\sum\limits_{p=1}^{m}\sum\limits_{e_{p}\leq J_{1}\leq J-e_{p}}\frac{J!}{J_{1}!(J-J_{1})!}\overline{D}^{J-J_{1}}\mathcal{I}D^{e_{i}}\mathcal{I}
  \overline{D}^{J_{1}}\mathcal{I}$.
\end{itemize}
\end{lemma}
\begin{proof}
Letting $I:=e_{i}$ for $1\leq i\leq m$, then
$$D^{e_{i}}\mathcal{I}=D^{e_{i}}( F H^{-1}G^{*})=D^{e_{i}} F H^{-1}G^{*}+F D^{e_{i}}H^{-1}G^{*}.$$
This means that (1) holds for $I=e_{i}, 1\leq i\leq m$. Without losing generality, assume that  (1) holds for some $I\in\mathbb{Z}_{+}^{m}$.
For any $K,J\in\mathbb{Z}_{+}^{m}$ and $1\leq l\leq m$, there is
\begin{equation}\label{le003}
\begin{aligned}
\frac{(J+e_{l})!}{K![J-(K-e_{l})]!}
=\frac{J!}{(K-e_{l})![J-(K-e_{l})]!}+
\frac{J!}{K!(J-K)!},
\end{aligned}
\end{equation}
so
\begin{eqnarray*}\label{eq8}
D^{I+e_{l}}\mathcal{I}&=&D^{e_{l}}\left(\sum_{0\leq I_{1}\leq I}\frac{I!}{I_{1}!(I-I_{1})!}D^{I_{1}} F D^{I-I_{1}}H^{-1}G^{*}\right)\\ \nonumber
&=&\sum_{0\leq I_{1}\leq I}\frac{I!}{I_{1}!(I-I_{1})!}(D^{I_{1}+e_{l}} F D^{I-I_{1}}H^{-1}+D^{I_{1}} F D^{I-I_{1}+e_{l}}H^{-1})G^{*}\\ \nonumber
&=&(D^{e_{l}} F D^{I}H^{-1}+ F D^{I+e_{l}}H^{-1})G^{*}+\cdots\\ \nonumber
&&+\frac{I!}{(I'-e_{l})!(I-(I' - e_{l}))!}(D^{I'} F D^{I-I'+e_{l}}H^{-1}+D^{I'-e_{l}} F D^{I-I'+2e_{l}}H^{-1})G^{*}\\ \nonumber
&&+\frac{I!}{I'!(I-I')!}(D^{I'+e_{l}} F D^{I-I'}H^{-1}+D^{I'} F D^{I-I'+e_{l}}H^{-1})G^{*}\\ \nonumber
&&+\frac{I!}{(I'+e_{l})!(I-(I' + e_{l}))!}(D^{I'+2e_{l}} F D^{I-I'-e_{l}}H^{-1}+D^{I'+e_{l}} F D^{I-I'}H^{-1})G^{*}\\ \nonumber
&&\quad+\cdots+(D^{I+e_{l}} F H^{-1}+D^{I} F D^{e_{l}}H^{-1})G^{*}\\ \nonumber
&=&F D^{I+e_{l}}H^{-1}G^{*}+\left(\frac{I!}{(e_{l}-e_{l})!(I-(e_{l}-e_{l}))!}+
\frac{I!}{e_{l}!(I-e_{l})!}\right)D^{e_{l}} F D^{I}H^{-1}G^{*}+\cdots\\ \nonumber
&&+\left(\frac{I!}{(I'-e_{l})!(I-(I'-e_{l}))!}+\frac{I!}{I'!(I-I')!}\right)D^{I'} F D^{I-I'+e_{l}}H^{-1}G^{*}\\ \nonumber
&&+\left(\frac{I!}{I'!(I-I')!}+\frac{I!}{(I'+e_{l})!(I-(I' + e_{l}))!}\right)D^{I'+e_{l}} F D^{I-I'}H^{-1}G^{*}\\ \nonumber
&&+\cdots+D^{I+e_{l}} F H^{-1}G^{*}\\ \nonumber
&=&\sum_{0\leq I'\leq I+e_{l}}\frac{(I+e_{l})!}{I'!((I+e_{l})-I')!}D^{I'} F D^{(I+e_{l})-I'}H^{-1}G^{*}.
\nonumber
\end{eqnarray*}
It follows that (1) holds for any $I\in\mathbb{Z}_{+}^{m}$.
By $(\overline{D}^{e_{i}}\mathcal{I})\mathcal{I}=0,$ we have that
\begin{equation}\label{le005}
  \begin{aligned}
\overline{D}^{e_{j}}(D^{e_{i}}\mathcal{I})
=\overline{D}^{e_{j}}(D^{e_{i}}\mathcal{I}\mathcal{I})
=(\overline{D}^{e_{j}}D^{e_{i}}\mathcal{I})\mathcal{I}+(D^{e_{i}}\mathcal{I})(\overline{D}^{e_{j}}\mathcal{I})=(D^{e_{i}}\mathcal{I})(\overline{D}^{e_{j}}\mathcal{I})-(\overline{D}^{e_{j}}\mathcal{I})( D^{e_{i}}\mathcal{I}).
  \end{aligned}
\end{equation}
Thus, (2) is holds for $I=e_{i}$, where $1\leq i\leq m.$ Without loss of generality, assume that (2)
holds for some $I\in\mathbb{Z}_{+}^{m}$.
From (\ref{le003}), we obtain that for any $1\leq l,j,p\leq m, p\neq l,$
\begin{align}\label{eq9}
&D^{I}\mathcal{I}\overline{D}^{e_{j}}
\mathcal{I}D^{e_{l}}\mathcal{I}+D^{e_{l}}\mathcal{I}\overline{D}^{e_{j}}\mathcal{I}D^{I}
\mathcal{I}\,+\sum\limits_{e_{l}\leq I_{1}\leq I-e_{l}}\frac{I!}{I_{1}!(I-I_{1})!}(D^{I+e_{l}-I_{1}}\mathcal{I}
\overline{D}^{e_{j}}\mathcal{I}D^{I_{1}}\mathcal{I}+D^{I-I_{1}}\mathcal{I}\overline{D}^{e_{j}}
\mathcal{I}D^{I_{1}+e_{l}}\mathcal{I})\\ \nonumber
=&D^{I}\mathcal{I}\overline{D}^{e_{j}}
\mathcal{I}D^{e_{l}}\mathcal{I}
+\frac{I!}{e_{l}!(I-e_{l})!}(D^{I}\mathcal{I}\overline{D}^{e_{j}}
\mathcal{I}D^{e_{l}}\mathcal{I}
+D^{I-e_{l}}\mathcal{I}\overline{D}^{e_{j}}\mathcal{I}D^{2e_{l}}\mathcal{I})+\cdots\\ \nonumber
&+\frac{I!}{(I'-e_{k})!(I-(I'-e_{k}))!}(D^{I+e_{l}-I'+e_{k}}\mathcal{I}
\overline{D}^{e_{j}}\mathcal{I} D^{I'-e_{k}}\mathcal{I}+D^{I-I'+e_{k}}\mathcal{I}\overline{D}^{e_{j}}\mathcal{I} D^{I'-e_{k}+e_{l}}\mathcal{I})\\ \nonumber
&+\frac{I!}{(I'-e_{k}+e_{l})!(I-(I'-e_{k}+e_{l}))!}(D^{I-I'+e_{k}}\mathcal{I}
\overline{D}^{e_{j}}\mathcal{I} D^{I'-e_{k}+e_{l}}\mathcal{I}+D^{I-I'+e_{k}-e_{l}}\mathcal{I}\overline{D}^{e_{j}}\mathcal{I} D^{I'-e_{k}+2e_{l}}\mathcal{I})+\cdots\\ \nonumber
&+\frac{I!}{(I'-e_{l})!e_{l}!}(D^{2e_{l}}\mathcal{I}\overline{D}^{e_{j}}\mathcal{I} D^{I-e_{l}}\mathcal{I}+D^{e_{l}}\mathcal{I}\overline{D}^{e_{j}}\mathcal{I} D^{I}\mathcal{I})+D^{e_{l}}\mathcal{I}\overline{D}^{e_{j}}\mathcal{I}D^{I}
\mathcal{I}\\ \nonumber
=&\sum_{e_{l}\leq I'\leq I}\frac{(I+e_{l})!}{I'!(I+e_{l}-I')!}
(D^{I+e_{l}-I'}\mathcal{I}\overline{D}^{e_{j}}\mathcal{I}D^{I'}\mathcal{I}),
\nonumber
\end{align}
and
\begin{align}\label{eq10}
&\sum\limits_{e_{p}\leq I_{1}\leq I-e_{p}}\frac{I!}{I_{1}!(I-I_{1})!}(D^{I+e_{l}-I_{1}}\mathcal{I}
\overline{D}^{e_{j}}\mathcal{I}D^{I_{1}}\mathcal{I}+D^{I-I_{1}}\mathcal{I}\overline{D}^{e_{j}}
\mathcal{I}D^{I_{1}+e_{l}}\mathcal{I})\\ \nonumber
=&\frac{I!}{e_{p}!(I-e_{p})!}(D^{I+e_{l}-e_{p}}\mathcal{I}\overline{D}^{e_{j}}
\mathcal{I}D^{e_{p}}\mathcal{I}
+D^{I-e_{p}}\mathcal{I}\overline{D}^{e_{j}}\mathcal{I}D^{e_{p}+e_{l}}\mathcal{I})\\ \nonumber
&+\frac{I!}{(e_{p}+e_{l})!(I-(e_{p}+e_{l}))!}(D^{I-e_{p}}\mathcal{I}
\overline{D}^{e_{j}}\mathcal{I} D^{e_{p}+e_{l}}\mathcal{I}+D^{I-e_{p}-e_{l}}\mathcal{I}\overline{D}^{e_{j}}\mathcal{I} D^{e_{p}+2e_{l}}\mathcal{I})+\cdots\\ \nonumber
&+\frac{I!}{(I'-e_{k})!(I-(I'-e_{k}))!}(D^{I+e_{l}-I'+e_{k}}\mathcal{I}
\overline{D}^{e_{j}}\mathcal{I} D^{I'-e_{k}}\mathcal{I}+D^{I-I'+e_{k}}\mathcal{I}\overline{D}^{e_{j}}\mathcal{I} D^{I'-e_{k}+e_{l}}\mathcal{I})\\ \nonumber
&+\frac{I!}{(I'-e_{k}+e_{l})!(I-(I'-e_{k}+e_{l}))!}(D^{I-I'+e_{k}}\mathcal{I}
\overline{D}^{e_{j}}\mathcal{I} D^{I'-e_{k}+e_{l}}\mathcal{I}+D^{I-I'+e_{k}-e_{l}}\mathcal{I}\overline{D}^{e_{j}}\mathcal{I} D^{I'-e_{k}+2e_{l}}\mathcal{I})+\cdots\\ \nonumber
&+\frac{I!}{(I'-e_{p})!e_{p}!}(D^{e_{l}+e_{p}}\mathcal{I}\overline{D}^{e_{j}}\mathcal{I} D^{I-e_{p}}\mathcal{I}+D^{e_{p}}\mathcal{I}\overline{D}^{e_{j}}\mathcal{I} D^{I-e_{p}+e_{l}}\mathcal{I})\\ \nonumber
=&\sum_{e_{p}\leq I'\leq I+e_{l}-e_{p}}\frac{(I+e_{l})!}{I'!(I+e_{l}-I')!}
(D^{I+e_{l}-I'}\mathcal{I}\overline{D}^{e_{j}}\mathcal{I}D^{I'}\mathcal{I}).
\nonumber
\end{align}
Therefore, for $I+e_{l}\in\mathbb{Z}_{+}^{m}$,
by (\ref{le005})-(\ref{eq10}), we have that
\begin{equation*}
  \begin{aligned}
&\quad\,\,\overline{D}^{e_{j}}D^{I+e_{l}}\mathcal{I}\\
&=D^{e_{l}}(\overline{D}^{e_{j}}D^{I}\mathcal{I})\\
&=D^{e_{l}}\left(D^{I}\mathcal{I}\overline{D}^{e_{j}}\mathcal{I}-\overline{D}^{e_{j}}\mathcal{I}D^{I}\mathcal{I}
-\sum_{p=1}^{m}\sum_{e_{p}\leq I_{1}\leq I-e_{p}}\frac{I!}{I_{1}!(I-I_{1})!}D^{I-I_{1}}\mathcal{I}\overline{D}^{e_{j}}
\mathcal{I}D^{I_{1}}\mathcal{I}\right)\\
&=D^{I+e_{l}}\mathcal{I}\overline{D}^{e_{j}}\mathcal{I}-\overline{D}^{e_{j}}\mathcal{I} D^{I+e_{l}}\mathcal{I}-D^{I}\mathcal{I}\overline{D}^{e_{j}}
\mathcal{I}D^{e_{l}}\mathcal{I}-D^{e_{l}}\mathcal{I}\overline{D}^{e_{j}}\mathcal{I}D^{I}\mathcal{I}\\
&\quad-\sum_{p=1}^{m}\!\sum_{e_{p}\leq I_{1}\leq I-e_{p}}
\frac{I!}{I_{1}!(I-I_{1})!}(D^{I+e_{l}-I_{1}}\mathcal{I}\overline{D}^{e_{j}}\mathcal{I}D^{I_{1}}\mathcal{I}+D^{I-I_{1}}\mathcal{I}\overline{D}^{e_{j}}\mathcal{I}D^{I_{1}+e_{l}}\mathcal{I}).\\
&=D^{I+e_{l}}\mathcal{I}\overline{D}^{e_{j}}\mathcal{I}-\overline{D}^{e_{j}}\mathcal{I}D^{I+e_{l}}\mathcal{I}-
\sum_{p=1}^{m}\sum_{e_{p}\leq I'\leq I+e_{l}-e_{p}}\frac{(I+e_{l})!}{I'!(I+e_{l}-I')!}
(D^{I+e_{l}-I'}\mathcal{I}\overline{D}^{e_{j}}\mathcal{I}D^{I'}\mathcal{I})\\
\end{aligned}
\end{equation*}
The above completes the proof of (1) and (2). Similarly, we obtain proof (3) and (4) are also hold.
\end{proof}

Let $\Omega\subset\mathbb{C}^{m}$ be a bounded domain, and let $\mathcal{I}\in\mathcal{I}_{n}(\Omega,\mathcal{U})$ be an extended holomorphic curve on $C^{*}$-algebra $\mathcal{U}$. For any $\lambda_{0}\in\Omega$, there exists an open neighborhood $\Omega_{0}\subset\Omega$ of $\lambda_{0}$ such that
$$\mathcal{I}(\lambda)=\sum_{I,J\in\mathbb{Z}_{+}^{m}}
\frac{D^{I}\overline{D}^{J}\mathcal{I}(\lambda_{0})}{I!J!}
(\lambda-\lambda_{0})^{I}(\overline{\lambda}-\overline{\lambda}_{0})^{J},\quad \lambda\in\Omega_{0}.$$
This indicates that $\{D^{I}\overline{D}^{J}\mathcal{I}:I,J\in\mathbb{Z}_{+}^{m}\}$ is the unitary invariant of the extended holomorphic curve $\mathcal{I}$ in $\mathcal{I}_{n}(\Omega,\mathcal{U}).$
M. Martin and N. Salinas  provided in 2.10 \cite{Flag} that every derivative $D^{I}\overline{D}^{J}P(\lambda), I,J\in\mathbb{Z}_{+}^{m},$ of the holomorphic curve $P(\lambda)$ may be expressed as the sum of monomials of the form
$$\pm(\overline{D}^{J_{1}}P)(D^{I_{1}}P)\cdots(\overline{D}^{J_{t}}P)(D^{I_{t}}P)
\quad\text{and}\quad\pm(D^{I_{1}}P)(\overline{D}^{J_{1}}P)\cdots(D^{I_{t}}P)
(\overline{D}^{J_{t}}P).
$$ The following result shows that it is also held on extended holomorphic curves.

\begin{corollary}\label{cor2}
Let $\mathcal{I}\in\mathcal{I}_{n}(\Omega,\mathcal{U})$. Then the derivative $D^{I}\overline{D}^{J}\mathcal{I}(\lambda), I,J\in\mathbb{Z}_{+}^{m},$ may be expressed as the sum of monomials of the form
$$\pm(D^{I_{1}}\mathcal{I})(\overline{D}^{J_{1}}\mathcal{I})\cdots(D^{I_{t}}\mathcal{I})
(\overline{D}^{J_{t}}\mathcal{I}).$$
where $t\geq 1,$ $I_{1}+\cdots+I_{t}=I,J_{1}+\cdots+J_{t}=J$, $I_{1},J_{t}\in\mathbb{Z}_{+}^{m}$ and $\{I_{2},\ldots,I_{t},J_{1},\ldots,J_{t-1}\}\subseteq\mathbb{Z}_{+}^{m}\setminus\{0\}.$
\end{corollary}
\begin{proof}
From (1) and (3) of Lemma \ref{le00}, it is obtained that
$D^{I}\mathcal{I}\mathcal{I}=D^{I}\mathcal{I}$
and
$\mathcal{I}\overline{D}^{J}\mathcal{I}=\overline{D}^{J}\mathcal{I}$ for any $I,J\in\mathbb{Z}_{+}^{m}.$
Then we have that
$$(D^{I_{1}}\mathcal{I})(\overline{D}^{J_{1}}\mathcal{I})\cdots(D^{I_{t}}\mathcal{I})
(\overline{D}^{J_{t}}\mathcal{I})=(\overline{D}^{J_{1}}\mathcal{I})\cdots(D^{I_{t}}\mathcal{I})
(\overline{D}^{J_{t}}\mathcal{I}),\quad I_{1}=0,$$
$$(D^{I_{1}}\mathcal{I})(\overline{D}^{J_{1}}\mathcal{I})\cdots(D^{I_{t}}\mathcal{I})
(\overline{D}^{J_{t}}\mathcal{I})=(D^{I_{1}}\mathcal{I})
(\overline{D}^{J_{1}}\mathcal{I})\cdots(D^{I_{t}}\mathcal{I}),\quad J_{t}=0,$$
and
$$(D^{I_{1}}\mathcal{I})(\overline{D}^{J_{1}}\mathcal{I})\cdots(D^{I_{t}}\mathcal{I})
(\overline{D}^{J_{t}}\mathcal{I})=(\overline{D}^{J_{1}}\mathcal{I})\cdots(D^{I_{t}}\mathcal{I})
,\quad I_{1}=J_{t}=0.$$
From Lemma \ref{le00}, assume that $D^{I}\overline{D}^{J}\mathcal{I}(\lambda)$ may be expressed as the sum of monomials of the form
$$\pm(D^{I_{1}}\mathcal{I})(\overline{D}^{J_{1}}\mathcal{I})\cdots(D^{I_{t}}\mathcal{I})
(\overline{D}^{J_{t}}\mathcal{I}),$$
where $I_{1}+\cdots+I_{t}=I$ and $J_{1}+\cdots+J_{t}=J.$
Without losing generality, suppose that $$D^{I}\overline{D}^{J}\mathcal{I}=(D^{I_{1}}\mathcal{I})(\overline{D}^{J_{1}}\mathcal{I})
\cdots(D^{I_{t}}\mathcal{I})
(\overline{D}^{J_{t}}\mathcal{I}),$$
where $I_{1}+\cdots+I_{t}=I,J_{1}+\cdots+J_{t}=J$, $I_{1},J_{t}\in\mathbb{Z}_{+}^{m}$ and $\{I_{2},\ldots,I_{t},J_{1},\ldots,J_{t-1}\}\subseteq\mathbb{Z}_{+}^{m}\setminus\{0\}.$ By (2) and (4) of Lemma \ref{le00}, we obtain that
\begin{equation*}
 \begin{aligned}
&\quad D^{I+e_{l}}\overline{D}^{J}\mathcal{I}\\
&=D^{e_{l}}(D^{I}\overline{D}^{J}\mathcal{I})\\
&=\sum\limits_{r=1}^{t}\left[(D^{I_{1}}\mathcal{I})(\overline{D}^{J_{1}}\mathcal{I})
\cdots(D^{I_{r}+e_{l}}\mathcal{I})
(\overline{D}^{J_{r}}\mathcal{I})\cdots(D^{I_{t}}\mathcal{I})
(\overline{D}^{J_{t}}\mathcal{I})\right]\\
&\quad+\sum\limits_{r=1}^{t}\left[(D^{I_{1}}\mathcal{I})(\overline{D}^{J_{1}}\mathcal{I})
\cdots(D^{I_{r}}\mathcal{I})(D^{e_{l}}\overline{D}^{J_{r}}\mathcal{I})
\cdots(D^{I_{t}}\mathcal{I})(\overline{D}^{J_{t}}\mathcal{I})\right]\\
&=\sum\limits_{r=1}^{t}\left[(D^{I_{1}}\mathcal{I})(\overline{D}^{J_{1}}\mathcal{I})
\cdots(D^{I_{r}+e_{l}}\mathcal{I})
(\overline{D}^{J_{r}}\mathcal{I})\cdots(D^{I_{t}}\mathcal{I})
(\overline{D}^{J_{t}}\mathcal{I})\right]\\
&\quad-\sum\limits_{r=1}^{t}\sum\limits_{p=1}^{m}\sum\limits_{e_{p}\leq J'_{r}\leq J_{r}-e_{p}}\frac{J_{r}!}{J'_{r}!(J_{r}-J'_{r})!}\left[(D^{I_{1}}\mathcal{I})(\overline{D}^{J_{1}}\mathcal{I})
\cdots(D^{I_{r}}\mathcal{I})(
\overline{D}^{J_{r}-J'_{r}}\mathcal{I}D^{e_{l}}\mathcal{I}
  \overline{D}^{J'_{r}}\mathcal{I}
)
\cdots(D^{I_{t}}\mathcal{I})(\overline{D}^{J_{t}}\mathcal{I})\right]\\
&\quad -\left[(D^{I_{1}}\mathcal{I})(\overline{D}^{J_{1}}\mathcal{I})
\cdots(D^{I_{r}}\mathcal{I})(\overline{D}^{J_{r}}\mathcal{I})
\cdots(D^{I_{t}}\mathcal{I})(\overline{D}^{J_{t}}\mathcal{I})(D^{e_{l}}\mathcal{I})\right]
\end{aligned}
\end{equation*}
for $J_{t}\neq 0$, otherwise,
\begin{equation*}
 \begin{aligned}
&\quad D^{I+e_{l}}\overline{D}^{J}\mathcal{I}\\
&=\sum\limits_{r=1}^{t}\left[(D^{I_{1}}\mathcal{I})(\overline{D}^{J_{1}}\mathcal{I})
\cdots(D^{I_{r}+e_{l}}\mathcal{I})
(\overline{D}^{J_{r}}\mathcal{I})\cdots(D^{I_{t}}\mathcal{I})\right]\\
&\quad-\sum\limits_{r=1}^{t-1}\sum\limits_{p=1}^{m}\sum\limits_{e_{p}\leq J'_{r}\leq J_{r}-e_{p}}\frac{J_{r}!}{J'_{r}!(J_{r}-J'_{r})!}\left[(D^{I_{1}}\mathcal{I})(\overline{D}^{J_{1}}\mathcal{I})
\cdots(D^{I_{r}}\mathcal{I})(
\overline{D}^{J_{r}-J'_{r}}\mathcal{I}D^{e_{l}}\mathcal{I}
  \overline{D}^{J'_{r}}\mathcal{I}
)
\cdots(D^{I_{t}}\mathcal{I})\right]
\end{aligned}
\end{equation*} for $J_{t}=0$. This shows that the corollary holds for $D^{I+e_{l}}\overline{D}^{J}\mathcal{I}, 1\leq l\leq m$, similarly, we can prove that it holds for $D^{I}\overline{D}^{J+e_{k}}\mathcal{I}, 1\leq k\leq m$.
\end{proof}

\begin{lemma}\label{lm4}
Let $\mathcal{I}\in\mathcal{I}_{n}(\Omega,\mathcal{U})$. Then the covariant derivative $\mathscr{K}_{I,J}(\mathcal{I}), I,J\in\mathbb{Z}_{+}^{m},$ may be expressed as the sum of monomials of the form
$$\pm(\overline{D}^{J_{1}}\mathcal{I}D^{I_{1}}\mathcal{I})^{l_{1}}
(\overline{D}^{J_{2}}\mathcal{I}D^{I_{2}}\mathcal{I})^{l_{2}}...
(\overline{D}^{J_{t}}\mathcal{I}D^{I_{t}}\mathcal{I})^{l_{t}},$$
where $I_{r}=s_{r,p}e_{p}+m_{r,l}e_{l},$ $J_{r}=s_{r,q}e_{q}+m_{r,k}e_{k},$ $ s_{r,p},s_{r,q}\geq 0,$ $m_{r,l},m_{r,k}\in\{0,1\}$ for $I_{r},J_{r}\in\mathbb{Z}_{+}^{m}\setminus\{0\}$, $l_{1}|I_{1}|+\cdots+l_{t}|I_{t}|=|I|+1$, $l_{1}|J_{1}|+\cdots+l_{t}|J_{t}|=|J|+1$ and $1\leq p,q, l,k\leq m.$
\end{lemma}
\begin{proof}
From Lemma \ref{D3}, we obtain that each covariant derivative $\mathscr{K}_{I,J}(\mathcal{I}), I,J\in\mathbb{Z}_{+}^{m},$ may be expressed as the sum of the form $\mathscr{K}_{i_{p}e_{p},j_{q}e_{q}}(\mathcal{I}), i_{p},j_{q}\geq 0, 1\leq p,q\leq m$.
So we only need to prove that for the covariant derivative $\mathscr{K}_{i_{p}e_{p},j_{q}e_{q}}(\mathcal{I})$, $\mathscr{K}_{i_{p}e_{p},j_{q}e_{q}}(\mathcal{I})$ can be expressed as the sum of monomials of the form
$$\pm(\overline{D}^{J_{1}}\mathcal{I}D^{I_{1}}\mathcal{I})^{l_{1}}
(\overline{D}^{J_{2}}\mathcal{I}D^{I_{2}}\mathcal{I})^{l_{2}}...
(\overline{D}^{J_{t}}\mathcal{I}D^{I_{t}}\mathcal{I})^{l_{t}},$$
where $I_{r}=s_{r,p}e_{p}+m_{r,l}e_{l},$ $J_{r}=s_{r,q}e_{q}+m_{r,k}e_{k},$ $0\leq s_{r,p}\leq i_{p},$ $0\leq s_{r,q}\leq j_{q},$ $m_{r,l},m_{r,k}\in\{0,1\}$ for $I_{r},J_{r}\in\mathbb{Z}_{+}^{m}\setminus\{0\}$, $l_{1}|I_{1}|+\cdots+l_{t}|I_{t}|=i_{p}+1$, $l_{1}|J_{1}|+\cdots+l_{t}|J_{t}|=j_{q}+1$ and $1\leq l,k\leq m.$
Since $\partial_{l}\mathcal{I}\mathcal{I}=\partial_{l}\mathcal{I}$, $\mathcal{I}\partial_{l}\mathcal{I}=0$ and $\overline{\partial}_{k}\mathcal{I}=\mathcal{I}\overline{\partial}_{k}\mathcal{I}$ for any $1\leq l,k\leq m$, we have that
$$\mathcal{I}\overline{\partial}_{k}\partial_{l}\mathcal{I}=
-\overline{\partial}_{k}\mathcal{I}\partial_{l}\mathcal{I}\quad\text{
and}\quad \overline{\partial}_{k}\partial_{l}\mathcal{I}=
\partial_{l}\mathcal{I}\overline{\partial}_{k}\mathcal{I}
-\overline{\partial}_{k}\mathcal{I}\partial_{l}\mathcal{I}.$$
Thus,
by Theorem \ref{thm1}, we have that
$\mathscr{K}(\mathcal{I})=
\overline{\partial}\mathcal{I}\partial\mathcal{I}=
\sum\limits_{1\leq l,k\leq m}\left(\overline{D}^{e_{k}}\mathcal{I}D^{e_{l}}\mathcal{I}\right)$,
$$\mathscr{K}_{e_{p},0}(\mathcal{I})
=\mathcal{I}(\partial_{p}\mathscr{K}(\mathcal{I}))
=\mathcal{I}\partial_{p}\overline{\partial}\mathcal{I}\partial\mathcal{I}+
\mathcal{I}\overline{\partial}\mathcal{I}\partial_{p}\partial\mathcal{I}
=\overline{\partial}\mathcal{I}\partial_{p}\partial\mathcal{I}=
\sum\limits_{1\leq l,k\leq m}\left(\overline{D}^{e_{k}}\mathcal{I}D^{e_{p}+e_{l}}\mathcal{I}\right),$$
$$\mathscr{K}_{0,e_{q}}(\mathcal{I})
=(\overline{\partial}_{q}\mathscr{K}(\mathcal{I}))\mathcal{I}
=\overline{\partial}_{q}\overline{\partial}\mathcal{I}\partial\mathcal{I}\mathcal{I}+
\overline{\partial}\mathcal{I}\overline{\partial}_{q}\partial\mathcal{I}\mathcal{I}=
\overline{\partial}_{q}\overline{\partial}\mathcal{I}\partial\mathcal{I}=
\sum\limits_{1\leq l,k\leq m}(\overline{D}^{e_{q}+e_{k}}\mathcal{I}D^{e_{l}}\mathcal{I}),$$
\begin{equation*}
 \begin{aligned}
\mathscr{K}_{0+e_{p},e_{q}}(\mathcal{I})&
=\mathcal{I}(\partial_{p}(\mathscr{K}_{0,e_{q}}(\mathcal{I})))\\
&=\overline{\partial}_{q}\overline{\partial}\mathcal{I}\partial_{p}\partial\mathcal{I}
-\overline{\partial}\mathcal{I}\partial_{p}\mathcal{I}\overline{\partial}_{q}\mathcal{I}
\partial\mathcal{I}-\overline{\partial}_{q}\mathcal{I}\partial_{p}\mathcal{I}
\overline{\partial}\mathcal{I}\partial\mathcal{I}\\
&=\sum_{1\leq l,k\leq m}\left(\overline{D}^{e_{q}+e_{k}}\mathcal{I} D^{e_{p}+e_{l}}\mathcal{I}-\overline{D}^{e_{k}}\mathcal{I}D^{e_{p}}\overline{D}^{e_{q}}\mathcal{I} D^{e_{l}}\mathcal{I}-\overline{D}^{e_{q}}\mathcal{I} D^{e_{p}}\mathcal{I}\overline{D}^{e_{k}}\mathcal{I} D^{e_{l}}\mathcal{I}\right),
\end{aligned}
\end{equation*}
and
\begin{equation*}
 \begin{aligned}
\mathscr{K}_{e_{p},0+e_{q}}(\mathcal{I})&=(\overline{\partial}_{q}(\mathscr{K}_{e_{p},0}
(\mathcal{I})))\mathcal{I}\\
&=\overline{\partial}_{q}\overline{\partial}\mathcal{I}
\partial_{p}\partial\mathcal{I}-\overline{\partial}\mathcal{I}\partial_{p}\mathcal{I}
\overline{\partial}_{q}\mathcal{I}\partial\mathcal{I}-\overline{\partial}\mathcal{I}
\partial\mathcal{I}\overline{\partial}_{q}\mathcal{I}\partial_{p}\mathcal{I}\\
&=\sum_{1\leq l,k\leq m}\left(\overline{D}^{e_{q}+e_{k}}\mathcal{I} D^{e_{p}+e_{l}}\mathcal{I}-\overline{D}^{e_{k}}\mathcal{I}D^{e_{p}}\overline{D}^{e_{q}}
\mathcal{I} D^{e_{l}}\mathcal{I}-\overline{D}^{e_{k}}\mathcal{I} D^{e_{l}}\mathcal{I}\overline{D}^{e_{q}}\mathcal{I} D^{e_{p}}\mathcal{I}\right).\\
\end{aligned}
\end{equation*}
It follows that this lemma holds for any $0\leq i_{p},j_{q}\leq1$.
Assume that this lemma holds for some $i_{p},j_{q}\in\mathbb{Z}$.
Now we just need to prove that it also holds for $\mathscr{K}_{i_{p}e_{p}+e_{p},j_{q}e_{q}}(\mathcal{I})$ and $\mathscr{K}_{i_{p}e_{p},j_{q}e_{q}+e_{q}}(\mathcal{I})$.
Without loss of generality, suppose that
$$\mathscr{K}_{i_{p}e_{p},j_{q}e_{q}}(\mathcal{I})
:=(\overline{D}^{J_{1}}\mathcal{I}D^{I_{1}}\mathcal{I})^{l_{1}},$$ where
$I_{1}=s_{1,p}e_{p}+m_{1,l}e_{l},$ $J_{1}=s_{1,q}e_{q}+m_{1,k}e_{k},$ $0\leq s_{1,p}\leq i_{p},$ $0\leq s_{1,q}\leq j_{q},$ $m_{1,l},m_{1,k}\in\{0,1\}$, $l_{1}>0$ and
$l_{1}|I_{1}|=i_{p}+1, l_{1}|J_{1}|=j_{q}+1$.
Then from Lemma \ref{le00}, $\mathcal{I}D^{e_{p}}\mathcal{I}=0,$ $D^{e_{p}}\mathcal{I}D^{I_{1}}\mathcal{I}=D^{I_{1}}\mathcal{I}D^{e_{p}}\mathcal{I}=0,$ $\mathcal{I}\overline{D}^{J_{1}}\mathcal{I}=\overline{D}^{J_{1}}\mathcal{I},$
and $D^{I_{1}}\mathcal{I}\mathcal{I}=D^{I_{1}}\mathcal{I},$
 we know that
\begin{equation}
  \begin{aligned}
&\quad \mathscr{K}_{i_{p}e_{p}+e_{p},j_{q}e_{q}}(\mathcal{I})\\
&=\mathcal{I}(\partial_{p}(\mathscr{K}_{i_{p}e_{p},j_{q}e_{q}}(\mathcal{I})))\\
&=\mathcal{I}(D^{e_{p}}(\overline{D}^{J_{1}}\mathcal{I}D^{I_{1}}\mathcal{I})^{l_{1}})\\
&=\mathcal{I}\sum\limits_{r=0}^{l_{1}-1}\left[(\overline{D}^{J_{1}}\mathcal{I}D^{I_{1}}\mathcal{I})^{r}
(D^{e_{p}}\overline{D}^{J_{1}}\mathcal{I}D^{I_{1}}\mathcal{I}
+\overline{D}^{J_{1}}\mathcal{I}D^{I_{1}+e_{p}}\mathcal{I})
(\overline{D}^{J_{1}}\mathcal{I}D^{I_{1}}\mathcal{I})^{l_{1}-1-r}\right]\\
&=\mathcal{I}\left(D^{e_{p}}\mathcal{I}\overline{D}^{J_{1}}\mathcal{I}-
\overline{D}^{J_{1}}\mathcal{I}D^{e_{p}}\mathcal{I}-\sum_{e_{q}\leq J_{1}^{'}\leq J_{1}-e_{q}}\frac{J_{1}!}{J_{1}^{'}!(J_{1}-J_{1}^{'})!}
(\overline{D}^{J_{1}-J_{1}^{'}}\mathcal{I}D^{e_{p}}\mathcal{I}\overline{D}^{J_{1}^{'}}\mathcal{I})
\right)D^{I_{1}}\mathcal{I}(\overline{D}^{J_{1}}\mathcal{I}D^{I_{1}}\mathcal{I})^{l_{1}-1}\\
&\quad+\overline{D}^{J_{1}}\mathcal{I}D^{I_{1}+e_{p}}\mathcal{I}
(\overline{D}^{J_{1}}\mathcal{I}D^{I_{1}}\mathcal{I})^{l_{1}-1}+\mathcal{I}\sum_{r=1}^{l_{1}-1}\left[(\overline{D}^{J_{1}}\mathcal{I}D^{I_{1}}\mathcal{I})^{r}
\left(D^{e_{p}}\mathcal{I}\overline{D}^{J_{1}}\mathcal{I}-
\overline{D}^{J_{1}}\mathcal{I}D^{e_{p}}\mathcal{I}\right.\right.\\
&\quad-\sum_{e_{q}\leq J_{1}^{'}\leq J_{1}-e_{q}}\frac{J_{1}!}{J_{1}'!(J_{1}-J_{1}^{'})!}
(\overline{D}^{J_{1}-J_{1}'}\mathcal{I}D^{e_{p}}\mathcal{I}\overline{D}^{J_{1}'}\mathcal{I})
\Big)D^{I_{1}}\mathcal{I}(\overline{D}^{J_{1}}\mathcal{I}D^{I_{1}}\mathcal{I})^{l_{1}-r-1}\\
&\quad+\left.(\overline{D}^{J_{1}}\mathcal{I}D^{I_{1}}\mathcal{I})^{r}(\overline{D}^{J_{1}}
\mathcal{I}D^{I_{1}+e_{p}}\mathcal{I})(\overline{D}^{J_{1}}\mathcal{I}D^{I_{1}}\mathcal{I})
^{l_{1}-r-1}\right]
\nonumber
  \end{aligned}
\end{equation}
\begin{equation}
  \begin{aligned}
&=\sum\limits_{r=0}^{l_{1}-1}\left[(\overline{D}^{J_{1}}\mathcal{I}D^{I_{1}}\mathcal{I})^{r}
(\overline{D}^{J_{1}}\mathcal{I}D^{I_{1}+e_{p}}\mathcal{I})
(\overline{D}^{J_{1}}\mathcal{I}D^{I_{1}}\mathcal{I})^{l_{1}-1-r}\right] \\
&\quad-\sum\limits_{r=0}^{l_{1}-1}
\left[(\overline{D}^{J_{1}}\mathcal{I}D^{I_{1}}\mathcal{I})^{r}
\sum_{e_{q}\leq J_{1}^{'}\leq J_{1}-e_{q}}\frac{J_{1}!}{J_{1}'!(J_{1}-J_{1}^{'})!}
(\overline{D}^{J_{1}-J_{1}'}\mathcal{I}D^{e_{p}}\mathcal{I}\overline{D}^{J_{1}'}\mathcal{I}
D^{I_{1}}\mathcal{I})(\overline{D}^{J_{1}}\mathcal{I}D^{I_{1}}\mathcal{I})^{l_{1}-1-r}\right].
\nonumber
  \end{aligned}
\end{equation}
It follows that this lemma holds for $\mathscr{K}_{i_{p}e_{p}+e_{p},j_{k}e_{k}}(\mathcal{I})$. Similarly, we also prove that this lemma holds for $\mathscr{K}_{i_{p}e_{p},j_{k}e_{k}+e_{k}}(\mathcal{I})$. This completes the proof.
\end{proof}

{\bf Proof of the Theorem \ref{lm1}.}
Since $\mathcal{I}_{1}\sim_{u}\mathcal{I}_{2}$, there is an unitary $U\in\mathcal{U}$ such that $$U\mathcal{I}_{1}(\lambda)=\mathcal{I}_{2}(\lambda)U,\quad \lambda\in\Omega.$$
Thus,
$$UD^{i_{l}e_{l}+k_{p}e_{p}}\mathcal{I}_{1}(\lambda)=D^{i_{l}e_{l}+k_{p}e_{p}}\mathcal{I}_{2}(\lambda)U
\quad\textit{and}\quad
U\overline{D}^{j_{k}e_{k}+m_{q}e_{q}}\mathcal{I}_{1}(\lambda)=\overline{D}^{j_{k}e_{k}+m_{q}e_{q}}\mathcal{I}_{2}(\lambda)U$$
for any $i_{l},k_{p},j_{k},m_{q}\geq 0.$ This means that
$$U\overline{D}^{j_{k}e_{k}+m_{q}e_{q}}\mathcal{I}_{1}(\lambda)D^{i_{l}e_{l}+k_{p}e_{p}}\mathcal{I}_{1}(\lambda)=
\overline{D}^{j_{k}e_{k}+m_{q}e_{q}}\mathcal{I}_{1}(\lambda)D^{i_{l}e_{l}+k_{p}e_{p}}\mathcal{I}_{1}(\lambda)U.$$
By Lemma \ref{lm4}, the covariant derivative $\mathscr{K}_{I,J}(\mathcal{I}), I,J\in\mathbb{Z}_{+}^{m},$ may be expressed by as a sum of monomials of form
$$\pm(\overline{D}^{J_{1}}\mathcal{I}D^{I_{1}}\mathcal{I})^{l_{1}}
(\overline{D}^{J_{2}}\mathcal{I}D^{I_{2}}\mathcal{I})^{l_{2}}...
(\overline{D}^{J_{t}}\mathcal{I}D^{I_{t}}\mathcal{I})^{l_{t}},$$
where $I_{r}=s_{r,l}e_{l}+m_{r,p}e_{p},$ $J_{r}=s_{r,j}e_{j}+m_{r,q}e_{q}.$
Then $U\mathscr{K}_{I,J}(\mathcal{I}_{1})(\lambda)U^{*}=\mathscr{K}_{I,J}(\mathcal{I}_{2})(\lambda)$ for any $\lambda\in\Omega$.

From Theorem \ref{thm1}, we obtain the relationship between the curvature and the classical curvature of  extended holomorphic curves.
The unitary equivalence of extended holomorphic curves from Theorem \ref{lm1} implies the unitary equivalence of their curvatures, and now we give the relationship between the classical curvatures in this case.

\begin{proposition}
Let $\mathcal{I}_{1},~\mathcal{I}_{2}\in\mathcal{I}_{n}(\Omega,\mathcal{U})$, if $\mathcal{I}_{1}\sim_{u}\mathcal{I}_{2},$ then exists an invertible operator $Y_{\lambda}$ such that $$Y_{\lambda}\mathcal{K}_{\mathcal{I}_{1},\lambda^{I},\overline{\lambda}^{J}}(\lambda)=
\mathcal{K}_{\mathcal{I}_{2},\lambda^{I},\overline{\lambda}^{J}}(\lambda)Y_{\lambda},\quad \lambda\in\Omega.$$
\end{proposition}
\begin{proof}
If $\mathcal{I}_{1}\sim_{u}\mathcal{I}_{2},$ there is a unitary operator $U\in\mathcal{U}$ such that $U\mathcal{I}_{1}(\lambda)=\mathcal{I}_{2}(\lambda)U$ for any $\lambda\in\Omega.$
By Theorem \ref{lm1}, we obtain that $U\mathscr{K}_{I,J}(\mathcal{I}_{1})(\lambda)=\mathscr{K}_{I,J}(\mathcal{I}_{2})(\lambda)U$ for any $I,J\in\mathbb{Z}_{+}^{m}$.
Setting
$$\mathcal{I}_{1}(\lambda)= F_{1}(\lambda)H_{1}^{-1}(\lambda)G_{1}^{*}(\lambda) \quad\textup{ and }\quad \mathcal{I}_{2}(\lambda)= F_{2}(\lambda)H_{2}^{-1}(\lambda)G_{2}^{*}(\lambda),$$ where $H_{i}(\lambda)=\langle F_{i}(\lambda),G_{i}(\lambda)\rangle=G_{i}^{*}(\lambda) F_{i}(\lambda)$ for $i=1,2.$
From Theorem \ref{thm1}, for any $\lambda\in\Omega$ and $I,J\in\mathbb{Z}_{+}^{m}$, we have that
$$U\left(F_{1}(\lambda)\mathcal{K}_{\mathcal{I}_{1},\lambda^{I},\overline{\lambda}^{J}}
(\lambda)H_{1}^{-1}(\lambda)G_{1}^{*}(\lambda)\right)=
\left(F_{2}(\lambda)\mathcal{K}_{\mathcal{I}_{2},\lambda^{I},\overline{\lambda}^{J}}(\lambda)
H_{2}^{-1}(\lambda)G_{2}^{*}(\lambda)\right)U,$$
that is
$$H_{2}^{-1}(\lambda)G_{2}^{*}(\lambda)U F_{1}(\lambda)\mathcal{K}_{\mathcal{I}_{1},\lambda^{I},\overline{\lambda}^{J}}(\lambda)
=\mathcal{K}_{\mathcal{I}_{2},\lambda^{I},\overline{\lambda}^{J}}(\lambda)H_{2}^{-1}(\lambda)
G_{2}^{*}(\lambda)U F_{1}(\lambda).$$
Letting $Y_{\lambda}:=H_{2}^{-1}(\lambda)G_{2}^{*}(\lambda)U F_{1}(\lambda)$ and $Z_{\lambda}:=H_{1}^{-1}(\lambda)G_{1}^{*}(\lambda)U^{*} F_{2}(\lambda),$
then
\begin{equation}
\begin{aligned}
Y_{\lambda}Z_{\lambda}
&=H_{2}^{-1}(\lambda)G_{2}^{*}(\lambda)U F_{1}(\lambda)H_{1}^{-1}(\lambda)G_{1}^{*}(\lambda)U^{*} F_{2}(\lambda)\\
&=H_{2}^{-1}(\lambda)G_{2}^{*}(\lambda)U\mathcal{I}_{1}(\lambda)U^{*} F_{2}(\lambda)\\
&=H_{2}^{-1}(\lambda)G_{2}^{*}(\lambda)\mathcal{I}_{2}(\lambda) F_{2}(\lambda)\\
&=H_{2}^{-1}(\lambda)G_{2}^{*}(\lambda) F_{2}(\lambda)H_{2}^{-1}(\lambda)G_{2}^{*}(\lambda) F_{2}(\lambda)\\
&=H_{2}^{-1}(\lambda)H_{2}(\lambda)H_{2}^{-1}(\lambda)H_{2}(\lambda)\\
&=I_{n}.
\nonumber
\end{aligned}
\end{equation}
Similarly, we have that $Z_{\lambda}Y_{\lambda}=H_{1}^{-1}(\lambda)G_{1}^{*}(\lambda)U^{*} F_{2}(\lambda)H_{2}^{-1}(\lambda)G_{2}^{*}(\lambda)U F_{1}(\lambda)=I_{n}.$ Hence $$Y_{\lambda}\mathcal{K}_{\mathcal{I}_{1},\lambda^{I},\overline{\lambda}^{J}}(\lambda)= \mathcal{K}_{\mathcal{I}_{2},\lambda^{I},\overline{\lambda}^{J}}(\lambda)Y_{\lambda},\quad \lambda\in\Omega$$ for the invertible operator $Y_{\lambda}=H_{2}^{-1}(\lambda)G_{2}^{*}(\lambda)U F_{1}(\lambda).$
\end{proof}

In the last part of this subsection we consider the similar properties of the extended holomorphic curve and show some interesting applications to the Cowen-Douglas class.

\begin{lemma}\label{lm3}
Let $\mathcal{I}_{1},\mathcal{I}_{2}\in\mathcal{I}_{n}(\Omega,\mathcal{U})$. Then the following statements are equivalent:
\begin{itemize}
  \item [(1)]$\mathcal{I}_{1}\sim_{s}\mathcal{I}_{2}$.
  \item[(2)]There is an invertible operator $X\in\mathcal{U}$ such that
$$XD^{I}\mathcal{I}_{1}(\lambda)\overline{D}^{J}\mathcal{I}_{1}(\lambda)=
D^{I}\mathcal{I}_{2}(\lambda)\overline{D}^{J}\mathcal{I}_{2}(\lambda)X,\quad \lambda\in\Omega, \, I,J\in\mathbb{Z}_{+}^{m}.$$
\end{itemize}
\end{lemma}
\begin{proof}
On the one hand, if
$\mathcal{I}_{1}\sim_{s}\mathcal{I}_{2}$, there is an invertible operator $X\in\mathcal{U}$ such that $X\mathcal{I}_{1}(\lambda)=\mathcal{I}_{2}(\lambda)X$ for any $\lambda\in\Omega$.
We have that
$$XD^{I}\mathcal{I}_{1}(\lambda)=D^{I}\mathcal{I}_{2}(\lambda)X\quad\textup{and}\quad X\overline{D}^{J}\mathcal{I}_{1}(\lambda)=\overline{D}^{J}\mathcal{I}_{2}(\lambda)X,
\quad\lambda\in\Omega,I,J\in\mathbb{Z}_{+}^{m}.$$
It follows that
$$X\overline{D}^{J}\mathcal{I}_{1}(\lambda)D^{I}\mathcal{I}_{1}(\lambda)=\overline{D}^{J}\mathcal{I}_{2}(\lambda)D^{I}\mathcal{I}_{2}(\lambda)X,\quad \lambda\in\Omega,\,I,J\in\mathbb{Z}_{+}^{m}.$$

On the other hand,
if
$$X\overline{D}^{J}\mathcal{I}_{1}(\lambda)D^{I}\mathcal{I}_{1}(\lambda)=\overline{D}^{J}\mathcal{I}_{2}(\lambda)D^{I}\mathcal{I}_{2}(\lambda)X,\quad \lambda\in\Omega,\,I,J\in\mathbb{Z}_{+}^{m}$$
for some invertible operator $X\in\mathcal{U}$.
By Corollary \ref{cor2}, we know that for any $I,J\in\mathbb{Z}_{+}^{m},$ $D^{I}\overline{D}^{J}\mathcal{I}_{k}$ may be expressed by as a sum of monomials of the form $$\pm(D^{I_{1}}\mathcal{I}_{k})(\overline{D}^{J_{1}}\mathcal{I}_{k})\cdots(D^{I_{t}}\mathcal{I}_{k})
(\overline{D}^{J_{t}}\mathcal{I}_{k}),$$
where $k=1,2$, $I_{1}+\cdots+I_{t}=I,~J_{1}+\cdots+J_{t}=J.$ Thus,
$$X\overline{D}^{J}D^{I}\mathcal{I}_{1}(\lambda)=
\overline{D}^{J}D^{I}\mathcal{I}_{2}(\lambda)X,
\quad \lambda\in\Omega,\,I,J\in\mathbb{Z}_{+}^{m}.$$
Note that for any $\lambda_{0}\in\Omega$, there exists an open neighborhood $\Omega_{0}\subset\Omega$ of $\lambda_{0}$ such that
\begin{equation*}
\begin{aligned}
&\mathcal{I}_{1}(\lambda)=\sum_{I,J\in\mathbb{Z}_{+}^{m}}\frac{D^{I}\overline{D}^{J}
\mathcal{I}_{1}(\lambda_{0})}{I!J!}(\lambda-\lambda_{0})^{I}
(\overline{\lambda}-\overline{\lambda}_{0})^{J},\\ &\mathcal{I}_{2}(\lambda)=\sum_{I,J\in\mathbb{Z}_{+}^{m}}
\frac{D^{I}\overline{D}^{J}\mathcal{I}_{2}(\lambda_{0})}{I!J!}
(\lambda-\lambda_{0})^{I}(\overline{\lambda}-\overline{\lambda}_{0})^{J},
\quad\lambda\in\Omega_{0},~I,J\in\mathbb{Z}_{+}^{m}.
\end{aligned}
\end{equation*}
Then we have that $X\mathcal{I}_{1}(\lambda)=\mathcal{I}_{2}(\lambda)X$ for any $\lambda\in\Omega$, that is $\mathcal{I}_{1}\sim_{s}\mathcal{I}_{2}$.
\end{proof}

\begin{proposition}\label{p2}
Let $\textbf{T}=(T_{1},\cdots T_{m})\in B_{n}^{m}(\Omega)\cap\mathcal{L}(\mathcal{H}_{1})^{m}$ and $\textbf{S}=(S_{1},\cdots,S_{m})\in B_{n}^{m}(\Omega)\cap\mathcal{L}(\mathcal{H}_{2})^{m}.$ Then $\textbf{T}\sim_{s}\textbf{S}$ if and only if $\mathcal{I}_{\textbf{T}}\sim_{s}\mathcal{I}_{\textbf{S}}.$
\end{proposition}
\begin{proof}
If $\textbf{T}\sim_{s}\textbf{S},$ there is an invertible operator $X\in\mathcal{L}(\mathcal{H}_{1},\mathcal{H}_{2})$ such that $X\textbf{T}=\textbf{S}X,$
and then $X\ker\mathscr{D}_{\textbf{T}-\lambda}\subseteq \ker\mathscr{D}_{\textbf{S}-\lambda}.$ Since $\dim\ker\mathscr{D}_{\textbf{T}-\lambda}=\dim\ker\mathscr{D}_{\textbf{S}-\lambda}=n,$ we have that
$X\ker\mathscr{D}_{\textbf{T}-\lambda}=\ker\mathscr{D}_{\textbf{S}-\lambda}$ for any $\lambda\in\Omega.$
Thus, $\langle Xx_{2},y\rangle=\langle x_{2},X^{*}y\rangle=0$ for any $y\in \ker\mathscr{D}_{\textbf{S}-\lambda}$ and $x=x_{1}\oplus x_{2}\in\mathcal{H}_{1},$
where $x_{1}\in \ker\mathscr{D}_{\textbf{T}-\lambda}$ and $x_{2}\in[\ker\mathscr{D}_{\textbf{T}-\lambda}]^{\perp}.$
It follows that
$$X\mathcal{I}_{\textbf{T}}(\lambda)x=Xx_{1}=Xx_{1}+\mathcal{I}_{\textbf{S}}(\lambda)Xx_{2}
=\mathcal{I}_{\textbf{S}}(\lambda)Xx_{1}+\mathcal{I}_{\textbf{S}}(\lambda)Xx_{2}=
\mathcal{I}_{\textbf{S}}(\lambda)Xx,\quad x\in\mathcal{H}_{1}.$$
So $X\mathcal{I}_{\textbf{T}}(\lambda)=\mathcal{I}_{\textbf{S}}(\lambda)X$ for any $\lambda\in\Omega$ and the invertible operator $X\in\mathcal{L}(\mathcal{H}_{1},\mathcal{H}_{2})$.

On the contrary, if there is an invertible operator $X$ such that $X\mathcal{I}_{\textbf{T}}(\lambda)=\mathcal{I}_{\textbf{S}}(\lambda)X$ for any $\lambda\in\Omega$,
From $\mathcal{I}_{\textbf{T}}(\lambda)$ and $\mathcal{I}_{\textbf{S}}(\lambda)$ are idempotent operators on $\mathcal{H}_{1}$ to $\ker\mathscr{D}_{\textbf{T}-\lambda}$ and $\mathcal{H}_{2}$ to $\ker\mathscr{D}_{\textbf{S}-\lambda}$, respectively.
we obtain that $X\ker\mathscr{D}_{\textbf{T}-\lambda}=\ker\mathscr{D}_{\textbf{S}-\lambda}$ for any $\lambda\in\Omega$, and then $X\textbf{T}=\textbf{S}X$.
\end{proof}

Proposition \ref{p2} above shows that the similarity of the Cowen-Douglas class is equivalent to the similarity of its corresponding extended holomorphic curve, and  now we give an example to understand it intuitively.

\begin{example}
Let $\textbf{T}=(T_{1},...,T_{m}),$ $\textbf{S}=(S_{1},...,S_{m})\in\mathcal{B}_{1}^{m}(\mathbb{B}_{m})$ be commuting tuples of unilateral weighted shifts on reproducing kernel Hilbert spaces $\mathcal{H}_{K}$ and $\mathcal{H}_{\widehat{K}}$ with reproducing kernel $K$ and $\widehat{K},$ respectively. If
$$K(\omega,\lambda)=\sum\limits_{I\in\mathbb{Z}_{+}^{m}}a_{I}^{2}\omega^{I}
\overline{\lambda}^{I}\quad\textit{and}\quad
\widehat{K}(\omega,\lambda)=\sum\limits_{I\in\mathbb{Z}_{+}^{m}}b_{I}^{2}
\omega^{I}\overline{\lambda}^{I},\quad \omega,\lambda\in\mathbb{B}_{m},$$
where $a_{I}, b_{I}$ are real numbers. Then  $\mathcal{I}_{\textbf{T}}\sim_{s}\mathcal{I}_{\textbf{S}}$ if and only if $m\leq\frac{b_{I}}{a_{I}}\leq M$ for any $I\in\mathbb{Z}_{+}^{m}$, where the constants $0< m,M<\infty.$
\end{example}
\begin{proof}
From Proposition \ref{p2}, we only need to prove that $\textbf{T}\sim_{s}\textbf{S}$ if and only if $m\leq\frac{b_{I}}{a_{I}}\leq M$ for any $I\in\mathbb{Z}_{+}^{m}$ and the constants $0< m,M<\infty.$
Letting
$$\alpha(\lambda):=K(\cdot,\overline{\lambda}),\quad
\beta(\lambda):=\widehat{K}(\cdot,\overline{\lambda}),\quad \text{and} \quad\gamma(\lambda):=H(\cdot,\overline{\lambda}),$$ where $H(\omega,\lambda)=\sum\limits_{I\in\mathbb{Z}_{+}^{m}}\frac{b_{I}^{4}}
{a_{I}^{2}}\omega^{I}\overline{\lambda}^{I}.$
So we set
$$\mathcal{I}_{\textbf{T}}(\lambda)=\alpha(\lambda)
\left(\gamma^{*}(\lambda)\alpha(\lambda)\right)^{-1}\gamma^{*}(\lambda)
\quad \text{and}
\quad \mathcal{I}_{\textbf{S}}(\lambda)=\beta(\lambda)
\left(\beta^{*}(\lambda)\beta(\lambda)\right)^{-1}\beta^{*}(\lambda),\quad\lambda\in\mathbb{B}_{m},$$
where $\gamma^{*}(\lambda)\alpha(\lambda)=\beta^{*}(\lambda)\beta(\lambda)$ is a polynomial.
If $\textbf{T}\sim_{s}\textbf{S}$, by Lemma \ref{lm3} and Proposition \ref{p2},
there is an invertible operator $X$ such that
$$XD^{I}\mathcal{I}_{\textbf{T}}(\lambda)\overline{D}^{J}\mathcal{I}_{\textbf{T}}(\lambda)=
D^{I}\mathcal{I}_{\textbf{S}}(\lambda)\overline{D}^{J}\mathcal{I}_{\textbf{S}}(\lambda)X,\quad \lambda\in\Omega, \, I,J\in\mathbb{Z}_{+}^{m}.$$
We obtain that $XD^{I}\alpha(\lambda)\overline{D}^{J}\gamma^{*}(\lambda)=
D^{I}\beta(\lambda)\overline{D}^{J}\beta^{*}(\lambda)X$ from Lemma \ref{le00}.
If we take $\lambda=0,$ then
$$D^{I}\alpha(0)\overline{D}^{J}\gamma^{*}(0)
=\frac{a_{I}b_{J}^{2}}{a_{J}}I!J!e_{I,J}\quad\textup{and}\quad
D^{I}\beta(0)\overline{D}^{J}\beta^{*}(0)=b_{I}b_{J}I!J!e_{I,J}.$$
where $e_{I,J}$ denote the infinite matrix which satisfies that $(I,J)$th entry equals to 1 and other entries are all 0.
Setting $((x_{I,J}))$ to be the matrix form of $X,$ without loss of generality, suppose that $x_{0,0}\neq 0,$ then
$((x_{I,J}))\frac{a_{I}b_{J}^{2}}{a_{J}}e_{I,J}=b_{I}b_{J}e_{I,J}((x_{I,J})).$
Take $J=I+e_{i}, 1\leq i\leq m,$ we have that $x_{I,I}=\frac{b_{I}a_{I+e_{i}}}{a_{I}b_{I+e_{i}}}x_{I+e_{i},I+e_{i}},$
and then
$$\begin{cases}x_{0,0}=\frac{b_{0}a_{e_{i}}}{a_{0}b_{e_{i}}}x_{e_{i},e_{i}} \\
x_{e_{i},e_{i}}=\frac{b_{e_{i}}a_{2e_{i}}}{a_{e_{i}}b_{2e_{i}}}x_{2e_{i},2e_{i}}\\
\cdots\\
x_{(n-1)e_{i},(n-1)e_{i}}=\frac{b_{(n-1)e_{i}}a_{ne_{i}}}{a_{(n-1)e_{i}}b_{ne_{i}}}
x_{ne_{i},ne_{i}}\\
\cdots\\
\end{cases}.$$
It follows that $x_{0,0}=\frac{b(0)a(ne_{i})}{a(0)b(ne_{i})}x_{ne_{i},ne_{i}}$ for any positive integer $n$.
 From the relationship between commuting tuple of unilateral weighted shifts
and reproducing kernel of Hilbert space, we have that
$$T_{i}\textbf{e}_{I}=T_{i}a_{I}z^{I}=a_{I}T_{i}(z^{I})=a_{I}z^{I+e_{i}}
=\frac{a_{I}}{a_{I+e_{i}}}a_{I+e_{i}}z^{I+e_{i}}
=\frac{a_{I}}{a_{I+e_{i}}}\textbf{e}_{I+e_{i}},$$
where $\left\{\textbf{e}_{I}=a_{I}z^{I}\right\}_{I\in\mathbb{Z}_{+}^{m}}$ is an orthogonal basis of the Hilbert space $\mathcal{H}_{K}$.
Thus, for any $I\in\mathbb{Z}_{+}^{m},$
\begin{equation}\label{eq11}
\frac{\prod\limits_{k=0}^{n-1}\frac{a_{I+ke_{i}}}{a_{I+(k+1)e_{i}}}}
{\prod\limits_{k=0}^{n-1}\frac{b_{I+ke_{i}}}{b_{I+(k+1)e_{i}}}}=
\frac{a_{I}b_{I+ne_{i}}}{b_{I}a_{I+ne_{i}}},\quad n\geq 1.
\end{equation}
By the Theorem 3.1 \cite{Ex3.14} and (\ref{eq11}), we know that $\textbf{T}\sim_{s}\textbf{S}$ if and only if $m\leq\frac{b_{I}}{a_{I}}\leq M$ for any $I\in\mathbb{Z}_{+}^{m}$ and the constants $0< m,M<\infty.$
\end{proof}

\section{\sf The flag structure of the class $\mathcal{FB}_{n}(\Omega)$}\label{section3}
In this section, $\Omega\subset\mathbb{C}$ be a domain, then the class $\mathcal{B}_{n}^{1}(\Omega)$, i.e., $\mathcal{B}_{n}(\Omega)$, is the Cowen-Douglas operator. M. J. Cowen and R. G. Douglas \cite{CD, CD3} show that curvature and its covariant derivatives are complete unitary invariants of the Cowen-Douglas class. But these curvatures and their covariant derivatives are generally very difficult to calculate except for the Hermitian holomorphic vector bundle of rank 1.
K. Ji, C. L. Jiang, D. K. Keshari, and G. Misra \cite{JJKM1, FBnJJKM} isolate $\mathcal{FB}_{n}(\Omega)$, a subclass of operators with flag structure in the class $\mathcal{B}_{n}(\Omega)$, and give that the curvature and the second fundamental forms are the complete set of unitary invariants of this subclass.
C. L. Jiang, K. Ji, and D. K. Keshari \cite{JJK2023} defined a subclass of $\mathcal{FB}_{n}(\Omega)$, $\mathcal{CFB}_{n}(\Omega)$, and proved that  $\mathcal{CFB}_{n}(\Omega)$ is norm dense in $\mathcal{B}_{n}(\Omega)$ and that
the curvature and the second fundamental form completely
characterize the similarity invariants for the class $\mathcal{CFB}_{n}(\Omega)$.

\begin{definition}\cite{FBnJJKM}
Let $\mathcal{FB}_{n}(\Omega)$ be the set of all bounded linear operators $T$ defined on some complex separable Hilbert space $\mathcal{H}=\mathcal{H}_{0}\oplus\cdots\oplus\mathcal{H}_{n-1},$ which are of the form
\begin{equation*}
T=
\begin{pmatrix}
T_{0}  &  S_{01} & S_{02} & \cdots & S_{0,n-1} \\
0      &  T_{1}  & S_{12} & \cdots & S_{1,n-1} \\
  0    &    0    & T_{2}  & \cdots & S_{2,n-1} \\
 \vdots&   \vdots& \vdots & \ddots & \vdots    \\
    0  &   0     &    0   & \cdots & T_{n-1}
\end{pmatrix},
\end{equation*}
where the operator $T_{i}\in\mathcal{B}_{1}(\Omega)\cap\mathcal{L}(\mathcal{H}_{i})$, $0\leq i\leq n-1,$ and $S_{i,i+1}:\mathcal{H}_{i}\rightarrow\mathcal{H}_{i+1},$ is assume to be a non-zero intertwining operator, namely, $T_{i}S_{i,i+1}=S_{i,i+1}T_{i+1},\,0\leq i\leq n-2.$
\end{definition}
All quasi-homogeneous operators are in the class $\mathcal{FB}_{n}(\Omega),$ (see Theorem 3.3 \cite{JJM2017}).  S. S. Ji and the second author  \cite{JX2022} show that the trace of the curvature of homogeneous operators in $\mathcal{FB}_{n}(\Omega)$ is equal to the sum of the curvature of the diagonal operators of the operator.
If the operators in $\mathcal{FB}_{n}(\Omega)$ are unitary equivalent, then the unitary operator has the following structure.

\begin{lemma}\cite{FBnJJKM}\label{le001}
Any two operators $T$ and $\widetilde{T}$ in $\mathcal{FB}_{n}(\Omega)$ are unitarily equivalent if and only if there exists unitary operator $U=\text{diag}\,\{U_{0},U_{1},\ldots,U_{n-1}\}$ such that $UT=\widetilde{T}U.$
\end{lemma}

\begin{lemma}\label{lm4.4}\cite{CD}
Operators $T$ and $\widetilde{T}$ in $\mathcal{B}_{1}(\Omega)$ are unitarity equivalent if and only if $\mathcal{K}_{T}(\lambda)=\mathcal{K}_{\widetilde{T}}(\lambda)$ for any $\lambda\in\Omega.$
\end{lemma}

\begin{definition}\cite{Jet}
Let $E$ be an Hermitian holomorphic line bundle over $\Omega$ and $\gamma$ be a holomorphic cross-sections of $E$. Suppose that $\gamma(\lambda),\gamma^{'}(\lambda),\gamma^{(2)}(\lambda),\cdots,\gamma^{(k)}(\lambda)$ are independent for each $\lambda$ in $\Omega$.
We say that $\mathcal{J}_{k}(E)$ is a $k$-jet bundle of rank $k+1$ obtained from $E$ and $\mathcal{J}_{k}(E)$ has an associated holomorphic frame $\mathcal{J}_{k}(\gamma)=\{\gamma,\gamma^{'},\gamma^{(2)},\cdots,\gamma^{(k)}\}$.
\end{definition}




Let  $T\in\mathcal{B}_{n}(\Omega)\cap\mathcal{L}(\mathcal{H}),$ and  ${\gamma_{0}, \gamma_{1}, \ldots, \gamma_{n-1}}$ be a holomorphic
frame for the Hermitian holomorphic vector bundle $E_{T}$. Define $\alpha=(\gamma_{0},\gamma_{1},\ldots,\gamma_{n-1})$ and $P_{T}(\lambda)=\alpha(\lambda)h^{-1}(\lambda)\alpha^{*}(\lambda),$ where $h(\lambda)=\alpha^{*}(\lambda)\alpha(\lambda)$ is the metric of $E_{T},$
we have that $P_{T}:\Omega\rightarrow\mathcal{P}(\mathcal{L}(\mathcal{H}))$ is an holomorphic curve and $P_{T}(\lambda)$ is the matrix form of the projection operator from $\mathcal{H}$ to $\ker(T-\lambda)$.

\begin{lemma}\label{prop22}
Let $T\in\mathcal{B}_{n}(\Omega)\cap\mathcal{L}(\mathcal{H}_{1}), S\in\mathcal{B}_{n}(\Omega)\cap\mathcal{L}(\mathcal{H}_{2}).$ Then
$T\sim_{u}S$
if and only if
$$UP_{T}(\lambda)=P_{S}(\lambda)U,\quad \lambda\in\Omega$$ for some unitary operator $U\in\mathcal{L}(\mathcal{H}_{1},\mathcal{H}_{2}).$
\end{lemma}
\begin{proof}
If $T\sim_{u}S,$
we obtain that $U\ker(T-\lambda)=\ker(S-\lambda)$ from $\dim\ker(T-\lambda)=\dim\ker(S-\lambda)=n$ for any $\lambda\in\Omega$ and some unitry operator $U\in\mathcal{L}(\mathcal{H}_{1},\mathcal{H}_{2})$.
For any $x=x_{1}\oplus x_{2}\in\mathcal{H}_{1},$ where $x_{1}\in\ker(T-\lambda)$ and $x_{2}\in[\ker(T-\lambda)]^{\perp},$ we have that
$UP_{T}(\lambda)x=Ux_{1}.$
And then for any $y\in\ker(S-\lambda),$ there is a $\widehat{x}\in\ker(T-\lambda)$ such that $y=U\widehat{x}.$
Thus, $\langle y,Ux_{2}\rangle=\langle U\widehat{x},Ux_{2}\rangle=\langle U^{*}U\widehat{x},x_{2}\rangle=\langle\widehat{x},x_{2}\rangle=0.$
It follows that  $Ux_{2}\in[\ker(S-\lambda)]^{\perp}.$ So
$P_{S}(\lambda)Ux=P_{S}(\lambda)\left(Ux_{1}\oplus
Ux_{2}\right)=Ux_{1}\oplus 0=UP_{T}(\lambda)x$ for any $x\in\mathcal{H}_{1}$,
and then $UP_{T}(\lambda)=P_{S}(\lambda)U.$

On the contrary, if there is an unitary operator $U$ such that $UP_{T}(\lambda)=P_{S}(\lambda)U$ for any $\lambda\in\Omega.$ This means that $U\ker(T-\lambda)\subseteq\ker(S-\lambda).$
Since $\dim\ker(T-\lambda)=\dim\ker(S-\lambda)=n$ and $U$ is an unitary operator, we know that
$U\ker(T-\lambda)=\ker(S-\lambda)$
and $UT=SU.$
\end{proof}


\begin{lemma}\label{lm22}
Let
\begin{math}
T=
\left(
\begin{smallmatrix}
T_{0}  &  S \\
0  &  T_{1} \\
\end{smallmatrix}
\right)
\in\mathcal{FB}_{2}(\Omega)\cap\mathcal{L}(\mathcal{H}).
\end{math}
Then $P_{T}(\lambda): \mathcal{H} \rightarrow \ker(T-\lambda)$ may be expressed in the following form:
\begin{equation*}
P_{T}(\lambda)=(I-\theta(\lambda))
\begin{pmatrix}
P_{T_{0}}(\lambda)  &  0 \\
0  &  P_{T_{1}}(\lambda) \\
\end{pmatrix}
+\theta(\lambda)
\begin{pmatrix}
F_{P_{T}}(\lambda)  &  S_{P_{T}}(\lambda) \\
S_{P_{T}}^{*}(\lambda)  &  0 \\
\end{pmatrix}
,
\end{equation*}
where
\begin{math}
\left(
\begin{smallmatrix}
F_{P_{T}}(\lambda)  &  S_{P_{T}}(\lambda) \\
S_{P_{T}}^{*}(\lambda)  &  0 \\
\end{smallmatrix}
\right)
\end{math}
: $\mathcal{J}_{\{t_{0}\}}^{(1)} \rightarrow \ker(T-\lambda)$ is isomorphic and satisfies
\begin{equation*}
\begin{pmatrix}
F_{P_{T}}(\lambda)  &  S_{P_{T}}(\lambda) \\
S_{P_{T}}^{*}(\lambda)  &  0 \\
\end{pmatrix}
\begin{pmatrix}
t_{0}(\lambda)\\
0 \\
\end{pmatrix}
=
\begin{pmatrix}
t_{0}(\lambda)\\
0 \\
\end{pmatrix}
,\quad
\begin{pmatrix}
F_{P_{T}}(\lambda)  &  S_{P_{T}}(\lambda) \\
S_{P_{T}}^{*}(\lambda)  &  0 \\
\end{pmatrix}
\begin{pmatrix}
t'_{0}(\lambda)\\
0 \\
\end{pmatrix}
=
\begin{pmatrix}
t'_{0}(\lambda)\\
t_{1}(\lambda)\\
\end{pmatrix},
\end{equation*}
$t_{1}$ is a non-zero holomorphic cross-section of $E_{T_{1}}$, $t_{0}=-St_{1},$ and $\theta\Vert t_{1}\Vert^{2}+(1+\theta)\mathcal{K}_{T_{0}}\Vert t_{0}\Vert^{2}=0.$
\end{lemma}

\begin{proof}
Letting $\gamma_{0}(\lambda)=t_{0}(\lambda)$ and $\gamma_{1}(\lambda)=t^{\prime}_{0}(\lambda)-t_{1}(\lambda),$
where $t_{1}$ is an non-zero cross-section of $E_{T_{1}}$ and $t_{0}=-St_{1}.$ Then $\{\gamma_{0}, \gamma_{1}\}$ is a holomorphic frame of $E_{T}$ and
\begin{equation}
  \begin{aligned}
&P_{T}(\lambda)\\
=&\begin{pmatrix}
\gamma_{0}(\lambda) ,& \gamma_{1}(\lambda) \\
\end{pmatrix}
\begin{pmatrix}
\langle\gamma_{0}(\lambda),\gamma_{0}(\lambda)\rangle & \langle\gamma_{1}(\lambda),\gamma_{0}(\lambda)\rangle\\
\langle\gamma_{0}(\lambda),\gamma_{1}(\lambda)\rangle & \langle\gamma_{1}(\lambda),\gamma_{1}(\lambda)\rangle\\
\end{pmatrix}
^{-1}
\begin{pmatrix}
\gamma_{0}^{*}(\lambda)\\
\gamma_{1}^{*}(\lambda)\\
\end{pmatrix}\\
=&\begin{pmatrix}
t_{0}(\lambda) & t^{\prime}_{0}(\lambda) \\
0 & t_{1}(\lambda) \\
\end{pmatrix}
\begin{pmatrix}
h_{0}(\lambda) & \frac{\partial}{\partial\lambda} h_{0}(\lambda) \\
\frac{\partial}{\partial\overline{\lambda}} h_{0}(\lambda) &
\,\, \frac{\partial^{2}}{\partial\lambda\partial\overline{\lambda}} h_{0}(\lambda)+h_{1}(\lambda) \\
\end{pmatrix}
^{-1}
\begin{pmatrix}
t_{0}^{*}(\lambda) & 0 \\
t_{0}'^{*}(\lambda) & t_{1}^{*}(\lambda) \\
\end{pmatrix}\\
=&\frac{h_{0}(\lambda)h_{1}(\lambda)}{h_{0}(\lambda)h_{1}(\lambda)
-\mathcal{K}_{T_{0}}(\lambda)h_{0}^{2}(\lambda)}
\begin{pmatrix}
t_{0}(\lambda)h_{0}^{-1}(\lambda)t_{0}^{*}(\lambda) & 0 \\
0 & t_{1}(\lambda)h_{1}^{-1}(\lambda)t_{1}^{*}(\lambda) \\
\end{pmatrix}-\frac{\mathcal{K}_{T_{0}}(\lambda)h_{0}^{2}(\lambda)}{h_{0}(\lambda)h_{1}(\lambda)
-\mathcal{K}_{T_{0}}(\lambda)h_{0}^{2}(\lambda)}
\\
&
\begin{pmatrix}
\frac{-t_{0}(\lambda)\frac{\partial^{2}}{\partial\lambda\partial\overline{\lambda}}
h_{0}(\lambda)t_{0}^{*}(\lambda)+t_{0}(\lambda)\frac{\partial }{\partial\lambda} h_{0}(\lambda)t_{0}'^{*}(\lambda)+t'_{0}(\lambda)\frac{\partial}{\partial\overline{\lambda}}
h_{0}(\lambda)t_{0}^{*}(\lambda)-t'_{0}(\lambda)h_{0}(\lambda)t_{0}'^{*}(\lambda)}
{\mathcal{K}_{T_{0}}(\lambda)h_{0}^{2}(\lambda)} & \,\,\frac{t_{0}(\lambda)
\frac{\partial}{\partial\lambda} h_{0}(\lambda)t_{1}^{*}(\lambda)-t'_{0}(\lambda)h_{0}(\lambda)t_{1}^{*}(\lambda)
}{\mathcal{K}_{T_{0}}(\lambda)h_{0}^{2}(\lambda)} \\
\frac{t_{1}(\lambda)\frac{\partial}{\partial\overline{\lambda}} h_{0}(\lambda)t_{0}^{*}(\lambda)-t_{1}(\lambda)h_{0}(\lambda)t_{0}'^{*}(\lambda)
}{\mathcal{K}_{T_{0}}(\lambda)h_{0}^{2}(\lambda)} &  0\\
\end{pmatrix}\\
:=&(1-\theta(\lambda))\begin{pmatrix}P_{T_{0}}(\lambda)  &  0 \\0  &  P_{T_{1}}(\lambda) \\ \end{pmatrix} +\theta(\lambda) \begin{pmatrix} F_{P_{T}}(\lambda)  &  S_{P_{T}}(\lambda) \\ S_{P_{T}}^{*}(\lambda)  &  0 \\ \end{pmatrix},
\nonumber
  \end{aligned}
\end{equation}
where $h_{0}(\lambda)=\Vert t_{0}(\lambda)\Vert^{2}$, $h_{1}(\lambda)=\Vert t_{1}(\lambda)\Vert^{2}$ and $\theta(\lambda)=-\frac{\mathcal{K}_{T_{0}}(\lambda)h_{0}^{2}(\lambda)}{h_{0}(\lambda)h_{1}(\lambda)
-\mathcal{K}_{T_{0}}(\lambda)h_{0}^{2}(\lambda)}$. We have that $\theta(\lambda)\Vert t_{1}(\lambda)\Vert^{2}+(1+\theta(\lambda))\mathcal{K}_{T_{0}}(\lambda)\Vert t_{0}(\lambda)\Vert^{2}=0$ for all $\lambda\in\Omega,$
\begin{equation*}
\begin{aligned}
&\begin{pmatrix}
F_{P_{T}}(\lambda)  &  S_{P_{T}}(\lambda) \\
S_{P_{T}}^{*}(\lambda)  &  0 \\
\end{pmatrix}
\begin{pmatrix}
t_{0}(\lambda)\\
0 \\
\end{pmatrix}\\
=&
\begin{pmatrix}
\frac{t_{0}(\lambda)\left(-\frac{\partial^{2}}
{\partial\lambda\partial\overline{\lambda}} h_{0}(\lambda)t_{0}^{*}(\lambda)
t_{0}(\lambda)+\frac{\partial }{\partial\lambda} h_{0}(\lambda)t_{0}'^{*}(\lambda)t_{0}(\lambda)\right)}{\mathcal{K}_{T_{0}}
(\lambda)h_{0}^{2}(\lambda)}+\frac {t'_{0}(\lambda)
\left(\frac{\partial }{\partial\overline{\lambda}} h_{0}(\lambda)t_{0}^{*}(\lambda)t_{0}(\lambda)
-h_{0}(\lambda)t_{0}'^{*}(\lambda)t_{0}(\lambda)\right)}
{\mathcal{K}_{T_{0}}(\lambda)h_{0}^{2}(\lambda)}  \\
\frac{t_{1}(\lambda)\left(\frac{\partial}{\partial\overline{\lambda}} h_{0}(\lambda)t_{0}^{*}(\lambda)t_{0}(\lambda)
-h_{0}(\lambda)t_{0}'^{*}(\lambda)t_{0}(\lambda)\right)}
{\mathcal{K}_{T_{0}}(\lambda)h_{0}^{2}(\lambda)}\\
\end{pmatrix}
\\
=&
\begin{pmatrix}
\frac{t_{0}(\lambda)\left(-\frac{\partial^{2}}{\partial\lambda\partial\overline{\lambda}} h_{0}(\lambda)h_{0}(\lambda)+\frac{\partial}{\partial\lambda} h_{0}(\lambda)\frac{\partial}{\partial\overline{\lambda} } h_{0}(\lambda)\right)}{\mathcal{K}_{T_{0}}(\lambda)h_{0}^{2}(\lambda)}+\frac {t'_{0}(\lambda)\left(\frac{\partial}{\partial\overline{\lambda}} h_{0}(\lambda)h_{0}(\lambda)-h_{0}(\lambda)\frac{\partial}{\partial\overline{\lambda}} h_{0}(\lambda)\right)}{\mathcal{K}_{T_{0}}(\lambda)h_{0}^{2}(\lambda)}  \\
\frac{t_{1}(\lambda)\left(\frac{\partial}{\partial\overline{\lambda}} h_{0}(\lambda)h_{0}(\lambda)-h_{0}(\lambda)\frac{\partial}{\partial\overline{\lambda}} h_{0}(\lambda)\right)}{\mathcal{K}_{T_{0}}(\lambda)h_{0}^{2}(\lambda)}\\
\end{pmatrix}
\\
=&
\begin{pmatrix}
t_{0}(\lambda)\\
0 \\
\end{pmatrix},
\end{aligned}
\end{equation*}
and
\begin{equation*}
\begin{aligned}
&\begin{pmatrix}
F_{P_{T}}(\lambda)  &  S_{P_{T}}(\lambda) \\
S_{P_{T}}^{*}(\lambda)  &  0 \\
\end{pmatrix}
\begin{pmatrix}
t'_{0}(\lambda) \\
0 \\
\end{pmatrix}\\
=&
\begin{pmatrix}
\frac{t_{0}(\lambda)\left(-\frac{\partial}{\partial\overline{\lambda}}  h_{0}(\lambda)t_{0}^{*}(\lambda)t'_{0}(\lambda)+\frac{\partial }{\partial\lambda} h_{0}(\lambda)t_{0}'^{*}(\lambda)t'_{0}(\lambda)\right)}{\mathcal{K}_{T_{0}}
(\lambda)h_{0}^{2}(\lambda)}+\frac {t'_{0}(\lambda)\left(\frac{\partial }{\partial\lambda} h_{0}(\lambda)t_{0}^{*}(\lambda)t'_{0}(\lambda)-
h_{0}(\lambda)t_{0}'^{*}(\lambda)t'_{0}(\lambda)\right)}
{\mathcal{K}_{T_{0}}(\lambda)h_{0}^{2}(\lambda)}  \\
\frac{t_{1}(\lambda)\left(\frac{\partial}{\partial\overline{\lambda}} h_{0}(\lambda)t_{0}^{*}(\lambda)t'_{0}(\lambda)-h_{0}(\lambda)
t_{0}'^{*}(\lambda)t'_{0}(\lambda)\right)}{\mathcal{K}_{T_{0}}(\lambda)h_{0}^{2}(\lambda)}\\
\end{pmatrix}
\\
=&
\begin{pmatrix}
\frac{t_{0}(\lambda)\left(-\frac{\partial^{2}}
{\partial\lambda\partial\overline{\lambda}} h_{0}(\lambda)\frac{\partial}{\partial\overline{\lambda}}h_{0}(\lambda)+\frac{\partial }{\partial\lambda} h_{0}(\lambda)\frac{\partial^{2}}
{\partial\lambda\partial\overline{\lambda}} h_{0}(\lambda)\right)}{\mathcal{K}_{T_{0}}h_{0}^{2}(\lambda)}+\frac {t'_{0}(\lambda)\left(\frac{\partial}{\partial\overline{\lambda}} h_{0}(\lambda)\frac{\partial }{\partial\lambda} h_{0}(\lambda)-h_{0}(\lambda)\frac{\partial^{2}}
{\partial\lambda\partial\overline{\lambda}}h_{0}(\lambda)\right)}{\mathcal{K}_{T_{0}}(\lambda)
h_{0}^{2}(\lambda)}  \\
\frac{t_{1}(\lambda)\left(\frac{\partial}{\partial\overline{\lambda}} h_{0}(\lambda)\frac{\partial }{\partial\lambda} h_{0}(\lambda)-h_{0}(\lambda)\frac{\partial^{2}}{\partial\overline{\lambda}^{2}} h_{0}(\lambda)\right)}{\mathcal{K}_{T_{0}}(\lambda)h_{0}^{2}(\lambda)}\\
\end{pmatrix}
\\
=&
\begin{pmatrix}
t'_{0}(\lambda)\\
t_{1}(\lambda)\\
\end{pmatrix}.
\end{aligned}
\end{equation*}
\end{proof}

From the above lemma, we obtain that for all $\lambda\in\Omega$, $F_{P_{T}}(\lambda)t_{0}(\lambda)=t_{0}(\lambda),$ and $F_{P_{T}}(\lambda)t'_{0}(\lambda)=t'_{0}(\lambda)$, so $F_{P_{T}}(\lambda)$ is an identity mapping from the holomorphic frame $\mathcal{J}_{1}(t_{0})=\{t_{0},t_{0}^{'}\}$
to $ \ker(T-\lambda)$, i.e.,
\begin{equation*}
P_{T}(\lambda)=(I-\theta(\lambda))
\begin{pmatrix}
P_{T_{0}}(\lambda)  &  0 \\
0  &  P_{T_{1}}(\lambda) \\
\end{pmatrix}
+\theta(\lambda)
\begin{pmatrix}
I_{P_{T}}(\lambda) &  S_{P_{T}}(\lambda) \\
S_{P_{T}}^{*}(\lambda)  &  0 \\
\end{pmatrix}
.
\end{equation*}


\begin{thm}
Suppose that
\begin{math}
T=
\left(
\begin{smallmatrix}
T_{0}  &  S \\
0  &  T_{1} \\
\end{smallmatrix}
\right)
\in\mathcal{FB}_{2}(\Omega)\cap\mathcal{L}(\mathcal{H})
\end{math}
and
\begin{math}
\widetilde{T}=
\left(
\begin{smallmatrix}
\widetilde{T}_{0}  &  \widetilde{S} \\
0  &  \widetilde{T}_{1} \\
\end{smallmatrix}
\right)
\in\mathcal{FB}_{2}(\Omega)\cap\mathcal{L}(\widetilde{\mathcal{H}}).
\end{math}
If $t_{1}$ and $\widetilde{t}_{1}$ are non-zero cross-section of the line bundle $E_{T_{1}}$ and $E_{\widetilde{T}_{1}}$, respectively, and $t_{0}=-St_{1}$, $\widetilde{t}_{0}=-\widetilde{S}\widetilde{t}_{1}.$
 Then following statements are equivalent:

\begin{itemize}
  \item [(1)] The orthogonal projection $P_{T}(\lambda)$ and $P_{\widetilde{T}}(\lambda)$ are unitarily equivalent.
  \item [(2)] The operators $T$ and $S$ are unitarily equivalent.
  \item [(3)] $\mathcal{K}_{T_{0}}=\mathcal{K}_{\widetilde{T}_{0}}$ (or $\mathcal{K}_{T_{1}}=\mathcal{K}_{\widetilde{T}_{1}}$) and $\frac{\Vert St_{1}\Vert^{2}}{\Vert t_{1}\Vert^{2}}=\frac{\Vert\widetilde{S}\widetilde{t}_{1}\Vert^{2}}{\Vert\widetilde{
      t}_{1}\Vert^{2}}$.
  \item [(4)] There exists a unitary $U$ such that
  \begin{math}
\left(
\begin{smallmatrix}
F_{P_{T}}  &  S_{P_{T}} \\
S_{P_{T}}^{*}  &  0 \\
\end{smallmatrix}
\right)=U^{\ast}
\left(
\begin{smallmatrix}
F_{P_{\widetilde{T}}}  &  S_{P_{\widetilde{T}}} \\
S_{P_{\widetilde{T}}}^{*}  &  0 \\
\end{smallmatrix}
\right)U,
\end{math}
where $U=\text{diag}\{U_{00},U_{11}\}$.
\end{itemize}
\end{thm}

\begin{proof}
$(1) \Leftrightarrow (2)$. Take $n=2$ in Lemma \ref{prop22}.

$(2) \Leftrightarrow (3)$. See Theorem 2.10 of \cite{FBnJJKM}.

$(3)\Rightarrow(4)$.
If $\mathcal{K}_{T_{0}}=\mathcal{K}_{\widetilde{T}_{0}}$ and $\frac{\Vert St_{1}\Vert^{2}}{\Vert t_{1}\Vert^{2}}=\frac{\Vert \widetilde{S}\widetilde{t}_{1}\Vert^{2}}{\Vert\widetilde{t}_{1}\Vert^{2}},$ we have that
$\frac{\Vert t_{0}(\lambda)\Vert}{\Vert t_{1}(\lambda)\Vert}=
\frac{\Vert \widetilde{t}_{0}(\lambda)\Vert}{\Vert \widetilde{t}_{1}(\lambda)\Vert}$ for all $\lambda\in\Omega.$ So $$\frac{h_{0}(\lambda)h_{1}(\lambda)}{h_{0}(\lambda)h_{1}(\lambda)-h_{0}^{2}(\lambda)
\mathcal{K}_{T_{0}}(\lambda)}=
\frac{\widetilde{h}_{0}(\lambda)\widetilde{h}_{1}(\lambda)}{\widetilde{h}_{0}(\lambda)
\widetilde{h}_{1}(\lambda)-\widetilde{h}_{0}^{2}(\lambda)\mathcal{K}_{\widetilde{T}_{0}}(\lambda)
}$$
and
$$\frac{h_{0}^{2}(\lambda)\mathcal{K}_{T_{0}}(\lambda)}
{h_{0}(\lambda)h_{1}(\lambda)-h_{0}^{2}(\lambda)\mathcal{K}_{T_{0}}(\lambda)}
=\frac{\widetilde{h}_{0}^{2}(\lambda)\mathcal{K}_{\widetilde{T}_{0}}(\lambda)}
{\widetilde{h}_{0}(\lambda)\widetilde{h}_{1}(\lambda)-
\widetilde{h}_{0}^{2}(\lambda)\mathcal{K}_{\widetilde{T}_{0}}(\lambda)},\quad \lambda\in\Omega,$$
where $h_{i}(\lambda)=\Vert t_{i}(\lambda)\Vert^{2}$ and $\widetilde{h}_{i}(\lambda)=\Vert \widetilde{t}_{i}(\lambda)\Vert^{2}$ for $i=0,1.$
By Lemma \ref{le001} and Lemma \ref{prop22}, there is an unitary operator \begin{math}
U=
\left(
\begin{smallmatrix}
U_{00}  &  0 \\
0  &  U_{11} \\
\end{smallmatrix}
\right)
\end{math} such that
$UP_{T}(\lambda)=P_{\widetilde{T}}(\lambda)U$, that is,
\begin{equation*}
\begin{aligned}
&\begin{pmatrix}
U_{00}  &  0 \\
0  &  U_{11} \\
\end{pmatrix}
\left[\frac{h_{0}(\lambda)h_{1}(\lambda)}{h_{0}(\lambda)h_{1}(\lambda)-h_{0}^{2}(\lambda)
\mathcal{K}_{T_{0}}(\lambda)}
\begin{pmatrix}
P_{T_{0}}(\lambda)  &  0 \\
0  &  P_{T_{1}}(\lambda)\\
\end{pmatrix}
-\frac{h_{0}^{2}(\lambda)\mathcal{K}_{T_{0}}(\lambda)}{h_{0}(\lambda)h_{1}(\lambda)-
h_{0}^{2}(\lambda)\mathcal{K}_{T_{0}}(\lambda)}
\begin{pmatrix}
F_{P_{T}}(\lambda)  &  S_{P_{T}} (\lambda)\\
S_{P_{T}}^{*}(\lambda)  &  0 \\
\end{pmatrix}
\right]\\
=&\left[\frac{\widetilde{h}_{0}(\lambda)\widetilde{h}_{1}(\lambda)}
{\widetilde{h}_{0}(\lambda)\widetilde{h}_{1}(\lambda)-\widetilde{h}_{0}^{2}(\lambda)
\mathcal{K}_{\widetilde{T}_{0}}(\lambda)}
\begin{pmatrix}
P_{\widetilde{T}_{0}}(\lambda)  &  0 \\
0  &  P_{\widetilde{T}_{1}}(\lambda)\\
\end{pmatrix}
-\frac{\widetilde{h}_{0}^{2}(\lambda)\mathcal{K}_{\widetilde{T}_{0}}(\lambda)}
{\widetilde{h}_{0}(\lambda)\widetilde{h}_{1}(\lambda)-\widetilde{h}_{0}^{2}(\lambda)
\mathcal{K}_{\widetilde{T}_{0}}(\lambda)}
\begin{pmatrix}
F_{P_{\widetilde{T}}}(\lambda)  &  S_{P_{\widetilde{T}}} (\lambda)\\
S_{P_{\widetilde{T}}}^{*}(\lambda)  &  0 \\
\end{pmatrix}
\right]
\begin{pmatrix}
U_{00}  &  0 \\
0  &  U_{11} \\
\end{pmatrix}.
\end{aligned}
\end{equation*}
It follows that
\begin{equation*}
\begin{pmatrix}
U_{00}  &  0 \\
0  &  U_{11} \\
\end{pmatrix}
\begin{pmatrix}
P_{T_{0}}(\lambda)  &  0 \\
0  &  P_{T_{1}}(\lambda)\\
\end{pmatrix}
=
\begin{pmatrix}
P_{\widetilde{T}_{0}}(\lambda)  &  0 \\
0  &  P_{\widetilde{T}_{1}}(\lambda)\\
\end{pmatrix}
\begin{pmatrix}
U_{00}  &  0 \\
0  &  U_{11} \\
\end{pmatrix},\quad \lambda\in\Omega.
\end{equation*}
Thus,
\begin{equation*}
\begin{pmatrix}
U_{00}  &  0 \\
0  &  U_{11} \\
\end{pmatrix}
\begin{pmatrix}
F_{P_{T}} (\lambda) &  S_{P_{T}} (\lambda)\\
S_{P_{T}}^{*} (\lambda) &  0 \\
\end{pmatrix}
=
\begin{pmatrix}
F_{P_{\widetilde{T}}} (\lambda) &  S_{P_{\widetilde{T}}}(\lambda) \\
S_{P_{\widetilde{T}}}^{*}(\lambda)  &  0 \\
\end{pmatrix}
\begin{pmatrix}
U_{00}  &  0 \\
0  &  U_{11} \\
\end{pmatrix},\quad \lambda\in\Omega.
\end{equation*}

$(4) \Rightarrow(3)$. If
there is a unitary operator $U=\text{diag}\{U_{00},U_{11}\}$ such that
 \begin{math}
U\left(
\begin{smallmatrix}
F_{P_{T}}  &  S_{P_{T}} \\
S_{P_{T}}^{*}  &  0 \\
\end{smallmatrix}
\right)
\end{math}
=
\begin{math}
\left(
\begin{smallmatrix}
F_{P_{\widetilde{T}}}  &  S_{P_{\widetilde{T}}} \\
S_{P_{\widetilde{T}}}^{*}  &  0 \\
\end{smallmatrix}
\right)U,
\end{math}
we have that
 \begin{equation}\label{e4.1}
\begin{pmatrix}
U_{00}  &  0 \\
0  &  U_{11} \\
\end{pmatrix}
\begin{pmatrix}
F_{P_{T}}(\lambda)  &  S_{P_{T}}(\lambda) \\
S_{P_{T}}^{*}(\lambda)  &  0 \\
\end{pmatrix}
\begin{pmatrix}
t_{0}(\lambda)\\
0\\
\end{pmatrix}
=
\begin{pmatrix}
F_{P_{\widetilde{T}}}(\lambda)  &  S_{P_{\widetilde{T}}}(\lambda) \\
S_{P_{\widetilde{T}}}^{*} (\lambda) &  0 \\
\end{pmatrix}
\begin{pmatrix}
U_{00}  &  0 \\
0  &  U_{11} \\
\end{pmatrix}
\begin{pmatrix}
t_{0}(\lambda)\\
0\\
\end{pmatrix},\quad \lambda\in\Omega.
\end{equation}
Without losing generality, assume that $U\begin{pmatrix}
t_{0}(\lambda)\\
0\\
\end{pmatrix}=\phi(\lambda)\begin{pmatrix}
\widetilde{t}_{0}(\lambda)\\
0\\
\end{pmatrix}+\psi(\lambda)\begin{pmatrix}
\widetilde{t}_{0}'(\lambda)\\
\widetilde{t}_{1}(\lambda)\\
\end{pmatrix}$ for $\phi,\psi\in \text{Hol} (\Omega).$
By Lemma \ref{lm22}, we obtain that
\begin{equation*}
\begin{aligned}
\begin{pmatrix}
U_{00}t_{0}(\lambda)\\
0\\
\end{pmatrix}
=&
\begin{pmatrix}
U_{00}  &  0 \\
0  &  U_{11} \\
\end{pmatrix}
\begin{pmatrix}
F_{P_{T}}(\lambda)  &  S_{P_{T}}(\lambda) \\
S_{P_{T}}^{*}(\lambda)  &  0 \\
\end{pmatrix}
\begin{pmatrix}
t_{0}(\lambda)\\
0\\
\end{pmatrix}
\\
=&
\begin{pmatrix}
F_{P_{\widetilde{T}}}(\lambda)  &  S_{P_{\widetilde{T}}}(\lambda) \\
S_{P_{\widetilde{T}}}^{*} (\lambda) &  0 \\
\end{pmatrix}
\left[\phi(\lambda)
\begin{pmatrix}
\widetilde t_{0}(\lambda)\\
0
\end{pmatrix}
+\psi(\lambda)
\begin{pmatrix}
\widetilde t_{0}'(\lambda)\\
\widetilde{t}_{1}(\lambda)\\
\end{pmatrix}
\right]
\\
=&
\begin{pmatrix}
\phi(\lambda)\widetilde{t}_{0}(\lambda)\\
0\\
\end{pmatrix}
+
\begin{pmatrix}
\psi(\lambda)\widetilde{t}_{0}'(\lambda)+\psi(\lambda)S_{P_{\widetilde{T}}}(\lambda)\widetilde{t}_{1}(\lambda)\\
\psi(\lambda)\widetilde{t}_{1}(\lambda)\\
\end{pmatrix},
\end{aligned}
\end{equation*}
then $\psi=0.$ Thus,
\begin{equation}\label{e4.1-1}
U_{00}t_{0}(\lambda)=\phi(\lambda)\widetilde{t}_{0}(\lambda),\quad \lambda\in\Omega.
\end{equation}
From (\ref{e4.1}), (\ref{e4.1-1}) and Lemma \ref{lm22}, we have that
\begin{equation*}
\begin{aligned}
 U
\begin{pmatrix}
t_{0}'(\lambda) \\
t_{1}(\lambda) \\
\end{pmatrix}
=&
U
\begin{pmatrix}
F_{P_{T}}(\lambda)  &  S_{P_{T}}(\lambda) \\
S_{P_{T}}^{*}(\lambda)  &  0 \\
\end{pmatrix}
\begin{pmatrix}
t_{0}^{\prime}(\lambda)\\
0\\
\end{pmatrix}
\\
=&
\begin{pmatrix}
F_{P_{\widetilde{T}}}(\lambda)  &  S_{P_{\widetilde{T}}}(\lambda) \\
S_{P_{\widetilde{T}}}^{*} (\lambda) &  0 \\
\end{pmatrix}
\begin{pmatrix}
\phi'(\lambda)\widetilde{t}_{0}(\lambda)+\phi(\lambda)\widetilde{t}_{0}'(\lambda)\\
0
\end{pmatrix}
\\
=&
\begin{pmatrix}
\phi'(\lambda)\widetilde{t}_{0}(\lambda)\\
0\\
\end{pmatrix}
+
\begin{pmatrix}
\phi(\lambda)\widetilde{t}_{0}'(\lambda) \\
\phi(\lambda)\widetilde{t}_{1}(\lambda)\\
\end{pmatrix}.
\end{aligned}
\end{equation*}
Thus,
\begin{equation}\label{e4.1-2}
U_{11}t_{1}(\lambda)=\phi(\lambda)\widetilde{t}_{1}(\lambda),\quad \lambda\in\Omega.
\end{equation}
From (\ref{e4.1-1}) and (\ref{e4.1-2}), we have that
$$\frac{\Vert St_{1}(\lambda)\Vert^{2}}{\Vert t_{1}(\lambda)\Vert^{2}}=\frac{\Vert t_{0}(\lambda)\Vert^{2}}{\Vert t_{1}(\lambda)\Vert^{2}}=\frac{\Vert U_{00}t_{0}(\lambda)\Vert^{2}}{\Vert U_{11}t_{1}(\lambda)\Vert^{2}}=\frac{\Vert \phi(\lambda)\widetilde{t}_{0}(\lambda)\Vert^{2}}{\Vert \phi(\lambda)\widetilde{t}_{1}(\lambda)\Vert^{2}}
=\frac{\vert\phi(\lambda)\vert^{2}\Vert\widetilde{t}_{0}(\lambda)\Vert^{2}}
{\vert\phi(\lambda)\vert^{2}\Vert\widetilde{t}_{1}(\lambda)\Vert^{2}}
=\frac{\Vert\widetilde{t}_{0}(\lambda)\Vert^{2}}{\Vert\widetilde{t}_{1}(\lambda)\Vert^{2}}
=\frac{\Vert \widetilde{S}\widetilde{t}_{1}(\lambda)\Vert^{2}}{\Vert\widetilde{t}_{1}(\lambda)\Vert^{2}},$$
$U_{00}\ker(T_{0}-\lambda)=\ker(\widetilde{T}_{0}-\lambda)$ and $U_{11}\ker(T_{1}-\lambda)=\ker(\widetilde{T}_{1}-\lambda)$ for any $\lambda\in\Omega$ and unitary operators $U_{00}$ and $U_{11}$. Since $T_{i},\widetilde{T}_{i}\in\mathcal{B}_{1}(\Omega),\,i=0,1,$  then $T_{i}\sim_{u}\widetilde{T}_{i},$ we obtain that $\mathcal{K}_{T_{0}}=\mathcal{K}_{\widetilde{T}_{0}}$ and $\mathcal{K}_{T_{1}}=\mathcal{K}_{\widetilde{T}_{1}}$ from Lemma \ref{lm4.4}.
The proof is complete.
\end{proof}

\begin{corollary} Let $T,\widetilde{T}\in\mathcal{FB}_{n}(\Omega).$ If $P_{T}\sim_{u}P_{\widetilde{T}},$ then $P_{T(i)}\sim_{u}P_{\widetilde{T}(i)}$, where \begin{equation*} T(i)= \begin{pmatrix} T_{0}  &  S_{0,1}  & \cdots & S_{0,i-1} \\   0   &  T_{1}   & \cdots & S_{1,i-1} \\ \vdots &  \vdots  & \ddots & \vdots    \\     0  &      0   &     \cdots  & T_{i-1} \end{pmatrix} \quad\text{and}\quad \widetilde{T}(i)= \begin{pmatrix} \widetilde{T}_{0}  &  \widetilde{S}_{0,1}  & \cdots & \widetilde{S}_{0,i-1} \\    0   &  \widetilde{T}_{1}   & \cdots & \widetilde{S}_{1,i-1} \\ \vdots &  \vdots  & \ddots & \vdots    \\     0  &      0   &    \cdots & \widetilde{T}_{i-1} \end{pmatrix}, \quad 1\leq i\leq n. \end{equation*} \end{corollary}

A detailed study on the Cowen-Douglas operators was provided by C. L. Jiang and Z. Y. Wang with the upper triangular representation theorem, (see Theorem 1.49 in \cite{JW1998}). If the operator $T\in\mathcal{B}_{n}(\Omega)\cap\mathcal{L}(\mathcal{H})$, then there exists an orthogonal decomposition $\mathcal{H}=\mathcal{H}_{0}\oplus\mathcal{H}_{1}\oplus\cdots
\mathcal{H}_{n-1}$ and operators $T_{0}, T_{1}, \ldots, T_{n-1}$
in $\mathcal{B}_{1}(\Omega)$ such that
\begin{equation*}
T=
\begin{pmatrix}
T_{0}  &  S_{01} &\cdots & S_{0,n-1} \\
0      &  T_{1}  & \cdots & S_{1,n-1} \\
 \vdots&   \vdots& \ddots & \vdots    \\
    0  &   0     &  \cdots & T_{n-1}
\end{pmatrix}.
\end{equation*}
Let $\{\gamma_{0},\gamma_{1},\cdots,\gamma_{n-1}\}$ be a holomorphic frame of $E_{T}$ and $t_{i}:\Omega\rightarrow Gr(1,\mathcal{H}_{i})$
 be a holomorphic frame of $E_{T_{i}}, 0\leq i\leq n-1.$
Then the relationship between $\{\gamma_{i}\}_{i=0}^{n-1}$
and $\{t_{i}\}_{i=0}^{n-1}$ can be found in \cite{JJKM1,JJM2017},
indicating that
$\gamma_{j}=\psi_{0,j}(t_{0})+\cdots+\psi_{i,j}(t_{i})+\cdots+
\psi_{j-1,j}(t_{j-1})+t_{j}$ for $0\leq j\leq n-1,$
where $\psi_{i,j}$ are certain holomorphic bundle maps.
Meanwhile, a more accurate description of the relationship between $\{\gamma_{i}\}_{i=0}^{n-1}$
and $\{t_{i}\}_{i=0}^{n-1}$ of the operator class $\mathcal{OFB}_{n}(\Omega)$ is provided in \cite{JJKM1}.
\begin{definition}
An operator $T\in\mathcal{L}(\mathcal{H})$ is said to be in the class $\mathcal{OFB}_{n}(\Omega)$, if $T\in\mathcal{FB}_{n}(\Omega)$ and $T$ has the following form on $\mathcal{H}=\mathcal{H}_{0}\oplus\mathcal{H}_{1}\oplus\cdots\oplus\mathcal{H}_{n-1},$
\begin{equation*}
T=
\begin{pmatrix}
T_{0}  &  S_{01} & 0       & \cdots &  0  \\
  0    &  T_{1}  & S_{11}  & \cdots & 0   \\
\vdots &   \vdots    & \ddots  & \ddots & \vdots   \\
0      &    0    &   \cdots  & T_{n-2} &S_{n-2,n-1}\\
  0    &      0  &   \cdots  &   0    & T_{n-1}\\
\end{pmatrix}
\end{equation*}
\end{definition}

Since $T$ and $\widetilde{T}$ in $\mathcal{B}_{n}(\Omega)$ are unitary equivalent if and only if $E_{T}$ is equivalent to $E_{\widetilde{T}}$, and their classical curvature and its covariant derivatives is also a completely unitary invariant. From section \ref{sec2.1}, it can be seen that the calculation of the curvature depends on the selection of holomorphic frames. Next, we will provide the structure of the invertible holomorphic matrix between different holomorphic frames of operators in $\mathcal{OFB}_{n}(\Omega)$.

We provide a special structure for the holomorphic frame of operator $T$ in $\mathcal{OFB}_{n}(\Omega)$. Let $t_{n-1}$
be an holomorphic frame of $E_{T_{n-1}}$ and $t_{i-1}=-S_{i-1,i}t_{i}$ for $1\leq i\leq n-1$. Suppose that
\begin{equation}\label{eq14}
\gamma_{k}=t_{0}^{(k)}+{k \choose 1}t_{1}^{(k-1)}+\cdots+i!{k \choose i}t_{i}^{(k-i)}+
\cdots+(k-1)!{k \choose k-1}t_{k-1}'+k!t_{k},\quad 0\leq k\leq n-1,
\end{equation}
 then
$\{\gamma_{0},\gamma_{1},\cdots,\gamma_{n-1}\}$ is a holomorphic frame of $E_{T}$.

\begin{lemma}\label{lm3-1}
Let $
T=\begin{tiny}
\begin{pmatrix}
T_{0}  &  S_{01}        & \cdots &  0  \\
  0    &  T_{1}   & \cdots & 0   \\
\vdots &   \vdots    & \ddots   & \vdots   \\
  0    &      0  &   \cdots     & T_{n-1}\\
\end{pmatrix}\end{tiny}\in\mathcal{OFB}_{n}(\Omega)$, and let $t_{n-1}$ and $\widetilde{t}_{n-1}$
be two different holomorphic frames of $E_{T_{n-1}}$.
Suppose that $\{\gamma_{0},\gamma_{1},\cdots,\gamma_{n-1}\}$ and $\{\widetilde{\gamma}_{0},\widetilde{\gamma}_{1},\cdots,\widetilde{\gamma}_{n-1}\}$
are two different holomorphic frames of $E_{T}$ that satisfy (\ref{eq14}).
Then there is an invertible holomorphic matrix $\phi=(\phi_{i,j})_{i,j=0}^{n-1}$
such that
$$
(\widetilde{\gamma}_{0},\widetilde{\gamma}_{1},\cdots,\widetilde{\gamma}_{n-1})=
(\gamma_{0},\gamma_{1},\cdots,\gamma_{n-1})\phi,
$$
where $\phi_{i,j}=0$ for $i>j$ and $\phi_{i,j}={j \choose i}\phi_{0,0}^{(j-i)}$ for $i\leq j.$
\end{lemma}
\begin{proof}
Since $\{\gamma_{0},\gamma_{1},\cdots,\gamma_{n-1}\}$ and $\{\widetilde{\gamma}_{0},\widetilde{\gamma}_{1},\cdots,\widetilde{\gamma}_{n-1}\}$
are two different holomorphic frames of $E_{T}$, there is an invertible holomorphic matrix $\phi=(\phi_{i,j})_{i,j=0}^{n-1}$
such that
$(\widetilde{\gamma}_{0},\widetilde{\gamma}_{1},\cdots,\widetilde{\gamma}_{n-1})=
(\gamma_{0},\gamma_{1},\cdots,\gamma_{n-1})\phi.$
It follows that
\begin{equation}\label{eq13}
\widetilde{\gamma}_{l}=\phi_{0l}\gamma_{0}+\phi_{1l}\gamma_{1}+\cdots+
\phi_{n-1,l}\gamma_{n-1},\quad 0\leq l\leq n-1.
\end{equation}
Take $l=0$ in (\ref{eq13}), and obtain $\widetilde{t}_{0}=\phi_{0,0}t_{0}$ and $\phi_{i,0}=0$ for $1\leq i\leq n-1$ from (\ref{eq14}).
Similarly, taking $l=k, 1\leq k\leq n-1,$ in (\ref{eq13}), we obtain that $\widetilde{t}_{0}^{(k)}=\phi_{0,k}t_{0}+\phi_{1,k}t_{0}^{\prime}+\cdots+
\phi_{k,k}t_{0}^{(k)}$ and $\phi_{i,k}=0$ for $k+1\leq i\leq n-1$.
Therefore, from $\widetilde{t}_{0}=\phi_{0,0}t_{0}$, we have that
$\widetilde{t}_{0}^{(k)}=\sum\limits_{i=0}^{k}{k \choose i}\phi_{0,0}^{(k-i)}t_{i}$, and then $\phi_{i,k}={k \choose i}\phi_{0,0}^{(k-i)}$ for $0\leq i\leq k.$
\end{proof}

\begin{remark}
Let $T,\widetilde{T}\in\mathcal{OFB}_{n+1}(\Omega)$ and $T\sim_{u}\widetilde{T},$
there is a unitary operator $U$ such that $UT=\widetilde{T}U.$
Suppose that $\{\gamma_{0},\gamma_{1},\cdots,\gamma_{n}\}$ and $\{\widetilde{\gamma}_{0},\widetilde{\gamma}_{1},\cdots,\widetilde{\gamma}_{n}\}$ are holomorphic frames satisfying (\ref{eq14}) for $E_{T}$ and $E_{\widetilde{T}},$ respectively.
If $(U\gamma_{0},\cdots,U\gamma_{n-1})
=(\widetilde{\gamma}_{0},\cdots,\widetilde{\gamma}_{n-1})$,
then $U(t_{i})=\phi\widetilde{t}_{i}, 0\leq i\leq n-1,$ for some  non-zero holomorphic function $\phi$.
\end{remark}
Let $M_{z}^{*}$ is Hardy shift operator and $T\in\mathcal{B}_{1}^{1}(\Omega)$. In infinite dimensions, by \cite{CD}, we know that the $M_{z}^{*}\sim_{u}T$ if and only if the classical curvature $\mathcal{K}_{M_{z}^{*}}=\mathcal{K}_{T}$ if and only if there is a holomorphic function $\phi(w)$ such that $\frac{1}{(1-\vert w\vert^{2})^{2}}=\vert\phi(w)\vert^{2}H_{T}(w),$ where $\frac{1}{(1-\vert w\vert^{2})^{2}}$ and  $H_{T}(w)$ are the Gram matrix of $M_{z}^{*}$ and $T$, respectively.

For the finite dimensional case.
Let $P(w)=P_{\ker(M_{z}^{*}-w)}$, $Q(w)=P_{\ker(T-w)}\in M_{n+1}(\mathbb{C})$ be two holomorphic curves, and $ran\,P(w)=\bigvee\{1,w,\cdots,w^{n}\}$, $ran\,Q(w)=\bigvee\{b_{1}(w),b_{2}(w),\cdots,b_{n}(w)\}$, where $w^{i},b_{i}(w),0\leq i\leq n$ are holomorphic functions. Then the necessary and sufficient conditions for $P\sim_{u}Q$ are given below.

Notice that $P\sim_{u}Q$ if and only if there is a holomorphic function $\phi(w)$ such that Gram matrix $H_{P}=\vert\phi\vert^{2}H_{Q}$, that is
\begin{align}\label{e3.4}
\sum\limits_{i=1}^{n}\vert w^{i}\vert^{2}=\vert\phi(w)\vert^{2}\sum\limits_{i=1}^{n}\vert b_{i}(w)\vert^{2}=\sum\limits_{i=1}^{n}\vert\widetilde{b_{i}}(w)\vert^{2},
\end{align}
where
$\widetilde{b_{i}}(w)=\phi(w)b_{i}(w)$.
The power series expansion for $\widetilde{b_{i}}(w)$, we have that $\widetilde{b_{i}}(w)=\sum\limits_{j=0}^{\infty}\beta_{ij}w^{j},\,0\leq i\leq n$, then $\vert\widetilde{b_{i}}(w)\vert^{2}=\widetilde{b_{i}}(w)\overline{\widetilde{b_{i}}(w)}=\sum\limits_{j,k=0}^{\infty}\beta_{ij}w^{j}\overline{\beta_{ik}w^{k}}$.
By (\ref{e3.4}), we have that $\sum\limits_{i=0}^{n}w^{i}\overline{w}^{i}=\sum\limits_{i=0}^{n}(\sum\limits_{j,k=0}^{\infty}\beta_{ij}w^{j}\overline{\beta_{ik}w^{k}})=\sum\limits_{j,k=0}^{\infty}(\sum\limits_{i=0}^{n}\beta_{ij}\overline{\beta_{ik}})w^{j}\overline{w}^{k}$,
then
$\sum\limits_{i=0}^{n}\beta_{ij}\overline{\beta_{ik}}
=\begin{cases}
1,\quad j=k=0,1,\cdots,n\\
0,\quad j\neq k,\textup{ or } j=k>n.
\end{cases}$
That is $\widetilde{b_{i}}(w)=\sum\limits_{j=0}^{n}\beta_{ij}w^{j},\,0\leq i\leq n$, $\sum\limits_{i=0}^{n}\vert\beta_{ij}\vert^{2}=1$ and $\sum\limits_{i=0}^{n}\beta_{ij}\overline{\beta_{ik}}=0,\,j\neq k$.
It follows that $b_{i}(w)=\frac{1}{\phi(w)}\beta_{ij}(w)^{j}$, $(\beta_{0,j},\cdots,\beta_{n,j})$ in the unit circle of $\mathbb{C}^{n+1}$ and $(\beta_{0,j},\cdots,\beta_{n,j})\perp(\beta_{0,k},\cdots,\beta_{n,k})$, where $1\leq i,j,k\leq n$ and $j\neq k$.
On the contrary can also be set up.

Let $\mathcal{J}_{E_{0}}^{(i)}$, $0\leq i\leq n-1$, be $n$-flag induced by the jet bundle of $E_{0}$, and $T,\widetilde{T}\in\mathcal{FB}_{n}(\Omega)$. Let $t_{0}$, $\widetilde{t}_{0}$ are non-zero cross-section of the line bundle $E_{0}$, $\widetilde{E}_{0}$, respectively, and $-t_{i}=-S_{i,i+1}t_{i+1}$, $-t_{i}^{(j-i-1)}=s_{i,j}t_{j}$. Then $E_{i}=\vee{\{t_{0},t'_{0}+t_{1},\cdots,t_{0}^{(i)}+t_{1}^{(i-1)}+\cdots+t_{i}\}}.$
Let $J=(J_0,J_1,\cdots,J_{n-1}),$ and $J_{i}:\mathcal{J}_{E_{0}}^{(i)}\rightarrow E_{i},0\leq i\leq n-1$, is isomorphic and satisfies
\begin{equation*}
J_{i}
\begin{pmatrix}
t_{0}^{(j)}\\
0 \\
\vdots\\
0\\
\end{pmatrix}
=
\begin{pmatrix}
t_{0}^{(j)}\\
t_{1}^{(j-1)}\\
\vdots\\
t_{j}\\
\end{pmatrix}
0\leq j\leq i\leq n-1.
\end{equation*}
Let $i:\mathcal{J}_{E_{0}}^{(j)}\rightarrow\mathcal{J}_{E_{0}}^{(j+1)}$ and $i:E_{(j)}\rightarrow E_{(j+1)}$, $0\leq j\leq n-2$, be the embedded mappings.
Then we have that
Let there is a unitary $U_0,U_i,0\leq i\leq n-1$, such that $U_0:\mathcal{J}_{E_{0}}^{(i)}\rightarrow\mathcal{J}_{\widetilde{E}_{0}}^{(i)}$ $U_i:E_i\rightarrow\widetilde{E}_{i}$ and $U_0t_0=\widetilde{t}_0$. Then we have the following graph that is commutative for any $0\leq i\leq n-1$, if and only if $E_T\stackrel{u}{\sim}E_{\widetilde{T}}$.

\begin{proof}
Since $U_{i}J_{i}=\widetilde{J}_{i}U_{0}$, then $U_{i}J_{i}(t_{0})=\widetilde{J}_{i}U_{0}(t_{0})$, that is $U_{i}t_{0}=\widetilde{J}_{i}\widetilde{t}_{0}=\widetilde{t}_{0}.$ Thus, $U_{i}t_{0}=\widetilde{t}_{0}$.
Without loss of generality, we suppose that $U_{i}t_{j-1}=\widetilde{t}_{j-1}$, then
$U_{i}J_{i}(t_{0}^{(j)})=\widetilde{J}_{i}U_{0}(t_{0}^{(j)}),$ that is
$$U_{i}(t_{0}^{(j)}+t_{1}^{(j-1)}+\cdots+t_{j})=\widetilde{J}_{i}\widetilde{t}_{0}^{(j)}=\widetilde{t}_{0}^{(j)}+\widetilde{t}_{1}^{(j-1)}+\cdots+\widetilde{t}_{j}.$$
It follows that $U_{i}t_{j}=\widetilde{t}_{j}, 0\leq j\leq i$.
That is $U_{i}=U^{0}\oplus U^{1}\oplus\cdots\oplus U^{i}.$
Thus $E_{T}\sim_{u}E_{\widetilde{T}}.$
On the contrary can also be set up.
\end{proof}

For any $T=((T_{i,j}))\in FB_n(\Omega)$,  $\mathcal{J}_{E_{0}}^{(i)}$ denote the Jet bundle induced by $E_0=E_{T_0}$ and $E_i=E_{T|_{\bigoplus\limits_{i=0}^{i}\mathcal{H}_i}}$.
Then we have two n-flag related to the holomorphic bundles as the following.
\begin{equation*}
 \begin{aligned}
&(1)\,\, E_0(\lambda)=\mathcal{J}_{E_{0}}^{(0)}(\lambda)\subset \mathcal{J}_{E_{0}}^{(1)}(\lambda) \subset \cdots \subset \mathcal{J}_{E_{0}}^{(i)}(\lambda)\cdots\subset \mathcal{J}_{E_{0}}^{(n-1)}(\lambda) \\
&(2)\,\,E_0(\lambda)=E_{T_0}(\lambda)\subset E_{T|_{\bigoplus\limits_{i=0}^{1}\mathcal{H}_i}}(\lambda) \subset \cdots \subset E_{T|_{\bigoplus\limits_{i=0}^{i}\mathcal{H}_i}}(\lambda)\cdots\subset E_{T|_{\bigoplus\limits_{i=0}^{n-1}\mathcal{H}_i}}(\lambda)\\
 \end{aligned}
\end{equation*}
Inspired by the theorem above, we can define the bundle map tuple $J=(J_0,J_1,\cdots,J_{n-1})$  between these two $n$-flag:
$$J_i:\mathcal{J}_{E_{0}}^{(i)}\rightarrow E_i$$ such that $$J_i(w)(t^{(k)}_0(\lambda))=t^{(k)}_0(\lambda)+t^{(k-1)}_1(\lambda)+\cdots+t_{k}(\lambda), k\leq i.$$

Then the following diagram is commutative:
\begin{equation*}
\xymatrix{
\mathcal{J}_{E_{0}}^{(0)}   \ar[d]_i \ar[r]^{J_0} &E_0\ar[d]^i\\
\mathcal{J}_{E_{0}}^{(1)}   \ar[d]_i \ar[r]^{J_1} &E_1\ar[d]_i\\
\vdots \ar[d]_i   &\vdots   \ar[d]_i\\
\mathcal{J}_{E_{0}}^{(n-1)} \ar[r]^{J_{n-1}}      &E_{n-1}}
\end{equation*}

\end{document}